\documentclass[10pt]{article}
\usepackage{graphicx} 
\usepackage[utf8]{inputenc}

\usepackage{amsmath, mathtools, amsfonts, mathrsfs, hyperref}
\usepackage{amsthm, amssymb}
\usepackage{bbm}
\usepackage{tikz-cd}
\usepackage{tikz}
\usetikzlibrary{shapes.geometric, arrows}
\usetikzlibrary{positioning}

\tikzstyle{process} = [rectangle, minimum width=3cm, minimum height=1cm, text centered, text width=3cm, draw=black, fill=green!30]
\tikzstyle{decision} = [diamond, minimum width=3cm, minimum height=1cm, text centered, text width = 3cm, draw=black, fill=orange!30]
\tikzstyle{arrow} = [thick,->,>=stealth]

\usepackage{soul}
\usepackage{comment}
\usepackage[labelfont=bf]{caption}
\usepackage{xcolor}
\usepackage{authblk}
\usepackage{stmaryrd}

\theoremstyle{definition}
\newtheorem*{Question}{Question}
\newtheorem*{Answer}{Answer}

\theoremstyle{plain}
\newtheorem{theorem}{Theorem}[section]
\newtheorem{lemma}[theorem]{Lemma}

\newtheorem{Proposition}[theorem]{Proposition}
\newtheorem{remark}[theorem]{Remark}
\newtheorem{Corollary}[theorem]{Corollary}

\newcommand{\bq}{\mathbb{Q}}
\newcommand{\qp}{\mathbb{Q}_p}
\newcommand{\zp}{\mathbb{Z}_p}

\newcommand{\norm}[1]{\vert #1 \vert}

\newcommand{\br}[1]{\overline{#1}}

\newcommand{\co}{\mathcal{O}}

\newcommand{\sL}{\mathcal{L}}
\newcommand{\brqp}{\br{\mathbb{Q}}_p}
\newcommand{\brFp}{\br{\mathbb{F}}_p}

\newcommand{\rmss}{{\mathrm{ss}}}

\newcommand{\bZ}{\mathbb{Z}}
\newcommand{\bQ}{\mathbb{Q}}
\newcommand{\Zp}{\bZ_p}
\newcommand{\Qp}{\bQ_p}
\newcommand{\bP}{\mathbb{P}}
\newcommand{\cL}{\mathcal{L}}
\newcommand{\IZind}{\mathrm{ind}_{IZ}^G\>}
\newcommand{\KZind}{\mathrm{ind}_{KZ}^G\>}
\newcommand{\IZKZind}{\mathrm{ind}_{IZ}^{KZ}\>}
\newcommand{\Oe}{\mathcal{O}_E}
\newcommand{\Fq}{\mathbb{F}_q}
\newcommand{\Fp}{\mathbb{F}_p}
\newcommand{\id}{\mathrm{id}}
\newcommand{\GL}{\mathrm{GL}}

\newcommand{\im}{\mathrm{Im}\>}
\newcommand{\Ker}{\mathrm{Ker}\>}
\newcommand{\SymE}[1]{\underline{\mathrm{Sym}}^{#1}E^2}
\newcommand{\SymO}[1]{\underline{\mathrm{Sym}}^{#1}\Oe^2}
\newcommand{\SymF}[1]{\mathrm{Sym}^{#1}\Fq^2}
\newcommand{\St}{\mathrm{St}}
\newcommand{\matUp}{\beta\begin{pmatrix}
1 & \lambda \\
0 & 1
\end{pmatrix}w}
\newcommand{\tWp}{\tilde{W}_p}
\newcommand{\tUp}{\tilde{U}_p}
\newcommand{\Bind}{\mathrm{ind}_B^G \>}
\newcommand{\C}[1]{\mathscr{C}^{#1}}
\newcommand{\Bnorm}[2]{\vert\vert #1 \vert\vert_{#2}}
\newcommand{\tB}{\tilde{B}}
\newcommand{\logL}{\log_{\cL}}
\newcommand{\sm}{\mathrm{smooth}}
\newcommand{\lattice}[1]{\tilde{\Theta}(#1)}
\newcommand{\latticeL}[1]{\tilde{\Theta}(#1, \cL)}
\newcommand{\ind}{\mathrm{ind}\>}
\newcommand{\sC}{\mathscr{C}}

\title{The semi-stable Local Langlands Correspondence}

\author{Eknath Ghate}
\affil{School of Mathematics, Tata Institute of Fundamental Research \\ Homi Bhabha Road, Mumbai - 400005, India.}

\begin{document}

\maketitle

\begin{abstract}
  We start with  background that goes into an Iwahori-theoretic reformulation of the mod $p$ Local Langlands Correspondence (\S 2).
  We then explain some classical $p$-adic functional analytic results (\S3) that go into defining the $p$-adic Banach space
  (\S 4) attached to a two-dimensional semi-stable representation $V_{k,\sL}$ of the Galois group of $\qp$ of weight $k$ and
  $\sL$-invariant $\sL$ under the $p$-adic Local Langlands correspondence. 
  We then sketch how to compute the reduction of a lattice in this Banach space, which along with the Iwahori mod $p$ LLC,
  allows one to  
  completely determine the mod $p$ reduction of $V_{k,\sL}$
  for  all weights $3 \leq k \leq p+1$ and all $\sL$ for $p \geq 5$ (\S 5). These notes are a summary of our joint
  work with Anand Chitrao \cite{CG23}.
  Emphasis is placed on motivation and background rather than completeness.
\end{abstract}  
  
\section{Introduction}

Let $p$ be a prime. It has been over 10 years now since I started working on mathematics connected with the $p$-adic Local
Langlands correspondence for $p$ a prime. The theory was initiated by Breuil about 20 years ago, and caused a mini-revolution
in the field of number theory. I remember him speaking about his vision at the ICM in Hyderabad in 2010 \cite{Breuil ICM}. 
A few years later Colmez made the
next quantum leap forward by making Breuil's correspondences functorial \cite{Col10b}. I was fortunate to have a ringside view of these
developments especially since some of them were made during the years 2007-2010 when we ran a joint Indo-French CEFIPRA project.
Meanwhile many other prominent mathematicians such as Berger, Dospinescu, Pa\u{s}k\={u}n{a}s, to name just a few,
contributed deep results along the way.

I entered the subject for the following reason. About 13 years ago, I was trying to show that certain mod $p$ Galois representations
attached to modular forms had large image. Indeed, I was trying to generalize Serre's famous conjecture 
that the global mod $p$ Galois representation attached to a non-CM rational elliptic curve has full image for all primes $p$ larger than
an absolute constant to the setting of modular forms. It turns out that the na\"ive generalization of this statement is necessarily false
but Pierre Parent and I could state a variant of this conjecture and make some mild progress on it for weight $2$ forms
\cite{GP12}.

One way to tackle this problem is to show that the corresponding restricted local mod $p$ Galois representation has large image.
Nothing like showing a group is large by showing that it has a large subgroup. This approach works to some extent but not entirely, not least because it turns out that the mod $p$ reductions of local Galois
representations attached to modular forms have not yet been written down in all cases. It was very disconcerting to me that there was such a glaring
gap in the subject. 
I decided to devote a good chunk of my future research time in making some
headway with this question.

It turns out  the $p$-adic Local Langlands program is ideally suited to studying the mod $p$ images of local
modular Galois representations.
Indeed, Breuil invented his theory to tackle exactly this problem. In the beginning he restricted to the so called crystalline case
where $p$ does not divide the level of the modular form. Let us from now on always restrict to odd primes $p$ (though in
some results in this Introduction we assume without warning that $p \geq 5$). Breuil showed that one could compute the local mod $p$ reductions for all
forms of weights $k \leq 2p+1$ and positive slope \cite{Bre03}, \cite{Bre03b}. Here the slope of a form is the $p$-adic valuation of its $p$-th Fourier
coefficient, where the valuation is normalized so that the valuation of $p$ is 1.
(The case of slope $0$ and all weights is classical and is due to Deligne.)
This generalized the work of his advisor Fontaine, although the details were worked out in \cite{Edi92},
who had earlier computed the reductions for weights $k \leq p+1$ and positive slope.

One obvious restriction
in these theorems is that the weight is bounded above, but there is no restriction on the slope. In an orthogonal direction, in a
short but influential paper, Buzzard and Gee \cite{BG09} used the $p$-adic Local Langlands correspondence to compute the mod $p$
reductions of crystalline Galois representations of slopes in the small range
$(0,1)$, but for all weights. Some
difficulties encountered at slope $1/2$ were only cleared up in a second paper \cite{BG13}. 
(The case of $p=2$ for slopes in this range is being treated in a forthcoming thesis of Arathy Venugopal.)
For historical completeness, let us 
mention that earlier and using a different method, Berger, Li and Zhu \cite{BLZ04}
had treated the case of slopes which are large compared to the weight, namely slopes which are larger
than $\lfloor \frac{k-2}{p-1} \rfloor$ (an interesting variant of this results was recently proved by Bergdall-Levin \cite{BL22}
who treated the
case of slope larger than $\lfloor \frac{k-1}{p} \rfloor$). 

This is where I entered the problem. In a series of papers with my coauthors (students, postdocs, colleagues), I first extended the result of
Buzzard-Gee to all slopes in $(0,2)$. This does not seem like much, but this marginal gain of a unit interval's worth of slopes
was to consume me for the better part of the first half of the
following decade. At first, the answers we got for the reduction seemed unpredictable, almost as if one was entering a fractal. There
were no general guidelines (other than a folklore conjecture of Breuil, Buzzard and Emerton which said that for fractional
slopes and even weights the reduction should be irreducible). Murphy's law (if things can go wrong, they will)
seemed to rule the roost - if a particular exception to a general rule for the shape of the reduction was in principle possible,
then it always wound up occurring.

Some initial headway for the case of slopes in $(1,2)$ was made with Abhik Ganguli for bounded weights \cite{GG15}, and then in a very nice paper with
Shalini Bhattacharya for all weights \cite{BG15} but again we were only able to partially treat the case of slope $3/2$. 
The missing case of slope $1$ was then
treated with Shalini Bhattacharya and Sandra Rozensztajn \cite{BGR18}. The complete picture for
slope $3/2$ was finally provided only a few years ago with Vivek Rai, though the paper \cite{GR19} appeared just this year.
(Beyond this range, I also wish to mention forthcoming work of Sudipta Majumder for slope $2$, some partial
results of Nagell-Pande and Arsovski for slopes in $(2,3)$, and a 
forthcoming project \cite{BGR25} with Shalini Bhattacharya and Ravitheja Vangala which aims to treat all fractional slopes 
in the range $(0,p)$ building on the foundations laid in \cite{GV22}.)
All these papers use the functoriality of
the $p$-adic Local Langlands Correspondence with respect to reduction (established by Berger if one is willing to work up to semi-simplification - as we mostly
were - and in general by Colmez), to reduce the question of studying the reductions of local crystalline Galois representations
to studying the reduction of the standard lattice in a certain $p$-adic Banach space. The slope 1 paper \cite{BGR18} also computes the reductions
of several other lattices, and in particular establishes criteria to distinguish between {\it peu} and {\it tr\`es ramifi\'ee}
cases.

Part of the problem encountered at the half-integral slopes $1/2$, $1$, $3/2, \ldots$ was that the reduction seemed to behave even more
erratically than usual at the so called exceptional weights $k$ (these are weights which are congruent to two more 
than twice the slope modulo $(p-1)$).
Based on the results in slope $1/2$ and $1$, and some cautionary computations of Rozensztajn, 
I eventually wound up making a conjecture
which I called the {\it zig-zag conjecture} which described the behaviour of the reduction for all positive half-integral
slopes less than or equal to $\frac{p-1}{2}$ and all sufficiently large exceptional weights. 
Roughly, the conjecture predicted that the reductions varied through an
alternating sequence of irreducible and reducible representations
depending on the (relative) sizes of two parameters. The statement appeared in a proceedings of an annual number theory conference
at RIMS in Kyoto, Japan \cite{Gha21}, where I was thrown to the wolves by my kind host Shinichi Kobayashi as the opening speaker (I thank him for this honor).
At the time it had become my  mission to settle the case of slope $3/2$ if only to prove that the conjecture
had some merit. However, over the years there was an uncomfortable truth that
began to emerge. The paper \cite{BG13} for slope $1/2$ was about 10 pages long, the one \cite{BGR18} for slope $1$ was about 50 pages long (though not all of it dealt with zig-zag), and the (unabridged arXiv version of the)
one \cite{GR19} for slope $3/2$ was just under 80 pages long (its entire focus was zig-zag at $3/2$). So clearly another approach would be required to prove the conjecture in general.
To my complete surprise this was to surface a few years later.

To explain this, let us return to Breuil's early foundational work. Apart from treating the crystalline case, Breuil also wrote some important papers in
the so called semi-stable case. This case occurs for modular forms for which $p$ exactly divides the level of the form (and does not divide the conductor
of the nebentypus character). In an important work with M\'ezard dating to more than 20 years ago, Breuil computed the mod $p$ reductions
of all semi-stable Galois representations of {\it even} weight $k \leq p-1$ using techniques from integral $p$-adic Hodge theory \cite{BM}. The case of {\it odd} weights
in this range was completed much later (about 6 years ago) in another {\it tour de force} by Guerberoff and Park \cite{GP} (and Lee and Park \cite{LP22}), though several constants
remained to be determined completely. For completeness, we also mention that an alternate algebro-geometric approach to computing the reduction involving the global sections of certain
bundles on the $p$-adic upper half-plane was developed by Breuil-M\'ezard in some cases \cite{BM10}, though this approach does indeed seem to require that the weight $k$ be even
since it involved $k/2$-powers of certain line bundles.

In any case, some time in 2022, I realized that all of these works in the parallel universe of semi-stable representations could be used to give a proof
of the zig-zag conjecture in the crystalline world, at least for most slopes and on the inertia group. This was perhaps one of the more important
observations that I have made in the past 10 years. Let us explain how it came about. A bit earlier than this, Anand Chitrao, Seidai Yasuda and I had been trying to use the above
mentioned works in the crystalline world to try and deduce results about $V_{k,\sL}$ in the semi-stable world, using a limiting argument in Colmez's blow-up space of
non-split rank $2$ trianguline $(\varphi,\Gamma)$-modules. We wrote a nice paper about this which appeared this year \cite{CGY21} and
which allowed one, for instance, to predict the exact shape of 
some of the above mentioned missing constants for small odd weights (e.g., for $k=5$ from the slope ${3}/{2}$ paper),
and, in general, to recover the work of
Breuil-M\'ezard and Guerberoff-Park on inertia {\it assuming} my zig-zag conjecture. The breakthrough came when I realized that one could reverse the
entire argument and instead deduce information about crystalline representations - in particular, a large portion of the zig-zag phenomenon -
from the literature in the semi-stable case. In fact, after this realization I could
immediately prove zig-zag up to slope $\frac{p-3}{2}$
(though in the first instance I could only prove it on the inertia group since, as already mentioned, the constants in the semi-stable world had not yet been
completely determined for odd weights). The missing cases of slope $\frac{p-2}{2}$ and $\frac{p-1}{2}$
would require extending the work of Guerberoff-Park \cite{GP} to the odd weight $k=p$ and the classical work of Breuil-M\'ezard \cite{BM} to the case of the even weight $k = p+1$.

The possible extension to these two weights was more than just a technicality. There was a theoretical obstruction. It turned out that the strongly
divisible modules occurring in integral $p$-adic Hodge theory were either not as well behaved ($k = p$,  see
\cite{Gao17}) or not even available (for $k \geq p+1$)
(although since then a theory of Breuil-Kisin modules has become available which works for all weights $k$). So an entirely new perspective was required.
Based on my experience with computing the reduction using the functoriality of the $p$-adic Local Langlands Correspondences in the crystalline world (a method initiated by Breuil and
Buzzard-Gee), I wondered whether Anand Chitrao and I might
be able to tackle the reduction problem in the semi-stable case in a similar manner. We were given to understand
that one
might have to wait for a very long time for this hope to be realized.

However, not ones to shy away from a challenge, Chitrao and I started work on this ambitious project. Initial gains were few and
far between. But, to make a long story short, we were eventually able to
compute the mod $p$ reductions of all semi-stable representations for weights $k \leq p+1$, including the cases of weight $k = p$
and $k = p+1$. We were also able to provide a complete and uniform treatment of all the constants involved.
The goal of this expository paper is to
explain this result and to expand on some of the mathematical background that goes into its proof.
But before I go further, let me record that this work
allowed one to complete the proof of my zig-zag conjecture (i.e., to extend the initial proof up to slope $\frac{p-1}{2}$,
and to determine all the constants that occur in the unramified characters on the decomposition group in the reduction, see \cite{Gha22}).

Already, in the early days, Breuil had written two important papers describing the Banach space attached to
a semi-stable representation $V_{k,\sL}$ of weight $k$ and $\sL$-invariant $\sL$ under the $p$-adic Local Langlands Correspondence. The first description, denoted by
$B(k,\sL)$ in \cite{Bre04},
involved some work of Schneider and Teitelbaum and used Morita duality
and seemed a bit abstract to us. But the second description, denoted by $\tB(k, \cL)$ in \cite{Bre10}, `only' used $p$-adic
functional analysis (a beautiful summary of which can be found in \cite{Col10a}, see $\S 3$) and this definition
seemed much more amenable to computation.
In this model, the Banach space $\tB(k, \cL)$ was nothing but the space of differentiable functions
on $\qp$ of order $r/2$ where $r = k-2$ (more precisely of type $\sC^{r/2}$, these notions are slightly different),
with a similar differentiability condition at $\infty$, modulo polynomial functions of degree at most $r$ and certain finite sums of  polynomial
times logarithmic ({poly$\cdot$log!}) functions (with the polynomial part having degree less than $r/2$). This description was something that you
could explain to a clever high school student learning calculus. Using it, we began to search for an integral structure (lattice)
on this Banach space.

It was not immediately clear how to proceed, but we soon realized that one could uniformize this
Banach space (much as in the crystalline world) by the compactly induced space $\mathrm{ind}_{IZ}^G \mathrm{Sym}^{k-2} E^2$ of
certain rational polynomial valued functions allowing us to define the (standard) lattice
in the Banach space as the (closure of the) image of the space $\mathrm{ind}_{IZ}^G \mathrm{Sym}^{k-2}{\mathcal O}_E^2$ of
integral polynomial valued functions under this uniformization. Here 
$I$ is the Iwahori subgroup of $G = \GL_2(\qp)$ and $Z$ is the center of $G$. It then `remained' to compute the reduction of this
lattice. In hindsight, this required several new creative computations. After much effort we were able to compute all
the Jordan-H\"older (JH) factors in the reduction
in terms of certain mod $p$ compactly induced spaces modulo the images of certain Iwahori-Hecke operators. Applying a
reformulation of Breuil's mod $p$ Local Langlands correspondence worked out by Chitrao, which is 
called the Iwahori mod $p$ Local Langlands correspondence in \cite{Chi23},
we could then return to the Galois side and compute the reductions of the semi-stable Galois representations $V_{k,\sL}$.
Our target to treat all weights in the range $3 \leq k \leq p+1$ was achieved in the paper \cite{CG23}, though we remark that
in principle our method can be used to treat  all weights $k \geq 3$ at least in principle (unlike the initial approach with strongly divisible modules).
Indeed, Anand and I have just begun to see whether our approach to computing the reduction of semi-stable representations
using $p$-adic Langlands can be used to recover very recent work of Bergdall, Levin, and Liu \cite{BLL22}
which uses Breuil-Kisin modules
to study the reduction of $V_{k,\sL}$ with $\sL$-invariant having very negative $p$-adic valuation.

Let us state the main result in the paper \cite{CG23} with Anand Chitrao. 
Let $p \geq 5$ be a prime and $E$ be a finite extension of $\qp$ containing $\sqrt{p}$.
We describe completely the semi-simplification of the reduction mod $p$ of
the irreducible two-dimensional semi-stable representation $V_{k, \cL}$ of $\mathrm{Gal}(\brqp/\qp)$
over $E$ with Hodge-Tate weights $(0,k-1)$ for $k \in [3, p + 1]$ and 
$\cL$-invariant $\cL \in E$. To do this
we need a little bit more notation. Let as usual $r = k-2$ so that $1 \leq r \leq p-1$.
    Let $v_-$ and $v_+$ be
    the largest and smallest integers, respectively, such that $v_- < r/2 < v_+$ for $r \geq 1$.
    For $n \geq 1$, let $H_n = \sum_{i = 1}^{n}\frac{1}{i}$ be the $n$-th partial harmonic sum.
    Set $H_0 = 0$ and $H_{\pm} = H_{v_{\pm}}$. Let $v_p$ be the $p$-adic valuation on $\brqp$
    normalized such that $v_p(p) = 1$.
    Let $$\nu = v_p(\sL - H_{-} - H_{+})$$
    be the $p$-adic valuation of $\sL$ shifted by the
    partial harmonic sums $H_{-}$ and $H_{+}$. 
    Everything hinges on the size of the parameter $\nu$. Let $\omega$ be the fundamental character of
    $\mathrm{Gal}(\br{{\mathbb Q}}_p/\qp)$ of
    level $1$. 
    Similarly, for an integer $c$ (with
    $p + 1 \nmid c$), let $\omega_2^c$ be an extension from the inertia subgroup $I_{\qp}$ to $\mathrm{Gal}(\brqp/\bq_{p^2})$
    of the $c$-th power of the
    fundamental character $\omega_2$ 
    of level $2$ chosen 
    so that the (irreducible) representation $\ind (\omega_2^c)$ obtained
    by inducing this character from $\mathrm{Gal}(\brqp/\bq_{p^2})$ to $\mathrm{Gal}(\brqp/\qp)$ has
    determinant $\omega^c$. Let $\mu_{\lambda}$ be the unramified character of $\mathrm{Gal}(\brqp/\qp)$ 
    sending geometric Frobenius to $\lambda \in \br{\mathbb{F}}_p^{*}$. 
    Our main theorem is the following result.

    \begin{theorem}[Chitrao-Ghate \cite{CG23}]\label{Main theorem in the second part of my thesis}
      For $k \in [3, p + 1]$ and for primes $p \geq 5$, the semi-simplification of the reduction mod $p$ of the
      semi-stable representation $V_{k, \cL}$ on $\mathrm{Gal}(\brqp/\qp)$ is given by an alternating
      sequence of irreducible and reducible representations:
    \[
    \br{V}_{k, \cL} \sim
    \begin{cases}
        \ind (\omega_2^{r+1 + (i-1)(p-1)}), & \text{ if $(i-1) - r/2 < \nu < i - r/2$} \\
        \mu_{\lambda_i}\omega^{r+1-i} \oplus \mu_{\lambda_i^{-1}}\omega^{i}, & \text{ if $\nu = i - r/2$},
    \end{cases}
\]
where 
  \begin{eqnarray*}
    \begin{cases} 
      1 \leq i \leq (r+1)/2 & \text{if $r$ is odd} \\
                          1 \leq i \leq (r+2)/2 & \text{if $r$ is even,}
    \end{cases}
 \end{eqnarray*}
  and the mod $p$ constants $\lambda_i$ are determined by the formulas:
    \begin{eqnarray*}
        \lambda_i & = & \br{(-1)^i \> i {r+1-i \choose i}\frac{(\cL - H_{-} - H_{+})}{p^{i-r/2}}}, \quad \text{ if } 1 \leq i < \dfrac{r + 1}{2} \\
        \lambda_{i} + \lambda_i^{-1} & = & \br{(-1)^i \> i{r + 1 - i \choose i}\frac{(\cL - H_{-} - H_{+})}{p^{i-r/2}}}, \quad \text{ if } i = \dfrac{r + 1}{2} \text{ and } r \text{ is odd}.
    \end{eqnarray*}
\end{theorem}
\noindent We remark that we have adopted 
  the following conventions in the statement of the theorem:
   \begin{itemize}
    \item The first interval $- r/2 < \nu < 1 - r/2$ is interpreted as $\nu < 1 - r/2$.
    \item If $r$ is odd, then the last case $\nu = 1/2$ should be interpreted as $\nu \geq 1/2$. If $r$ is even,
          then the interval $0 < \nu < 1$ should be interpreted as $\nu > 0$ and we drop the case $\nu = 1$.
    \end{itemize}
In any case, the theorem says that the reduction $\br{V}_{k,\sL}$ varies through an alternating sequence of 
irreducible and reducible mod $p$ 
representations as $\nu$ varies through finitely many marked points. 

\subsection{Idea of proof of Theorem~\ref{Main theorem in the second part of my thesis}}
A picture is worth a thousand words, and so we draw one to explain the proof. 
Let $B_{k,\sL} = \tilde{B}(k,\sL)$.
The following diagram commutes:

 \[
\begin{tikzcd}
  V_{k,\sL} \arrow[rrr, mapsto, "p-\text{adic LLC}"]  \arrow[dd, mapsto] & & & B_{k,\sL} \arrow[dd, mapsto] \\
   & &\\
  \br{V}_{k,\sL} \arrow[rrr, mapsto, "\text{Iwahori mod $p$ LLC}"]  &  & & \br{B}_{k,\sL}
\end{tikzcd}
\]
where the vertical maps are (semi-simplifications of) reductions of lattices in the corresponding spaces in the
top row.
One is trying to compute the left vertical map. But one computes instead the  right vertical
map since the bottom map is injective.

We now give a broad outline of the remaining text in terms of this picture. 
The paper is broken into four further sections.
The bottom map is explained in $\S 2$. The top right corner is explained in $\S 4$ based on 
the foundational material described in $\S 3$. The computation of the right vertical map is then
explained in $\S 5$.

\subsection{Notation}
\begin{itemize}
        \item $p \geq 5$ is a prime.
        \item $E$ is a finite extension of $\qp$ containing $\sqrt{p}$ and $\cL$.
              $\co_E$ is the ring of integers in $E$ with a uniformizer $\pi = \pi_E$ and residue field
$\Fq$. Note $\sqrt{p} \equiv 0 \mod \pi$. 
        \item $k$ denotes the weight of a semi-stable
representation and $r = k - 2$.
        \item $v_-$, $v_+$ are the largest, smallest integers,
respectively, such that $v_- < r/2 < v_+$ for $r \geq 1$.
        \item For $n \geq 1$, $H_n = \sum\limits_{i = 1}^{n}\frac{1}{i}$
and $H_0 = 0, H_{\pm} = H_{v_{\pm}}$.
        \item $v_p$ is the $p$-adic valuation normalized such that $v_p(p)
= 1$.
        \item $\nu = v_p(\sL - H_{-} - H_{+})$ is the valuation
of $\sL$ shifted by the partial
harmonic sums $H_{-}$ and $H_{+}$. 
        \item $I_{\qp}$ is the inertia subgroup of $\mathrm{Gal}(\brqp/\qp)$.
        \item $\omega$ is the fundamental character of $I_{\qp}$ of level
$1$;
        it has a canonical extension to $\mathrm{Gal}(\br{{\mathbb
Q}}_p/\qp)$.
        \item $\omega_2$ is the fundamental character of $I_{\qp}$ of
level $2$; for an integer $c$ with $p
+ 1 \nmid c$,
        choose an extension of $\omega_2^c$ to
$\mathrm{Gal}(\brqp/\bq_{p^2})$ so that the irreducible representation
        $\mathrm{ind}(\omega_2^c)$ obtained
        by inducing this extension from $\mathrm{Gal}(\brqp/\bq_{p^2})$ to
$\mathrm{Gal}(\brqp/\qp)$ has determinant $\omega^c$.
        \item $G$ is the group $\mathrm{GL}_2(\qp)$.
        \item $K$ is the maximal compact subgroup $\mathrm{GL}_2(\zp)$ of $G$.
        \item $I$ is the Iwahori subgroup of $G$ consisting of matrices in $K$ 
that are upper triangular $\!\!\!\!\mod p$.
        \item $B$ is the Borel subgroup of $G$ consisting of upper
triangular matrices.
        \item $\alpha = \begin{pmatrix}1 & 0 \\ 0 & p\end{pmatrix}$,
$\beta = \begin{pmatrix}0 & 1 \\ p & 0\end{pmatrix}$ and $w =
\begin{pmatrix}0 & 1 \\ 1 & 0\end{pmatrix}$. Note that $\beta =
\alpha w$.
        \item $I_n = \{[a_0] + [a_1]p + \cdots + [a_{n - 1}]p^{n - 1}
\> \vert \> a_i \in \Fp\}$ for $n \geq 1$, where $[\quad]$ denotes Teichm\"uller representative. $I_0 = \{0\}$.
        \item $V_{r} = \SymF{r}$ and 
        $\SymE{k - 2} := \norm{\det}^{\frac{k - 2}{2}} \otimes
\mathrm{Sym}^{k - 2}E^2$ for $k \geq 2$.
        \item $d^r$ for an integer $r$ denotes the character $IZ \to
\Fp^*$ which sends $\begin{pmatrix}a & b \\ c & d\end{pmatrix} \in
I$ to $d^r \!\!\! \mod p$ and which is trivial on the scalar
matrix $p$.
        \item For a representation $(\rho, V)$ of $IZ$ over $E$ or $\Fq$,
let $\IZind \rho$ be the vector space of functions $f : G \to V$
that are compactly supported modulo $IZ$ such that $f(hg) =
\rho(h)
        \cdot f(g)$, for all $h \in IZ$ and $g \in G$. The vector space
$\IZind \rho$ has a $G$ action:
        $g \cdot f(g') = f(g' g)$, for all $g, g' \in G$ and $f \in \IZind
\rho$. For $g \in G$ and $v \in V$, define the function $\llbracket g, v \rrbracket
\in \IZind \rho$ by
    \[
        \llbracket g, v\rrbracket(g') =
        \begin{cases}
            \rho(g'g) \cdot v, & \text{ if }g'g \in IZ \\
            0, & \text{ otherwise}.
        \end{cases}
    \]
        \item Let $(\Bind E)^{\sm}$ be the $E$-vector space of locally
constant functions from $G$ to $E$, with the action of $G$ given
by $g \cdot f(g') = f(g'g)$ for any $g, g' \in G$ and $f \in
(\Bind E)^{\sm}$.
        \item Let $\St$ be the Steinberg representation of $G$ over $E$,
i.e., $\St$ is the vector space of all locally constant functions
$H : \bP^1(\Qp) \to E$ modulo constant functions with the action
of $G$ given by $\left(\begin{pmatrix} a & b \\ c & d
\end{pmatrix}\cdot H\right)(z) = H\left(\dfrac{az + c}{bz +
d}\right)$.
        \item $[a] \in \{0, \ldots, p - 2\}$ denotes the class of $a \!\!
\mod p - 1$.
        \item $\delta_{a, b} = 1$ if $a = b$ and is $0$ otherwise.
\end{itemize}

    \section{Iwahori  mod $p$ Local Langlands Correspondence}
    \label{section Iwahori mod p  LLC}

  
    \subsection{Bruhat-Tits tree}
    \label{BT tree section}
  
    We start with the definition of the famous Bruhat-Tits tree.

    A lattice $L \subset \Qp^2$ is a
    free $\zp$-module of rank $2$. Two lattices $L$ and $L'$ are said to be homothetic if there is a non-zero
    scalar $z \in \qp^*$ such that $L' = z L$. The vertices of the tree are
    homothety classes $[L]$ of lattices $L \subset \Qp^2$.

    Two vertices represented by lattices $L$ and $L'$
    are joined by an edge 
    if $pL \subsetneq L' \subsetneq L$ (or equivalently $[L:L'] = p$). Note that this condition is symmetric in $[L]$ and $[L']$
    since $p(\frac{1}{p}L') \subsetneq L \subsetneq \frac{1}{p}L'$ and $\frac{1}{p}L'$ is homothetic to $L'$.
    The tree is a regular graph of valency $p+1$ since working mod $p$ there are $p+1$ subgroups of index $p$ in $\Zp^2/p\Zp^2 =
    ({\mathbb Z}/p{\mathbb Z})^2$.

    An oriented edge is an edge with a direction. The orientation can be specified by ordering the tuple
    $([L],[L'])$ so that $[L]$ is the origin or source of the edge and $[L']$ is the target or sink of the
    edge.
        
    \begin{lemma}
      \label{vertices-edges}
      Let $G = \mathrm{GL}_2(\qp) \supset K = \mathrm{GL}_2(\zp) \supset
           I = \{\left(\begin{smallmatrix} * & * \\ 0 & * \end{smallmatrix} \right) \mod p\} = $ the Iwahori subgroup.
           \begin{enumerate}
           \item The vertices of the tree are in  one-to-one correspondence with $G/KZ$.
           \item The oriented edges of the tree are in one-to-one correspondence with $G/IZ$.
           \end{enumerate}
    \end{lemma}            
    \begin{proof}    
        We claim that $G$ acts transitively on the vertices of the tree via
        $$ g \cdot [L] = [gL]$$
        for $g \in G$ and $L \subset \Qp^2$ a lattice.
        Note that  $gL$ is indeed a lattice in $\Qp^2$.
          Also if $L'$ is homothetic to $L$, then clearly $gL'$ is homothetic to $gL$. So the
          action is well defined. The transitivity of the action follows from the fact that given any two
          basis vectors $v_1 = (a,c)$ and $v_2 = (b,d)$ of $\Qp^2$ which span a particular lattice, the invertible matrix 
          $g = \left( \begin{smallmatrix} a & b \\ c & d \end{smallmatrix} \right)$ of $G$
          takes the standard lattice $L_0$ spanned by $e_1 = (1,0)$ and $e_2 = (0,1)$ to $v_1$ and $v_2$ respectively.
          Finally note that $KZ$ is the stabilizer of the standard lattice $L_0$. Indeed $KZ$ clearly stabilizes it.
          If $g \in G$ stabilizes $[L_0]$, then
          there exists a scalar $z \in \Qp^*$ such that $g \Zp^2 = z \Zp^2$. This means that $z^{-1} g \in K$. Thus $g \in KZ$.

          We similarly claim that $G$ acts transitively on the oriented edges of the tree via
          $$g \cdot ([L],[L']) = ([gL],[gL'])$$
          for $g$ in $G$ and $[L], [L']$ adjacent vertices in the tree.
          Note that $g \cdot ([zL], [z'L']) = ([gzL], [gz'L']) = ([zgL], [z'gL'])
          = ([gL], [gL'])$ for $g \in G$, $z,z' \in Z$. Also, if say $L' \subsetneq L$ has index $p$, then
          $gL' \subsetneq gL$ also has index $p$. So the action is well defined. For the transitivity, assume
          that $L' \subsetneq L$ has index $p$. By the `matrix game' from Prof. Murray Schacher's first year
          graduate algebra course at UCLA in 1991, there is a $\Zp$-basis $v_1$, $v_2$ of $L$ with respect to which
          $L'$ is spanned by $v_1$ and $pv_2$. Then the matrix $g$ in the previous paragraph takes
          the standard edge $([L_0], [\alpha L_0])$  where
          $\alpha = \left( \begin{smallmatrix} 1 & 0 \\ 0 & p \end{smallmatrix} \right)$ to $([L],[L'])$. 
          Finally note that $IZ$ is the stabilizer of $([L_0], [\alpha L_0])$. Indeed
          a small check shows that
          $$I = K \cap \alpha K \alpha^{-1}$$
          so clearly $I$ stabilizes both $L_0$ and $\alpha L_0$, so $IZ$ stabilizes the standard edge.
          Conversely if $g \in G$ stabilizes the standard edge, then $([gL_0], [g \alpha L_0]) = ([L_0], [\alpha L_0])$
          so by the previous paragraph, looking at the first coordinate we get $g \in KZ$, and looking at the second
          coordinate we get $\alpha^{-1} g \alpha \in KZ$, so that $g \in KZ \cap \alpha KZ \alpha ^{-1} = IZ$, as desired.
          Note that in the last equality, the containment `$\supset$' is obvious by the display just above;
          the reverse containment `$\subset$' follows by a short computation: if $g = kz = \alpha k'z' \alpha^{-1}$
          for $k, k' \in K$, $z,z'\in Z$, then taking determinants we see that $2v_p(z) = 2v_p(z')$, so we
          may cancel a like power of $p$ on both sides and assume $z$ and $z'$ are units, in which case
          $k \in K \cap \alpha K \alpha^{-1} = I$ and so $g = kz \in IZ$.          
    \end{proof}


\subsection{Hecke algebras}

We now recall the definitions of the spherical Hecke algebra and the Iwahori-Hecke algebra. The development proceeds in parallel.
Consider the following irreducible mod $p$ {\bf representations}:
\begin{itemize}
\item Let $V_r = \SymF{r} = \{\text{homogeneous polynomials } v(X,Y) \text{ of degree } r \text{ over } \Fq \}$
  for $0 \leq r \leq p-1$. Then $K$ acts on $V_r$ noting $K \twoheadrightarrow \mathrm{GL}_2(\Fq)$ which in turn 
  acts on $V_r$ via 
  $$\left( \begin{smallmatrix} a & b \\ c & d \end{smallmatrix} \right) \cdot v(X,Y) = v(aX+cY, bX+dY).$$
  Extend the action to $KZ$ by making $p \in Z$ act trivially.
  
\item Let $d^r$ be the one-dimensional space over $\Fq$ for $0 \leq r \leq p-1$.
  Then $I$ acts on $d^r$ noting $I \twoheadrightarrow B(\Fp)$ for
  the Borel subgroup $B(\Fp)$ of upper-triangular matrices which in turn acts on the space $d^r$ via the character $d^r$ given by 
  $$\left( \begin{smallmatrix} a & b \\ 0 & d \end{smallmatrix} \right) \cdot v = d^r v.$$
  Again, extend the action to $IZ$ by making $p \in Z$ act trivially.
\end{itemize}  

We may {\bf compactly induce} these representations to $G$ to obtain function spaces:
\begin{itemize}
\item Let $\mathrm{ind}_{KZ}^G V_r = \{ f : G \rightarrow V_r \> \vert \> f(kg) = k \cdot f(g), \forall k \in KZ, g \in G\}$.
  Here we only consider those $f$ which are compactly supported mod $KZ$, i.e., those $f$ supported on only finitely
  many right cosets of $KZ$ in $G$. For instance, for $g \in G$, $v \in V_r$, we may
  consider the elementary function $[g,v]$ supported on the coset $KZg^{-1}$ defined by
  $$ [g,v](g') = \begin{cases}
    g'g \cdot v(X,Y) & \text{if } g' \in KZg^{-1} \\
    0                & \text{otherwise.}
                 \end{cases} $$
    
\item Similarly, let $\mathrm{ind}_{IZ}^G d^r = \{ f : G \rightarrow \Fq(d^r) \> \vert \>  f(ig) =  d^r(i) f(g), \forall i \in IZ, g \in G\}$.
  Again we only consider those $f$ which are compactly supported mod $IZ$, i.e., supported on only finitely many
  right cosets of $IZ$ in $G$.  Again, for $g \in G$, $v \in d^r$, we may
  consider the elementary function $\llbracket g,v \rrbracket$ supported on the coset $IZg^{-1}$ defined by
  $$ \llbracket g,v \rrbracket (g') = \begin{cases}
    g'g \cdot v & \text{if } g' \in IZg^{-1} \\
    0                & \text{otherwise.}
                 \end{cases} $$
\end{itemize}

Consider now the structure of the corresponding {\bf Hecke algebras}. These algebras are by definition the algebra of
$G$-equivariant endomorphisms of the above compactly induced spaces:
\begin{itemize}
\item (Spherical Hecke algebra) $$\mathrm{End}_G(\mathrm{ind}_{KZ}^G V_r) = \Fq[T],$$ where the action of the
      Hecke operator $T$ on an elementary
      function is given by
      \begin{eqnarray}
        \label{def of T}
        T[g,v] = \sum_{\lambda \in \Fp} [g \left(\begin{smallmatrix} p & [\lambda] \\ 0 & 1 \end{smallmatrix} \right), v(X, -[\lambda]X+pY)] +  [g \alpha, v(pX, Y)].
      \end{eqnarray}
    \item (Iwahori-Hecke algebra)  $$\mathrm{End}_G(\mathrm{ind}_{IZ}^G d^r) = \begin{cases}
              \Fq[T_{-1,0}, T_{1,2}] \text{ with } T_{-1,0}T_{1,2} = 0 = T_{1,2} T_{-1,0} & \text{if } r \neq 0,p-1, \\
              \Fq[T_{1,0}, T_{1,2}] \text{ with }  T_{1,2}T_{1,0}T_{1,2} = -T_{1,2}, T_{1,0}^2 = 1 & \text{if } r = 0,p-1, \\
            \end{cases} $$
            where the action of the Iwahori-Hecke operators $T_{1,0}$, $T_{-1,0}$ and $T_{1,2}$   
            are given
            by the formulas
            \begin{eqnarray}
              \label{def of Iwahori Hecke ops}
              T_{1,0}\llbracket g,v \rrbracket & = &  \llbracket g \beta, v \rrbracket \nonumber \\
              T_{-1,0}\llbracket g,v \rrbracket & = & \sum_{\lambda \in \Fp} \llbracket g \left(\begin{smallmatrix} p & [\lambda] \\ 0 & 1 \end{smallmatrix} \right), v \rrbracket \\
              T_{1,2}\llbracket g,v \rrbracket & = & \sum_{\lambda \in \Fp} \llbracket g \left(\begin{smallmatrix} 1 & 0 \\ p[\lambda] & p \end{smallmatrix} \right), v \rrbracket, \nonumber
            \end{eqnarray}
            where $\beta = \left(\begin{smallmatrix} 0 & 1 \\ p & 0 \end{smallmatrix} \right)$.
            
\begin{remark}
  We will later see that when $r = 0,p-1$, the operator $T_{-1,0}$ satisfies
            \begin{eqnarray}
              \label{source vs sink}
                T_{-1,0} = T_{1,0} T_{1,2} T_{1,0} 
            \end{eqnarray}
            and so lies in the Hecke algebra, but since
            it is generated by the other two generators, it has been dropped from the list of generators. This relation also
            shows that in this case the Iwahori-Hecke algebra is {\it non-commutative} since otherwise $T_{1,0}$ would commute
            past the $T_{1,2}$ to give $T_{-1,0} = T_{1,2}$ since the square of $T_{1,0}$ is the identity, which as we shall
            see is a contradiction. 
            The Iwahori-Hecke algebra is clearly commutative when $r \neq 0,p-1$, as we
            see from the relations between the generators above.
\end{remark}
\end{itemize}

We now {\bf reinterpret the above formulas for the Hecke operators} in terms of some classical operators
on the vertices and edges in graph theory (in the case $r = 0$).
We make use of the fact that left coset representatives can be taken to be the inverses of right coset
representatives. 
\begin{itemize}
\item We have $\mathrm{ind}_{KZ}^G V_0 = \{f : KZ \backslash G \rightarrow \Fq \} = \{f': G/KZ \rightarrow \Fq \}$ where the first equality
  is by definition (we always implicitly assume the compactly supported condition) and the 
  second equality is obtained by setting $f'(gKZ) = f(KZg^{-1})$.
  Now, this last space of functions can be thought of as functions on the vertices of the tree by Lemma~\ref{vertices-edges}.
  Under this identification we have 
  $$[g,1] \leftrightarrow f'
    =  \text{ characteristic function of the vertex } [gL_0], $$
    since
    \begin{eqnarray*}
       f'(gKZ) =  [g,1](KZg^{-1}) & = & g^{-1}g \cdot 1  =  1, \\
       f'(g'KZ)  =  [g,1](KZg'^{-1}) & = &  0 \> \text{ if } g'KZ \neq gKZ.
    \end{eqnarray*}
    Using this observation, and identifying the characteristic function of a vertex of the tree with the vertex itself,
    we see that the formula for $T$ in \eqref{def of T} is nothing but the usual formula for the {\it sum-of-neighbours
    operator} on vertices of the tree from classical graph theory. Indeed, when $g = 1$, we have
    $$T[1,1] = \sum_{\lambda \in \Fp} [\left(\begin{smallmatrix} p & [\lambda] \\ 0 & 1 \end{smallmatrix} \right), 1] +  [\alpha,1].$$
    But $[1,1]$ corresponds to the standard lattice $L_0$, and as is well known, the lattices
    $\left(\begin{smallmatrix} p & [\lambda] \\ 0 & 1 \end{smallmatrix} \right) L_0$, for $\lambda \in \Fp$, and $\alpha L_0$
    form a complete set of sublattices of $L_0$ of index $p$.
    
\item Similarly, we have $\mathrm{ind}_{IZ}^G d^0 = \{f : IZ \backslash G \rightarrow \Fq \} = \{f': G/IZ \rightarrow \Fq \}$
    where the first equality is again by definition (again the compactly supported condition is implicit) and the 
    second equality is obtained by setting $f'(gIZ) = f(IZg^{-1})$.
    This time, the last space of functions can be thought of as functions on the oriented edges of the tree by Lemma~\ref{vertices-edges}.
    Under this identification we have 
    $$\llbracket g,1 \rrbracket \leftrightarrow f'
    = \text{ characteristic function of the edge } ([gL_0], [g\alpha L_0]),
    $$
    
    since
    \begin{eqnarray*}
      f'(gIZ) =  \llbracket g,1 \rrbracket(IZg^{-1}) & = & g^{-1}g \cdot 1  =  1, \\
      f'(g'IZ)  =  \llbracket g, 1 \rrbracket(IZg'^{-1}) & = & 0 \> \text{ if } g'IZ \neq gIZ.
    \end{eqnarray*}
    Using this and identifying the characteristic function of an edge of the tree with the edge itself,
    we see that the formulas for $T_{1,0}$, $T_{-1,0}$ and $T_{1,2}$ in \eqref{def of Iwahori Hecke ops}
    are nothing but the {\it flip}, {\it source} and {\it sink} operators on the
    oriented edges of the tree. 

    Indeed, when $g = 1$, we have 
    $$T_{1,0} \llbracket 1,1 \rrbracket = \llbracket \beta,1 \rrbracket.$$
    But $\llbracket 1,1 \rrbracket$ corresponds to the standard edge $([L_0],[\alpha L_0])$ and
    $\llbracket \beta,1 \rrbracket$ corresponds to the edge $([\beta L_0],[\beta \alpha L_0])$ =  $([\alpha L_0], [p L_0])$
    (since $\beta = \alpha w$ and $\beta \alpha = p w$) which is the standard edge with flipped orientation (since $[p L_0] = [L_0]$).
    
    Similarly, we have
    \begin{eqnarray*}
      T_{-1,0} \llbracket 1,1 \rrbracket & = & \sum_{\lambda \in \Fp}                              
         \llbracket \left(\begin{smallmatrix} p & [\lambda] \\ 0 & 1 \end{smallmatrix} \right), 1 \rrbracket.
    \end{eqnarray*}                                                                                      
    This time the functions on the right correspond to the edges $( [\left(\begin{smallmatrix} p & [\lambda] \\ 0 & 1 \end{smallmatrix} \right) L_0], [\left(\begin{smallmatrix} p & [\lambda] \\ 0 & 1 \end{smallmatrix} \right) \alpha L_0]) = ( [\left(\begin{smallmatrix} p & [\lambda] \\ 0 & 1 \end{smallmatrix} \right) L_0], [L_0])$ (since $\left(\begin{smallmatrix} p & [\lambda] \\ 0 & 1 \end{smallmatrix} \right) \alpha = p \left(\begin{smallmatrix} 1 & [\lambda] \\ 0 & 1 \end{smallmatrix} \right)$ and 
   $[p \left(\begin{smallmatrix} 1 & [\lambda] \\ 0 & 1 \end{smallmatrix} \right) L_0] = [L_0]$) which are exactly $p$ of the
   edges with target the standard vertex $[L_0]$, which is the source of the standard edge $([L_0], [\alpha L_0])$. Thus the formula 
   for $T_{-1,0}$ reduces 
   to that of the classical source operator on oriented edges in graph theory. 
   
   We remark that the subscripts in this source operator $T_{-1,0}$ also hint at its definition.  Label the vertices 
   corresponding to the lattice spanned
   by $p^{-n} e_1$, $e_2$ by $n$ for $n \in {\mathbb Z}$. Then the formula for $T_{-1,0}$ 
   takes the standard oriented edge $(0,1)$ to all the oriented edges with target the source vertex $0$ - including 
   the edge $(-1,0)$ corresponding to $\lambda = 0$!

    Finally, we have 
    \begin{eqnarray*}
      T_{1,2} \llbracket 1,1 \rrbracket & = & \sum_{\lambda \in \Fp}                              
         \llbracket \left(\begin{smallmatrix} 1 & 0 \\ p [\lambda] & p \end{smallmatrix} \right), 1 \rrbracket.
    \end{eqnarray*}                                                                                      
    This time the functions on the right correspond to the edges 
   $( [\left(\begin{smallmatrix} 1 & 0 \\ p  [\lambda] & p \end{smallmatrix} \right) L_0], [\left(\begin{smallmatrix} 1 & 0 \\ p [\lambda] & p \end{smallmatrix} \right) \alpha L_0]) = ( [\alpha L_0], [L_\lambda])$ 
   (since $\left(\begin{smallmatrix} 1 & 0 \\ p [\lambda] & p \end{smallmatrix} \right) = \alpha  \left(\begin{smallmatrix} 1 & 0 \\ 
   [\lambda] & 1 \end{smallmatrix} \right)$ and so  
   $[\alpha \left(\begin{smallmatrix} 1 & 0 \\ [\lambda] & 1 \end{smallmatrix} \right) L_0] = [\alpha L_0]$) 
   which we claim are exactly $p$ of the
   edges emanating from the sink of the standard edge $([L_0],[\alpha L_0])$. Moreover, none is the flip of the standard edge,
   else we would have $[L_\lambda] = [L_0]$, which implies that $\left(\begin{smallmatrix} 1 & 0 \\ p [\lambda] & p^2 \end{smallmatrix} \right) = kz$ for some $k\in K$ and $z \in Z$ which by taking determinants implies $z = p$ which implies that
   $\left(\begin{smallmatrix} p^{-1} & 0 \\ [\lambda] & p  \end{smallmatrix} \right) \in K$, a contradiction.
   Thus $T_{1,2}$ is the sink operator from classical graph theory. 

   With respect to the labeling of the vertices mentioned above, we see that $T_{1,2}$ takes the standard oriented edge
   $(0,1)$ to all edges
   emanating from its sink  - including the oriented edge $(1,2)$! Thus again the subscripts allow one to recall the definition
   of $T_{1,2}$.

   Using this interpretation of the Hecke operators $T_{1,0}, T_{-1,0}, T_{1,2}$ as the flip, source and sink operators,
   it is now possible to check all the relations satisfied by these operators.
   To start with, we see $T_{-1,0} \neq T_{1,2}$ by their
   graph theoretic interpretations. Moreover:
   \begin{itemize} 
     \item the relation $T_{1,0}^2 = 1$ is now self-evident since flipping orientation twice does nothing.
     \item the relation $T_{1,2} T_{1,0} T_{1,2} = -T_{1,2}$  follows, since on the left the standard edge $(0,1)$ 
              first maps  to the $p$ edges emanating from $1$ such as $(1,2)$, then $T_{1,0}$ flips them, then 
              $T_{1,2}$ maps each flipped edge such as $(2,1)$, to the other edges coming out of $1$, namely
              $(1,0)$ and the other $p-1$ edges coming out of $1$. 
              Summing the action of $T_{1,2}$ of each flipped edge, we see that
              $(1,0)$ gets counted $p$ times, whereas the other edges coming out of $1$ get counted $p-1$ times. 
              But $p = 0$ in the mod $p$ Hecke algebra, so we obtain the formula for $-T_{1,2}$! 
     \item finally the relation \eqref{source vs sink} follows since the operations of flipping an edge, then taking
              the edges coming out of the new sink and finally flipping back is exactly the same thing as  
              taking the edges coming into the source of the original edge!
   \end{itemize}
   For further relations involving the Iwahori-Hecke operators see the work of Barthel and Livn\'e \cite{BL94}, \cite{BL95}.
\end{itemize}

\subsection{The ABC theorem}

Using the above constructions we may now define some {\bf basic mod $p$ representations of $G$}. Let 
\begin{itemize} 
  \item $0 \leq r \leq p-1$,
  \item $\eta : \mathrm{Gal}(\brqp/\qp) \rightarrow \brFp^*$ be a smooth character (also thought of as a character of
           $\qp^*$ by pre-composing with the Artin map $\Qp^* \rightarrow  \mathrm{Gal}(\brqp/\qp)^\mathrm{ab}$,
           and even a character of $G$ by further pre-composing with the determinant $G \twoheadrightarrow \qp^*$), and,
 \item  $\lambda \in \brFp$.
\end{itemize}
With this notation we define the basic mod $p$ representation $\pi(r,\lambda,\eta)$ of $G$ in two ways:
\begin{enumerate}
  \item Let $$\pi(r,\lambda,\eta) = \dfrac{\mathrm{ind}_{KZ}^G V_r}{T - \lambda} \otimes \eta.$$ 
  \item Let $$\pi(r, \lambda, \eta) = \dfrac{\mathrm{ind}_{IZ}^G d^r}{(T_{-1,0} + \delta_{r,p-1} T_{1,0})
                      + (T_{1,2} + \delta_{r,0} T_{1,0} - \lambda)} \otimes \eta$$
                    where $\delta_{i,j} = 1$ if $i = j$ and is $0$ otherwise.
\end{enumerate}

Every subject should have an ABC theorem, and so does this one:

\begin{theorem}[Anandavardhanan, Borisagar, Chitrao]
  The two definitions of $\pi(r,\lambda,\eta)$ given in the spherical and Iwahori cases above coincide. 
\end{theorem}

\begin{proof} When $0 < r < p-1$, the statement is simpler. One needs to show for $\lambda \in \brFp$ (and $\eta = 1$)
  that
   $$ \dfrac{\mathrm{ind}_{IZ}^G d^r}{T_{-1,0} 
     + (T_{1,2}  - \lambda)} \simeq  \dfrac{\mathrm{ind}_{KZ}^G V_r}{T - \lambda}.$$
   For this, it suffices to show that there is an isomorphism 
   $$ \dfrac{\mathrm{ind}_{IZ}^G d^r}{T_{-1,0}} \simeq  \mathrm{ind}_{KZ}^G V_r$$
   with the action of $T_{1,2}$ on the left corresponding to the action of $T$ on the right. 
   This was proved in \cite{AB15} and we do not comment on it further.

   The case $r = 0,p-1$ was
   only proved more recently in \cite{Chi23}. As above, it suffices to show that there is an isomorphism
   $$\dfrac{\mathrm{ind}_{IZ}^G 1}{T_{-1,0} + \delta_{r,p-1} T_{1,0}} \simeq \mathrm{ind}_{KZ}^G V_r$$
   with the action of $T_{1,2} + \delta_{r,0} T_{1,0}$ on the left corresponding to the action of $T$ on the right.
   The argument is a bit more delicate since it involves the non-commutative
   Iwahori-Hecke algebra above. We sketch the main points here - for details see \cite{Chi23}.

   Extend the definition of $V_r$ above to any degree $r \geq 0$. Let $V_r^* \subset V_r$ be the subspace of polynomials
   which are divisible by the Dickson polynomial $\theta = X^{p}Y - XY^{p}$. Then
   $$\dfrac{V_{2p-2}}{V_{2p-2}^*} = V_0 \oplus V_{p-1}$$
   and the generators of these two spaces on the
   right are given by the polynomials $X^{2p-2} - X^{p-1}Y^{p-1} + Y^{2p-2}$ and $X^{2p-2}$ respectively.
   It is not hard to see that $$\dfrac{V_{2p-2}}{V_{2p-2}^*} \simeq \mathrm{ind}_{IZ}^{KZ} 1$$ under the usual
   `evaluation of a polynomial at the lower row of the matrix' map, with $Y^{2p-2} - X^{p-1}Y^{p-1}$ corresponding
   to the function on $KZ$ that takes the coset $IZ$ to $1$ and the other cosets to $0$.
   Inducing both sides from $KZ$ to $G$
   we obtain an isomorphism
   \begin{eqnarray*}
     \mathrm{ind}_{IZ}^G 1 & \rightarrow & \mathrm{ind}_{KZ}^G \dfrac{V_{2p-2}}{V_{2p-2}^*} \\
     \llbracket \id,1 \rrbracket & \mapsto & [\id , Y^{2p-2} - X^{p-1}Y^{p-1}].
   \end{eqnarray*}
   Now the relation $T_{1,2} T_{1,0} T_{1,2} = -T_{1,2}$ in the Iwahori-Hecke algebra shows that
   $-T_{1,2} T_{1,0}$ is an idempotent, so the left hand side decomposes as  $\im (T_{1,2}T_{1,0}) \oplus \im (1+T_{1,2}T_{1,0})$.
   The right hand side decomposes as $\mathrm{ind}_{KZ}^G V_0 \oplus  \mathrm{ind}_{KZ}^G V_{p-1}$. The main technical
   result in \cite{Chi23} is that the above isomorphism takes
   $\im (T_{1,2}T_{1,0})$ to $\mathrm{ind}_{KZ}^G V_{p-1}$ (since one checks $T_{1,2}T_{1,0} \llbracket \id, 1 \rrbracket \mapsto [1,X^{2p-2}])$
   and $\im (1+T_{1,2}T_{1,0})$ to $\mathrm{ind}_{KZ}^G V_0$ (since now $(1+T_{1,2}T_{1,0}) \llbracket \id, 1 \rrbracket \mapsto [1,Y^{2p-2}-X^{p-1}Y^{p-1}+X^{2p-2}]$), 
   and so induces isomorphisms:
   \begin{eqnarray}
     \label{technical}
     \dfrac{\mathrm{ind}_{IZ}^G 1}{T_{1,2} T_{1,0}} \simeq  \mathrm{ind}_{KZ}^G V_{0} & \text{and} & 
     \dfrac{\mathrm{ind}_{IZ}^G 1}{1+T_{1,2} T_{1,0}} \simeq  \mathrm{ind}_{KZ}^G V_{p-1},
   \end{eqnarray}
   where the action of $T_{-1,0} + T_{1,0}$ (respectively $T_{-1,0}$) on the left corresponds to the action of $T$ on the right.
   Here we have used the small checks that $\im (T_{1,2}T_{1,0}) \subset \ker(T_{-1,0} + T_{1,0})$ and similarly
   $\im (1+ T_{1,2}T_{1,0}) \subset \ker(T_{-1,0})$ so that these operators do indeed act on the quotients on the
   left in \eqref{technical}.

   Now clearly the flip involution $T_{1,0}$ induces an isomorphism
   $$\mathrm{ind}_{IZ}^G 1 \simeq \mathrm{ind}_{IZ}^G 1$$
   where the roles of the source operator $T_{-1,0}$ and the sink operator $T_{1,2}$ get interchanged.
   This induces an isomorphism (where we use \eqref{source vs sink})
   $$\dfrac{\mathrm{ind}_{IZ}^G 1}{T_{1,2} T_{1,0}} \simeq  \dfrac{\mathrm{ind}_{IZ}^G 1}{T_{1,0} (T_{1,2} T_{1,0}) = T_{-1,0} +  \delta_{0,p-1} T_{1,0}}$$
   where $T_{-1,0}+T_{1,0}$ on the left corresponds to the conjugate $T_{1,0}(T_{-1,0}+T_{1,0}) T_{1,0} = T_{1,2} + \delta_{0,0} T_{1,0}$ on the right.
   Similarly, we get an induced isomorphism (again by \eqref{source vs sink})
   $$\dfrac{\mathrm{ind}_{IZ}^G 1}{1+T_{1,2} T_{1,0}} \simeq  \dfrac{\mathrm{ind}_{IZ}^G 1}{T_{1,0}(1 + T_{1,2} T_{1,0}) = T_{-1,0} + \delta_{p-1,p-1} T_{1,0}}$$
   where $T_{-1,0}$ on the left corresponds to the conjugate $T_{1,0}(T_{-1,0}) T_{1,0} = T_{1,2} + \delta_{p-1,0} T_{1,0}$ on the right.
   Combining these two isomorphisms with the two in \eqref{technical} proves the theorem for $r = 0, p-1$. 
 \end{proof}
 
We remark that the mod $p$ representation $\pi(r,\lambda, \eta)$ of $G$ is mostly irreducible, e.g., it always is 
when $(r,\lambda) \neq (0, \pm 1)$, $(p-1, \pm 1)$. 

\subsection{mod $p$ Galois representations}

We will be brief here. Let $G_{{\mathbb Q}_p} = \mathrm{Gal}(\brqp/\Qp)$ be the Galois group of $\qp$ and $I_{{\mathbb Q}_p}$
be the inertia subgroup of $G_{{\mathbb Q}_p}$. 

\begin{lemma}
  Every $n$-dimensional irreducible representation $\bar\rho$ of $G_{{\mathbb Q}_p}$ over $\brFp$ is of
  the form $\mathrm{ind}_{G_F}^{G_{{\mathbb Q}_p}} \chi$
  for $F$ the unramified extension of $\Qp$ of degree $n$ and some character $\chi : G_F \rightarrow \brFp^*$.
\end{lemma}

\begin{proof}
  This is perhaps well-known, but it is difficult to find a proof (for
  general $n$) in the literature so we give a proof. 
  As for $n = 2$,  the proof starts by noting that the wild inertia subgroup $I_{{\mathbb Q}_p}^{\mathrm{w}}$ of  $I_{{\mathbb Q}_p}$
  acts trivially. Indeed, $\bar\rho$ has a non-zero  $I_{{\mathbb Q}_p}^{\mathrm{w}}$-invariant vector since the subgroup is a
  pro-$p$ group
  and we are working mod $p$. Moreover, $I_{{\mathbb Q}_p}^{\mathrm{w}}$ is normalized by $G_{{\mathbb Q}_p}$
  so the invariant vectors are  $G_{{\mathbb Q}_p}$-stable, and hence everything, by irreducibility. So
  $\bar\rho |_{I_{{\mathbb Q}_p}}$ factors through the tame inertia group $I_{{\mathbb Q}_p}^{\mathrm{t}} =
  I_{{\mathbb Q}_p} / I_{{\mathbb Q}_p}^{\mathrm{w}}$ which is abelian and prime to $p$, so the restricted representation
  $\bar\rho |_{I_{{\mathbb Q}_p}}$ is 
  a direct sum of mod $p$ characters of $I_{{\mathbb Q}_p}$.
  Let $\chi_0$ be one such, and say the $\chi_0$-isotypic component of $\bar\rho |_{I_{{\mathbb Q}_p}}$has dimension $d$. 
  A Frobenius element $\mathrm{Fr}_p$ of $G_{{\mathbb Q}_p}$ acts on $\chi_0$ by inner conjugation.
  Say that $\chi_i(g)  = \chi(\mathrm{Fr}^i_p g \mathrm{Fr}_p^{-i})$ for $i = 0, \ldots, n/d-1$ are the $n/d$ distinct characters
  so obtained so that $\chi_{n/d} = \chi_0$ and   $$\bar\rho |_{I_{{\mathbb Q}_p}} = \chi_0^{\oplus d} \oplus \chi_1^{\oplus d} \oplus
  \cdots \oplus \chi_{n/d -1}^{\oplus d}.$$ Now $\mathrm{Fr}_p^{n/d}$ preserves $\chi_0^d$, so has an eigen-vector. This will
  also be an eigen-vector under the action of $G_F = \mathrm{Gal}(\brqp/F)$ where $F$ is the maximal unramified extension of $\qp$
  of degree $n/d$ (since $G_F$ is generated by $I_{\qp}$ and $\mathrm{Fr}_p^{n/d}$), say with eigen-character $\chi$.
  Now $\chi \hookrightarrow \bar\rho |_{G_F}$
  implies by Frobenius reciprocity that
  there is a non-zero map $\mathrm{ind}_{G_F}^{G_{{\mathbb Q}_p}} \chi \rightarrow \bar\rho$. 
  This map must be surjective by irreducibility of the target and therefore an isomorphism for dimension reasons
  since the LHS has dimension $n/d$. So we must have $d = 1$, and the lemma follows.
\end{proof}

We return to the case of $n =2$. Recall that $\omega : G_{{\mathbb Q}_p} \rightarrow \Fp^*$ is the fundamental
character of level $1$ (the mod $p$ cyclotomic character). Let $F = {\bq_{p^2}}$ be the unramified
quadratic extension of $\qp$ and  ${G_{\bq_{p^2}}} =  \mathrm{Gal}(\brqp/\bq_{p^2})$. Then recall that 
$\omega^c_2 : I_{{\mathbb Q}_p} \rightarrow \mathbb{F}_{p^2}^*$ is the $c$-th power (for $p+1 \nmid c$) of the
fundamental character of level $2$, extended to $G_{\bq_{p^2}}$ so that
$\det \left(  \mathrm{ind }\> \omega_2^c \right) = \omega^c$, where the induction is from ${G_{\bq_{p^2}}}$ to ${G_{\qp}}$.
Finally, recall $\mu_{\lambda} :  G_{{\mathbb Q}_p} \rightarrow \brFp^*$ is the unramified character taking $\mathrm{Fr}_p^{-1}$
(geometric Frobenius) to $\lambda  \in \brFp^*$.

By the lemma, every two-dimensional (semi-simple) mod $p$ representation of $G_{{\mathbb Q}_p}$ is of the form
\begin{enumerate}
  \item $\mu_{\lambda} \omega^a \oplus \mu_{\lambda'} \omega^b$ (reducible case) 
  \item $\mathrm{ind }\> \omega_2^c \otimes \mu_{\lambda''}$ (irreducible case), 
\end{enumerate}
for some integers $a,b,c$ (with $p+1 \nmid c$), and some $\lambda,\lambda',\lambda'' \in \brFp^*$.

\subsection{Iwahori mod $p$ LLC}

We can now state an Iwahori theoretic version of Breuil's semi-simple mod $p$ LLC for $\qp$.
Recall that for $0 \leq r \leq p-1$, $\lambda \in \brFp$ and $\eta: \Qp^* \to \brFp^*$ a smooth character,
we had defined the following smooth mod $p$ representation of $G$ 
\begin{eqnarray*}
    \label{def of pi}
    \pi(r,\lambda, \eta) & :=      & \frac{\IZind d^r}{(T_{-1, 0} + \delta_{r, p - 1}T_{1, 0})
                                       + (T_{1, 2} + \delta_{r, 0}T_{1, 0} - \lambda)} \otimes \eta,
\end{eqnarray*}
where $\delta_{a,b} = 1$ if $a = b$ and is $0$ otherwise.

  \begin{theorem}[Iwahori mod $p$ LLC]\label{Iwahori mod p LLC}
    For $r \in \{0, \ldots, p - 1\}$, $\lambda \in \brFp$ and  $\eta: \Qp^* \to \brFp^*$ a smooth
    character, we have the following correspondence between mod
    $p$ representations of $G_{{\mathbb Q}_p}$ and certain smooth mod $p$ representations of $G = \mathrm{GL}_2(\qp)$. 
\begin{itemize}
            \item If $\lambda = 0$:
                \[
                    (\ind \omega_2^{r + 1}) \otimes \eta \>\> \longmapsto
                    \>\>  \pi(r,0,\eta) \quad \qquad \qquad \qquad \qquad \quad 
                \]
            \item If $\lambda \neq 0$:
                \begin{eqnarray*}
                    \> \> \> \> \> (\mu_{\lambda}\omega^{r + 1} \oplus
\mu_{\lambda^{-1}})\otimes \eta & \longmapsto &
       \pi(r,\lambda,\eta)^{\rmss} \oplus     \pi([p-3-r],\lambda^{-1},\eta \omega^{r+1})^{\rmss},                             
\\
            \end{eqnarray*}
where $[a] \in \{0,\ldots,p-2\}$ represents the class of $a$ modulo $(p-1)$.    
            \end{itemize}
\end{theorem}

\begin{proof}
  This follows immediately from the ABC theorem and an identical statement (the definition of the semi-simple mod $p$ LLC)
  due to Breuil \cite{Bre03b} but where the $\pi(r,\lambda, \eta)$ are defined using spherical induction. In fact, the end goal
  of \cite{Chi23} was to be able to state such an Iwahori theoretic version of the mod $p$ LLC.
\end{proof}

\section{Functions of a $p$-adic variable}

In this section, we recall some functional analysis for functions of one variable on $\qp$.
Our main reference is \cite{Col10a} which is an excellent introduction to the topic. 

\subsection{$p$-adic Banach spaces}

\subsubsection{$E$-Banach spaces}

Let $E \subset {\mathbb C}_p$ be a subfield, usually taken to be a finite extension of $\qp$, ${\mathcal O}_E$
its valuation ring, ${\mathfrak m}_E$ its maximal ideal, and $k_E = {\mathcal O}_E / {\mathfrak m}_E$ its residue
field. Let $v_p$ be the valuation on ${\mathbb C}_p$ normalized such that $v_p(p) = 1$ and let $|x|_p = p^{-v_p(x)}$
be the corresponding norm on ${\mathbb C}_p$.

Let $B$ be an $E$-vector space. A {\it valuation} on $B$ is a function $v_B : B \rightarrow {\mathbb R} \cup \{ \infty \}$
such that
\begin{itemize}
  \item [i)]   $v_B(x) = \infty \iff x = 0$
  \item [ii)]  $v_B(x+y) \geq \inf (v_B(x), v_B(y))$ for all $x, y \in B$
  \item [iii)] $v_B(\lambda x) = v_p(\lambda) + v_B(x)$ for all $\lambda \in E$, $x \in B$.
\end{itemize}
An {\it $E$-Banach space} is a topological vector space over $E$ with topology given by a valuation $v_B$
with respect to which the topology is complete. A {\it map} $f : B_1 \rightarrow B_2$ of $E$-Banach spaces
is a linear map that is continuous. A map $f : B_1 \rightarrow B_2$ is an isometry if it is bijective and
$v_{B_2}(f(x)) = v_{B_1}(x)$ for all $x \in B$ (this last condition implies that $f$ must be injective).

For example, if $I$ is an indexing set (assumed in these notes to be ${\mathbb N}_{\geq 0}$, mostly 
to avoid some extra language),
then the following are $E$-Banach spaces:
\begin{enumerate}
  \item (bounded sequences) $$l_\infty(I,E) = \{ (a_i)_{i \in I} \mid |a_i|_p \text{ is bounded above} \}$$
    with valuation $v_{l_\infty}((a_i)) = \inf_{i \in I} v_p(a_i)$.
  \item (null sequences) the subspace $$l^0_\infty(I,E) = \{ (a_i)_{i \in I} \mid |a_i|_p \rightarrow 0 \}$$ with
    the same valuation. 
\end{enumerate}
The space $l_\infty^0(I,E)$ is the closure in $l_\infty(I,E)$ of the subspace
    consisting of sequences with only finitely many non-zero terms.

\begin{Proposition}
  \label{Open map BS}
  \begin{enumerate}
  \item (Open mapping theorem)
    If $f: B_1 \rightarrow B_2$ is a map of $E$-Banach spaces with $f$ bijective, then $f^{-1}$ is continuous.
  \item (Banach-Steinhaus theorem) A limit of (linear, continuous) maps between Banach spaces is continuous.
  \end{enumerate}
\end{Proposition}

\subsubsection{ONB of a $p$-adic Banach space}

A family $(e_i)_{i \in I}$ of elements of an $E$-Banach space $B$ is an {\it orthonormal basis} (ONB) of $B$ if the map
\begin{eqnarray}
  \label{isometry}
  \begin{split}
  l_\infty^0(I,E) \quad & \rightarrow & B \\
  (a_i)_{i \in I} \quad & \mapsto & \sum_{i \in I} a_i e_i
  \end{split}
\end{eqnarray}
is an isometry of $E$-Banach spaces. That is,
\begin{itemize}
  \item [i)] every element $x \in B$ can be written uniquely as a convergent series $x = \sum_{i \in I} a_i e_i$ with $|a_i|_p \rightarrow 0$, and
  \item [ii)] $v_B(x) = \inf_{i \in I} v_p(a_i)$.
\end{itemize}
A bit less stringently, a family $(e_i)_{i \in I}$ of elements in $B$
is (only) a {\it Banach basis} if the map displayed above is `only' an isomorphism
of Banach spaces, that is, i) holds but ii) is replaced by the weaker condition:
\begin{itemize}
  \item [ii)$'$] There exists a constant $C \geq 0$ such that $-C +  \inf_{i \in I} v_p(a_i) \leq v_B(x) \leq  C + \inf_{i \in I} v_p(a_i)$.
\end{itemize}
For example, if $I$ is a set and if $\delta_i = (a_j)$ with $a_j = \delta_{i,j}$, then $(\delta_i)_{i \in I}$
forms an ONB of $l^0_\infty(I,E)$.

\begin{Proposition}
  Say $v_p(E)$ is discrete and $\pi_E$ is a uniformizer. Then
  \begin{enumerate}
    \item [i)] Every $E$-Banach space has a Banach basis.
    \item [ii)] An $E$-Banach space possesses an ONB $\iff v_B(B) = v_p(E)$. Moreover, under this hypothesis if
      $B^0 = \{ x \in B \mid v_B(x) \geq 0 \}$, then a family of elements $(e_i)_{i \in I}$ in $B^0$ is an ONB of
      $B \iff (\bar{e}_i)_{i \in I}$ is an algebraic
      basis of the $k_E$-vector space $\bar{B} := B^0/\pi_E B^0$.
  \end{enumerate}
\end{Proposition}  

\begin{proof}
  This proof is adapted from the proof of \cite[Proposition I.1.5]{Col10a} but reminds one of the proof given in Serre's 
  foundational article \cite{Ser62} which is also a good general reference.
  
  First note that we may assume $v_B(B) = v_p(E)$ by replacing $v_B$ by the equivalent valuation
  $v'_B$ defined by $v'_B(x) = v_p(\pi_E) \cdot \lfloor \frac{v_B(x)}{v_p(\pi_E)} \rfloor \in v_p(\pi_E) \cdot {\mathbb Z}$
  which is clearly $v_p(E)$-valued. If one proves ii), so that there is an ONB for $v'_B$, then this basis will be
  a Banach basis for $v_B$ (whose values differ from those of $v'_B$ by at most the constant $v_p(\pi_E)$), so i) will follow.

  Also note that the forward implication in the first statement in
  ii) is clear since $v_p$ is discrete on $E$ and so infimums are attained.
  So to prove ii), it suffices to show that if $v_B(B) = v_p(E)$, then for
    a family $(e_i)$ in $B^0$
  \begin{center}
    $(e_i)$ is an ONB of
    $B \iff (\bar{e}_i)$ is an algebraic basis of $\bar{B}$.
  \end{center}

  ($\Rightarrow$) Suppose $(e_i)$ is an ONB of $B$. 
  If $\bar{x} \in \bar{B}$ for $x \in B^0$,
  then $v_B(x) \geq 0$ so $x = \sum a_i e_i$ with $a_i \in {\mathcal O}_E$ and $a_i \rightarrow 0$. 
  Therefore $\bar{a}_i = 0$
  for $i \gg 0$. So $\bar{x} = \sum \bar{a}_i \bar{e}_i$ is a finite linear combination of the $\bar{e}_i$. So the $(\bar{e}_i)$
  form a spanning set of $\bar{B}$.

  If $\sum \bar{a}_i \bar{e}_i = 0$, with $\bar{a}_i \in k_E$, lifting to say $a_i \in {\mathcal O}_E$, then
  $x = \sum a_i e_i$ satisfies $\bar{x} = 0$, so $v_B(x) > 0$. But $v_B(x) = \inf v_p(a_i)$. So $v_p(a_i) >0$ for all $i$.
  So $\bar{a}_i = 0$ for all $i$. So the $(\bar{e}_i)$ are also linearly independent.   
  
  ($\Leftarrow$) Suppose 
  $(\bar{e}_i)$ 
  is an algebraic basis of $\bar{B}$.
  Let $s : k_E \rightarrow S \subset {\mathcal O}_E$
  be a system of representatives of $k_E$ with $s(0) = 0$. If $x \in B^0$ and $\bar{x}  = \sum \bar{a}_i \bar{e}_i \in \bar{B}$
  with $\bar{a}_i \in k_E$ almost all $0$, set $s(x) = \sum s(\bar{a}_i) e_i$. So $x - s(x) \in \pi_E B^0$.
  Construct by induction a sequence of elements $x_n$ of $B^0$ for $n \geq 0$ by
  setting $x_0 = x$ and $x_{n+1} = \frac{1}{\pi_E}(x_n - s(x_n))$ for $n \geq 0$.
  Then $$x = \sum_{n = 0}^k s(x_n) \pi_E^n + x_{k+1} \pi_E^{k+1}$$ for all $k \geq 0$. 
  Write $s(x_n) = \sum_{i \in I} s_{n,i} e_i$ with $s_{n,i} \in S$ (almost all $0$).
  Set $a_i = \sum_{n=0}^\infty s_{n,i} \pi_E^n$. Clearly $a_i \rightarrow 0$ since for $i$ sufficiently large compared to $n$
  the coefficient $s_{m,i}$ is equal to $0$ for all $m \leq n$. Note that $\sum_{i \in I} a_i e_i = x$.
  Thus the map \eqref{isometry} is surjective onto $B^0$, and therefore onto $B$ (by multiplying and later dividing
  by a sufficiently large power of $\pi_E$).

  Now suppose that $v_{l_\infty}(a_i) = 0$, i.e., $\inf v_p(a_i) = 0$. Then at least one $a_i$ is a unit, so that
  $\sum a_i e_i \neq 0 \mod \pi_EB^0$ since some $\bar{a}_i \neq 0$ and the $\bar{e}_i$ form a basis of $\bar{B}$.
  This implies that $0 \leq v_B(\sum a_i e_i) \lneq v_p(\pi_E)$. But $v_B(B) = v_p(E)$. So we must have
  $v_B(\sum a_i e_i) = 0$. So \eqref{isometry} is an isometry (and is injective).
\end{proof}

\subsubsection{Dual of an $E$-Banach}

We recall the notion of the dual space $B^*$ of an $E$-Banach space $B$: 
it is the $E$-vector space of linear forms $f : B \rightarrow E$ with (the strong) topology defined by the valuation
\begin{eqnarray*}
  \label{dual val}
  v_{B^*}(f) = \inf_{x\in B\setminus \{0\}} \left( v_p(f(x)) - v_B(x) \right)
\end{eqnarray*}
with respect to which it is complete.

\begin{Proposition} Let $I$ ($= {\mathbb N}_{\geq 0}$) be an indexing set.
  \begin{enumerate}
  \item [i)] If $a = (a_i)_{i \in I} \in l^0_\infty(I,E)$, $b = (b_i)_{i \in I} \in l_\infty(I,E)$, then
    the series $$f_b(a) = \sum_{i \in I} a_i b_i$$
    converges in $E$.
  \item [ii)] The map
    \begin{eqnarray*}
      l_\infty(I,E) & \rightarrow & l_\infty^0(I,E)^* \\
           b       &  \mapsto    &  f_b                        
    \end{eqnarray*}
    is an isometry.
  \end{enumerate}
\end{Proposition}

\begin{proof}
  i) is obvious since $a_i \rightarrow 0$ and $(b_i)$ bounded implies $a_i b_i \rightarrow 0$.
  We remark that if $a_j = (a_{i,j}) \rightarrow 0$ as $j \rightarrow 0$, then clearly $f_b(a_j) \rightarrow 0$ as well
  so $f_b$ is continuous and the map in ii) is well defined. The map in ii) is also continuous. Indeed
  \begin{eqnarray*}
    v_{ l_\infty^0(I,E)^*}(f_b) = \inf_{(a_i) \in  l_\infty^0(I,E) \setminus \{0\}} \left( v_p(f_b(a_i)) - v_{l_\infty}(a_i) \right)
    \geq \inf_{(a_i)} \left(  \inf_i (v_p(b_i) + v_p(a_i)) - \inf_i v_p(a_i) \right) \\
    \geq \inf_{(a_i)} \left(  \inf_i v_p(b_i) + \inf_i v_p(a_i) - \inf_i v_p(a_i) \right)  
    =  v_{l_\infty} (b). \qquad \qquad
   \end{eqnarray*}
  Injectivity of the map in ii) is also obvious (if $f_b = 0$, then $b_i = f_b(\delta_i) = 0$ for all $i \in I$,
  so $b = 0$). For surjectivity, let $f \in  l_\infty^0(I,E)^*$.
  If $b_i := f(\delta_i)$, then $v_p(b_i) = v_p(b_i) - v_{l_\infty}(\delta_i) \geq v_{ l_\infty^0(I,E)^*}(f)$
  so $b = (b_i) \in l_\infty(I,E)$ is bounded. Moreover $(f-f_b)(\delta_i) = b_i - b_i = 0$. But the space
  generated by the $\delta_i$ is dense in $l_\infty^0(I,E)$, and $f$ continuous implies $f - f_b$ is
  continuous, so must be identically $0$.  So $f = f_b$ and the map is surjective. It follows from 
  Proposition~\ref{Open map BS}, part i) that the map in ii) is open. Finally, the map in ii) is an isometry.
  Indeed, if $v_{l_\infty}(b) = 0$, then by the continuity proved above $v_{ l_\infty^0(I,E)^*}(f_b) \geq 0$. But also, there
  is a $b_i$ which must be a unit. Then $v_p(f_b(\delta_i)) = v_p(b_i) = 0$ and since $v_{l_\infty} (\delta_i) = 0$
  the infimum $v_{ l_\infty^0(I,E)^*}(f_b)$ of the difference over all non-zero $(a_i)$ is at most $0$, hence is $0$. 
\end{proof}

\subsection{Continuous functions on $\Zp$}

Let $\sC^0(\Zp,E) = \{ g : \Zp \rightarrow E \mid g \text{ is continuous}\}$. Since $\Zp$ is compact,
$g(\Zp)$ is bounded if $g \in \sC^0(\Zp,E)$, so $v_{\sC^0}(g) = \inf_{x \in \Zp} v_p(g(x))$ makes sense
and makes $\sC^0(\Zp,E)$ into an $E$-Banach space.

\subsubsection{Binomial polynomials}

For $n \in {\mathbb N}_{\geq 0}$, let $${x \choose n} = \begin{cases}
                                                            1 & \text{if } n = 0, \\
                                                            \dfrac{x(x-1)\cdots(x-n+1)}{n!} & \text{if } n \geq 1.
                                                          \end{cases} $$

\begin{Proposition}
\label{val 0}
  $v_{\sC^0}({x \choose n}) = 0$ for all $n \geq 0$.
\end{Proposition}

\begin{proof}
  The polynomial ${x \choose n}$ maps ${\mathbb Z}$ to ${\mathbb Z}$ and is continuous (it is a polynomial!), so
  takes $\Zp$ to $\Zp$, so certainly  $v_{\sC^0}({x \choose n}) \geq 0$. But ${n \choose n} = 1$,
  so  $v_{\sC^0}({x \choose n}) = 0$.
\end{proof}

\subsubsection{Mahler coefficients of continuous functions}

For $g \in \sC^0(\Zp,E)$, define $g^{[0]} = g$ and $g^{[k+1]}(x) =  g^{[k]}(x+1) - g^{[k]}(x)$ for $k \geq 0$.
An easy check shows that
\begin{eqnarray}
  \label{binomial derivative}
  {x \choose n}^{[k]} = \begin{cases}
                                    {x \choose n-k} & \text{for } 0 \leq k \leq n, \\
                                    0                       & \text{for } k \geq n.
                                  \end{cases}
\end{eqnarray}

Also define the {\it $n$-th Mahler coefficient} of $g$ to be
$$a_n(g) = g^{[n]}(0).$$   
Then one may check
\begin{eqnarray}
    g^{[n]}(x) & = &  \sum_{i=0}^n (-1)^i {n \choose i} g(x+n-i), \text{ so} \label{derivative formula} \\   
    a_n(g)  & = &  \sum_{i=0}^n (-1)^i {n \choose i} g(n-i). \label{coeff formula}
\end{eqnarray}

  \begin{lemma}
    $g \in \sC^0(\Zp,E) \implies v_{\sC^0}(g^{[p^k]}) \geq v_{\sC^0}(g) + 1$ for some $k \in {\mathbb N}_{\geq 0}$.
  \end{lemma}
  
  \begin{proof}
     Since $\Zp$ is compact, $g$ is uniformly continuous. So there exists $k \geq 0$ such that 
     $v_p(g(x+p^k) - g(x)) \geq v_{\sC^0}(g) + 1$  for all $x \in \zp$. Fix $x \in \Zp$.
     By \eqref{derivative formula} and adding and subtracting $g(x)$ we can write
     $$g^{[p^k]}(x) = \left( g(x+p^k) - g(x) \right)
                            + \left( \sum_{i=1}^{p^k-1} (-1)^i {p^k \choose i} g(x+p^k-i) \right)
                            + \left( (1+(-1)^{p^k}) g(x) \right).$$
     Now the valuation $v_p$ of each of the terms in parentheses on the right is at least $$v_{\sC^0}(g) + 1,$$ the 
     first by what we just deduced above,
     the second since $p | {p^k \choose i}$ for $i \neq 0, p^k$, and the last if $p = 2$ (though this term 
     vanishes if $p$ is odd).
     This proves that the valuation $v_p$ of the term on the left is also bounded below by the 
     above quantity, from which we deduce the
     lemma by varying over all $x \in \Zp$.
\end{proof}

\begin{theorem} \label{mahler}
  \begin{enumerate}
  \item [i)] If $g \in \sC^0(\Zp,E)$, then
    \begin{itemize}
    \item [a)] $a_n(g) \rightarrow 0$, and
    \item [b)] $\sum_{n=0}^\infty a_n(g) {x \choose n} = g(x)$, for all $x \in \Zp$.
    \end{itemize}
  \item [ii)] The map
    \begin{eqnarray*}
      \sC^0(\Zp,E) & \rightarrow & l_\infty^0({\mathbb N}_{\geq 0}, E) \\
      g          & \mapsto     & (a_n(g))_{n \geq 0}
    \end{eqnarray*}
    is an isometry.
  \end{enumerate}
\end{theorem}

\begin{proof}   
     The proof here adapted from the proof of \cite[Theorem I.2.3]{Col10b}; another good exposition 
     can be found in Hida's text book \cite{Hid93}.

     Let $g \in \sC^0(\Zp,E)$. By repeated use of the lemma, and the obvious 
     fact that $g^{[l+k]} = (g^{[l]})^{[k]}$ for $l,k \geq 0$, we see that there are $k_1, k_2, \ldots, k_m$ such that
     $v_{\sC^0}(g^{[p^{k_1} + \cdots + p^{k_m}]}) \geq v_{\sC^0}(g) + m$. Taking $m$ sufficiently
     large, we see that given $C \geq 0$, there exists $N$ such that $v_{\sC^0}(g^{[N]}) \geq C$.
     Then  by \eqref{coeff formula} we have
     $v_p(a_n(g)) = v_p(a_{n-N}(g^{[N]})) \geq v_{\sC^0}(g^{[N]}) \geq C$ for $n \geq N$. 
     We deduce that $a_n(g) \rightarrow 0$. This proves i) a).
     
   This shows that the map in ii) 
  \begin{eqnarray*}
    \sC^0(\Zp,E) & \rightarrow & l_\infty({\mathbb N}_{\geq 0},E) \\
      g          & \mapsto     & a(g) := (a_n(g))_{n \geq 0}
  \end{eqnarray*}
  is well defined.
  It is also clearly injective since  if the Mahler coefficients $a_n(g)$ of $g$ vanish, then by a recursive use of 
  \eqref{coeff formula} the
  continuous function $g$ vanishes on the dense set ${\mathbb N}_{\geq 0} \subset \zp$ and hence $g = 0$.
  The map is also continuous since again by \eqref{coeff formula}
  \begin{eqnarray}
    \label{A}
    v_{l_\infty} (a(g)) = \inf_{n \geq 0} v_p(a_n(g)) \geq   \inf_{n \geq 0, 0 \leq i \leq n} v_p(g(n-i)) \geq \inf_{x \in \Zp} v_p(g(x))
  = v_{\sC^0}(g).
  \end{eqnarray}
  We prove the surjectivity of the above map onto $l_\infty^0({\mathbb N}_{\geq 0}, E)$. We first claim that if
  $a = (a_n)_{n \geq 0} \in l_\infty^0({\mathbb N}_{\geq 0}, E)$, then
  $$s_\infty(x) = \sum_{k = 0}^\infty a_k {x \choose k}$$ 
  is a continuous function of $x \in \Zp$. First note that the sequence of partial sums $s_n(x)$
  converges uniformly to $s_\infty(x)$ on $\Zp$. Indeed, given $C \geq 0$, there exists $N_C$ such that
  $$v_p(s_n(x) - s_\infty(x)) = v_p \left( \sum_{k=n+1}^\infty a_k {x \choose k} \right) \geq
  \inf_{k \geq n+1} (v_p(a_k) + v_{\sC^0}({x \choose k}))
  \geq  \inf_{k \geq n+1} v_p(a_k) \geq C$$
  for all $n \geq N_C$, since $a_k \rightarrow 0$, and where we have used Proposition~\ref{val 0}.
  A uniform limit of continuous functions is continuous, so $g_a := s_\infty : \zp \rightarrow E$ is continuous.
  Now, by \eqref{binomial derivative}, $g_a^{[n]}(0) = a_n$, so under the above map $g_a \mapsto a$,
  and so the above map is surjective onto $l_\infty^0({\mathbb N}_{\geq 0}, E)$.
  
  Now if $g \in \sC^0(\zp,E)$, then $g$ and $g_{a(g)}$ have the same  Mahler coefficients
  so must be equal since the map above is injective. Thus  $g$ satisfies i) b). 
  
  Note that 
  \begin{eqnarray}
    \label{B}
      v_{\sC^0}(g_a) \geq \inf_{n \geq 0} v_{\sC^0}(a_n {x \choose n}) \geq \inf_{n \geq 0} v_p(a_n) = v_{l_\infty}(a)
  \end{eqnarray}
  by Proposition~\ref{val 0}.
  So if $g \in \sC^0(\zp,E)$, then by \eqref{A} and \eqref{B}, we have
  $$v_{l_\infty}(a(g)) \geq v_{\sC^0}(g) = v_{\sC^0}(g_{a(g)}) \geq v_{l_\infty} (a(g)).$$
  So $v_{l_\infty}(a(g)) = v_{\sC^0}(g)$ and the map in ii) is an isometry. 
\end{proof}

\begin{Corollary}
  The ${x \choose n}$ for $n \geq 0$ form an ONB of $\sC^0(\Zp,E)$.
\end{Corollary}

\subsection{Wavelet decompositions of continuous functions}

\subsubsection{Locally constant functions}
Let 
\begin{eqnarray*}
  \mathrm{LC}_h(\zp,E) & = & \{ g : \zp \rightarrow E \mid g |_{a + p^h \zp} \text{ is constant for all } a \in \zp \}. \\
  \mathrm{LC}(\zp,E) & =  & \bigcup_{h \geq 0} \mathrm{LC}_h(\zp,E).
\end{eqnarray*}

We start with the following well-known fact.

\begin{lemma}
  \label{LC dense in C^0}
  $\mathrm{LC}(\Zp,E) \subset \sC^0(\zp,E)$ is dense.
\end{lemma}

\begin{proof}
  Let $g \in \sC^0(\zp,E)$. Then $g$ continuous, $\zp$ compact $\implies g$ is uniformly continuous, so for all $C \geq 0$, 
  there exists an $m \geq 0$ such that $v_p(x-y) \geq m \implies v_p(g(x) - g(y)) \geq C$. Let
  $$g_m(x) = \sum_{i = 0}^{p^m - 1} g(i) {\mathbbm 1}_{i + p^m \Zp}(x).$$
  If $x \in \Zp$, then there exists a unique $0 \leq i \leq p^{m}-1$ such that $x \in i + p^m \zp$, so 
  $v_p(g(x) - g_m(x)) = v_p(g(x) - g(i)) \geq C$. Then 
  $v_{\sC^0}(g-g_m) = \inf_{x \in \zp} v_p(g(x) - g_m(x)) \geq C$. 
\end{proof}

If $i \in {\mathbb N}_{\geq 0}$, set $l(i)$ to be the smallest $n \geq 0$ such that $p^n > i$.
So $l(0) = 0$ and $l(i) = \lfloor \frac{\log i}{\log p} \rfloor + 1$.

\begin{Proposition}
  \label{LC basis}
  We have
  \begin{enumerate}
    \item The ${\mathbbm 1}_{i + p^{l(i)} \zp}$ for $0 \leq i \leq p^h-1$ form a basis of  $\mathrm{LC}_h(\zp,E)$.
    \item The ${\mathbbm 1}_{i + p^{l(i)} \zp}$ for $i \geq 0$ form a basis of  $\mathrm{LC}(\zp,E)$.
    \item The ${\mathbbm 1}_{i + p^{l(i)} \zp}$ for $i \geq 0 $ form an ONB of  $\sC^0(\zp,E)$.
  \end{enumerate}
\end{Proposition} 

\begin{proof}
  By definition, the ${\mathbbm1}_{i + p^h \Zp}$ form a basis of  $\mathrm{LC}_h(\zp,E)$. Also, for $i \leq p^h-1$
  $${\mathbbm 1}_{i + p^{l(i)} \zp} = \sum_{j =0}^{p^{h - l(i)} -1} {\mathbbm 1}_{i +j p^{l(i)}  + p^h \zp}.$$
  One checks that the change of basis matrix (of size $p^h \times p^h$) is lower-triangular with $1$'s on the diagonal, so
  is invertible, so i), ii) follow.
   
   Consider the isometry 
   \begin{eqnarray*}
     \{\text{a. e. zero sequences}\} & \rightarrow & \mathrm{LC}(\zp,E) \\
       (a_i)_{i \geq 0}       & \mapsto     & \sum_{i \geq 0} a_i {\mathbbm 1}_{i +p^{l(i)}  \zp}.
  \end{eqnarray*}
  Since the LHS is dense in $l_\infty^0({\mathbb N}_{\geq 0},E)$, the above map extends to an isometry
  \begin{eqnarray*}
     l_\infty^0({\mathbb N}_{\geq 0},E) & \rightarrow & \overline{\mathrm{LC}(\zp,E)} = \sC^0(\zp,E) 
  \end{eqnarray*}
  where we have used Lemma~\ref{LC dense in C^0}.
  \end{proof}

The ONB of $\sC^0(\zp,E)$ consisting of ${\mathbbm 1}_{i + p^{l(i)} \zp}$ is called the {\it basis of wavelets}.
If $g \in \sC^0(\zp,E)$ satisfies
$$g = \sum_{i \geq 0} b_i(g) {\mathbbm 1}_{i+p^{l(i)} \zp},$$
then the $b_i(g) \in E$ are called the {\it amplitude coefficients}  of $g$.

\subsubsection{Mahler coefficients of locally constant functions}

{\bf \noindent Notation:} If $z \in E$ and $v_p(z-1) > 0$, then $$g_z(x) = \sum_{n=0}^\infty {x \choose n} (z-1)^n$$ 
converges uniformly (one uses Proposition~\ref{val 0} again, see the proof of Theorem~\ref{mahler})
so is a continuous function of $x \in \zp$.
Also $g_z(k) = \sum_{n=0}^\infty {k \choose n} (z-1)^n = (z-1 +1)^k = z^k$, so 
we may speak of $x \mapsto z^x$ as a function on $\zp$.
If $z = \zeta_{p^n} \in \mu_{p^n}$ is a $p^n$-th root of unity, then $z^{p^n} = 1 \implies z^{x+ p^n\zp} = z^x$ for $x \in \zp$,
so $z^x$ is a locally constant function on $\zp$.

\begin{Proposition}
  \label{LC basis}
  \begin{enumerate}
    \item Say $\mu_{p^h} \subset E$. Then $\zeta_{p^h}^x$ for $\zeta \in \mu_{p^h}$  form a basis of $\mathrm{LC}_h(\zp,E)$. 
    \item Say $\mu_{p^\infty} \subset E$. Then $\zeta^x$ for $\zeta \in \mu_{p^\infty}$ form a basis of $\mathrm{LC}(\zp,E)$. 
  \end{enumerate}
\end{Proposition}

\begin{proof}
  For i), note that for $x \in \zp$, we have
  ${\mathbbm 1}_{a+p^h\zp}(x)= \frac{1}{p^h} \sum_{\zeta \in \mu_{p^h}} \zeta^{x-a}$ since
  $$\sum_{\zeta \in \mu_{p^h}} \zeta^x = \begin{cases} p^h & \text{if } x \in p^h \zp \\
                                                                                        0   & \text{otherwise.}
                                                              \end{cases}$$                                                             
  Thus the functions $\zeta^x$ for $\zeta \in \mu_{p^h}$ generate   
  $\mathrm{LC}_h(\zp,E)$ and even form a basis of  $\mathrm{LC}_h(\zp,E)$ 
  since their number  is $p^h$.
  
  Statement ii) follows from i).  
\end{proof}

\begin{remark}
For all $i,j \geq 0$, set
$$\alpha_{i,j} = \frac{1}{p^{l(i)}} \sum_{\zeta \in \mu_{p^{l(i)}}}\zeta^{-i} (\zeta - 1)^j$$
\vspace{1cm}
Using some algebraic number theoretic arguments (see \cite[Lemma I.3.5]{Col10a}), it can be shown that
$$\begin{cases}
     \alpha_{i,j} =0     & \text{if } j < i \\
     \alpha_{i,j} = 1     & \text{if } j = i \\
     v_p(\alpha_{i,j}) \geq \left\lfloor \dfrac{j - p^{l(i) -1}}{p^{l(i)} - p^{l(i) -1}}  \right\rfloor & \text{if } j > i.
    \end{cases}
 $$
Now 
$${\mathbbm 1}_{i+p^{l(i)}\zp}(x) =    \frac{1}{p^{l(i)}} \sum_{\zeta \in \mu_{p^{l(i)}}} \zeta^{x-i} 
      =    \sum_{n = 0}^\infty  {x \choose n} \sum_{\zeta \in \mu_{p^{l(i)}}}   \frac{1}{p^{l(i)}}  \zeta^{-i} (\zeta -1)^n$$          
 so its Mahler coefficients $a_n({\mathbbm 1}_{i+p^{l(i)}\zp})$ equal $\alpha_{i,n}$, which clearly tend to $0$ as $n \to \infty$. 
 This gives another proof of the surjectivity of the map in Theorem~\ref{mahler}, ii). Indeed considering
 that map to be taking values in the bigger space $l_\infty({\mathbb N}_{\geq 0},E)$, we note that the pre-image $B$ 
 of the closed subspace $l_\infty^0({\mathbb N}_{\geq 0}, E)$ is closed and, by the above remarks and 
 Proposition~\ref{LC basis}, part ii), contains
 $\mathrm{LC}(\zp,E)$. Then $B = \overline{B} \supset \overline{\mathrm{LC}(\zp,E)} = \sC^0(\zp, E)$,
 by Lemma~\ref{LC dense in C^0}. So $B = \sC^0(\zp, E)$.
\end{remark}

\subsection{Locally analytic functions}

\subsubsection{Locally analytic functions on a closed disk}

We will be brief. For $a \in E$ and $r \in {\mathbb R}$, let 
$$B(a,r)  = \{ x \in {\mathbb C}_p \mid v_p(x-a) \geq r \}$$
be the ball of radius $p^{-r}$ centered at $a$ in ${\mathbb C}_p$.

A function $g : B(a,r) \rightarrow  {\mathbb C}_p$ is {\it $E$-analytic} if there exists a sequence $a_k(g,a)$ of elements of $E$ for $k \geq 0$ such that $v_p(a_k(g,a) + kr) \rightarrow \infty$
as $k \rightarrow \infty$ and 
\begin{eqnarray}
    \label{val B(a,r)}
    g(x) = \sum_{k=0}^\infty a_k(g,a)(x-a)^k
\end{eqnarray}
for all $x \in B(a,r)$. Let
$${\mathrm{An}}(B(a,r),E) =  \{E\text{-analytic functions on } B(a,r)\}.$$
This is an $E$-Banach space under 
\begin{eqnarray*}
  v_{B(a,r)}(g) = \inf_{k \geq 0} v_p(a_k(g,a) + kr).
\end{eqnarray*}

\begin{Proposition}
If $g_1,g_2 \in {\mathrm{An}}(B(a,r),E)$, then $g_1g_2 \in {\mathrm{An}}(B(a,r),E)$.
\end{Proposition} 

\begin{proof}
  See \cite[Proposition I.4.2]{Col10a}.
\end{proof}

\begin{Proposition}
  If $g \in {\mathrm{An}}(B(a,r),E)$, then $v_{B(a,r)}(g) = \inf_{x \in B(a,r)} v_p(g(x))$.
\end{Proposition}

\begin{proof}
  See \cite[Proposition I.4.3]{Col10a}.
\end{proof}

\subsubsection{Locally analytic functions on $\Zp$}

Define for $h \in {\mathbb N}_{\geq 0}$
$$\mathrm{LA}_h(\zp,E) = \{g : \zp \rightarrow E \mid \text{for } a \in \zp, \>
    g\mid_{a + p^h \zp} = g_{a,h} 
    \text{ for some } g_{a,h} \in \mathrm{An}(B(a,h),E) \}. $$
This is an $E$-Banach space under 
\begin{eqnarray*}
    v_{{\mathrm LA}_h}(g) = \inf_{a \in \zp} v_{B(a,h)} (g_{a,h}).
\end{eqnarray*}

\begin{remark}
   \label{alt def val}
   There is an alternative description to $\mathrm{LA}_h(g)$. 
    For each $a \in \Zp$, write $g$  as 
    \begin{eqnarray}
        \label{val B(a,r) again}
      g(x) = \sum_{k=0}^\infty  \tilde{a}_k(g,a)  \left(\frac{x-a}{p^h}\right)^k.
    \end{eqnarray}
    Since the coefficients  in  \eqref{val B(a,r)}  and \eqref{val B(a,r) again}
    are related by $a_k(g,a) = \tilde{a}_k(g,a) p^{-hk}$,
    we see that  $v_{B(a,h)}(g) = \inf_{k \geq 0} v_p(\tilde{a}_k(g,a))$. Thus
    $$ v_{{\mathrm LA}_h}(g) = \inf_{a \in \zp} \inf_{k \geq 0} v_p(\tilde{a}_k(g,a)).$$
    Note that both here and just before the remark, we may restrict the infimum over $a \in \zp$ to 
    a (finite) set of representatives $a$ of $\zp/p^h\zp$.
\end{remark}

Let 
$$\mathrm{LA}(\zp,E) = \bigcup_{h \in {\mathbb N}_{\geq 0}} \mathrm{LA}_h(\zp,E).$$

\subsubsection{Mahler coefficients of locally analytic functions}

We only state the following interesting results as background since we will not use them.

\begin{theorem}
  The $\lfloor \frac{n}{p^h} \rfloor ! {x \choose n}$ for $n \in {\mathbb N}_{\geq 0}$ form an ONB of $\mathrm{LA}_h(\zp,E)$.
\end{theorem}

\begin{proof}
  See \cite[Theorem I.4.7]{Col10a}.
\end{proof}

\begin{Corollary}
  Say $g \in \sC^0(\zp,E)$. Then 
  $$g \in \mathrm{LA}(\zp,E) \iff \liminf_{n \geq 0} \left(\frac{v_p(a_n(g))}{n} \right) > 0.$$
\end{Corollary}

\begin{proof}
  See \cite[Corollary I.4.8]{Col10a}.
\end{proof}

\subsection{Functions of class $\sC^r$}

\subsubsection{Differentiable functions}

A function $g : \zp \rightarrow E$ is {\it differentiable} at $x_0 \in \Zp$ if
$$ \lim_{h \rightarrow 0} \dfrac{g(x_0 + h) - g(x_0)}{h}$$
exists. If the limit exists, then it is denoted as in the real world (pun intended) by $g'(x_0)$ or $g^{(1)}(x_0)$.
We say  $g : \zp \rightarrow E$ is {\it differentiable of order $1$} if $g$ is 
differentiable at each $x_0 \in \Zp$. As usual  such a function is continuous on $\zp$
since for all $x_0 \in \zp$, we have $\lim_{h \rightarrow 0} g(x_0 + h) - g(x_0) = g'(x_0) \lim_{h \rightarrow 0} h = 0$. 
More generally, say that $g : \zp \rightarrow E$ is {\it differentiable of order $k$} with the $k$-th derivative 
denoted by $g^{(k)}$ if $g$ is differentiable of order $k-1$ and $g^{(k-1)}$ is differentiable of order $1$.

We now come to one of the most important definition of these notes. Let $r \geq 0$ be a real number. We say
that a function $g : \zp \rightarrow E$ is of {\it class $\sC^r$} if there exist functions $$g^{(j)} : \zp \rightarrow E$$ 
for $0 \leq j \leq \lfloor r \rfloor$ such that if $\varepsilon_{g,r} : \zp \times \zp \rightarrow E$ is given by
\begin{eqnarray}
  \label{varepsilon}
  \varepsilon_{g,r}(x,y) = g(x+y) - \sum_{j=0}^{\lfloor r \rfloor} g^{(j)}(x) \frac{y^j}{j!}
\end{eqnarray}
and $C_{g,r} : {\mathbb N}_{\geq 0} \rightarrow {\mathbb R} \cup \{\infty\}$ is given by
$$C_{g,r}(h) = \inf_{x \in \zp \atop y \in p^h \zp} (v_p(\varepsilon_{g,r}(x,y)) - rh),$$
then $C_{g,r}(h) \rightarrow \infty$ as $h \rightarrow \infty$.

\begin{remark}
\begin{itemize}
\item Taking $y = 0$, we have $\varepsilon_{g,r}(x,0) = g(x) - g^{(0)}(x)$, so $$C_{g,r}(h) \leq v_p( g(x) - g^{(0)}(x)) - rh.$$
  Since the LHS goes to $\infty$ as $h \rightarrow \infty$, so must the RHS, and we deduce
  $g^{(0)}(x) = g(x)$ for all $x \in \Zp$.
\item One may check that  more generally the $g^{(j)}$ in the definition are the $j$-th derivatives of $g$ for $0 \leq j \leq   
  \lfloor r \rfloor$ (see Remark~\ref{derivative}).
\item Fix $x \in \Zp$. If follows from \eqref{varepsilon} that for $h \in {\mathbb N}_{\geq 0}$ and $y \in p^h \zp$
  $$v_p(\varepsilon_{g,r}(x,y))  \geq \min \left(v_p(g(x+y) - g(x)), \min_{j = 1, \ldots, \lfloor r \rfloor} v_p(\frac{g^{(j)}(x)}{j!} y^j)       
  \right).$$
  Now the LHS goes to $\infty$ as $h \rightarrow \infty$ (even after subtracting $rh$ from it). 
  But each of the terms except the first on the RHS also goes to $\infty$ as $h \rightarrow \infty$ since 
  $v_p(\frac{g^{(j)}(x)}{j!} y^j) \geq v_p(\frac{g^{(j)}(x)}{j!}) + jh$ and $x$ is fixed. Thus the first term 
  on the RHS must also go to $\infty$
  as $h \rightarrow \infty$ (else the inequality in the display above would be an equality and this would lead to a contradiction)
  and we deduce that $g$ is continuous at $x$. 
\item Finally if $r = 0$, then \eqref{varepsilon} gives $$v_p(\varepsilon_{g,r}(x,y)) - rh = v_p(g(x+y) - g(x)).$$ Since
  the infimum of the LHS over all $x \in \Zp$ and $y \in p^h \Zp$ goes to $\infty$ as $h \rightarrow \infty$, so must the RHS,     
  and we deduce that $g$ is uniformly continuous on 
  $\Zp$. This of course already follows from the previous bullet point which shows $g$ is continuous on the compact
  domain  $\zp$.
\end{itemize}
\end{remark}

We let
$$\sC^r(\zp,E) = \{g : \zp \rightarrow E \mid g \text{ is of class } \sC^r \}.$$
This is an $E$-Banach space under 
\begin{eqnarray}
  \label{val C^r prime}
  v'_{\sC^r}(g) = \min \left( \inf_{0 \leq j \leq \lfloor r \rfloor \atop x \in \zp} v_p(\frac{g^{(j)}(x)}{j!}), \inf_{x \in \Zp \atop y \in \Zp} (v_p(\varepsilon_{g,r}(x,y))-r v_p(y))  \right).
\end{eqnarray}

The following remark \cite[Remark I.5.2]{Col10a} is worth expanding on a bit:

\begin{remark}
\label{diff but not C^r}
If $g : \zp \rightarrow E$ is differentiable of order $r$, with the $r$-th derivative $g^{(r)}$ continuous, 
then $g$ does not necessarily lie in $\sC^r(\zp,E)$. Consider the function $g : \zp \rightarrow \zp$ given by 
$$g : \sum_{n=0}^\infty a_n p^n \mapsto \sum_{n=0}^\infty a_n p^{2n}.$$ Then one checks that
$v_p(g(x) - g(y)) \geq 2 v_p(x-y)$ for all $x,y \in \zp$. For instance, take $y =0$ and note $g(0) = 0$. 
Take $x \neq 0$ and let $n_0$ be the first index  in the base $p$ expansion of $x$ such that $a_{n_0} \neq 0$. 
Then the first non-zero coefficient in the base $p$ expansion of $g(x)$ is the coefficient of $p^{2n_0}$, 
so we see $v_p(g(x)) = 2n_0 = 2 v_p(x)$.
In any case, we have
$$v_p\left(\frac{g(x)-g(y)}{x-y}\right) \geq v_p(x-y),$$ so $g$ is differentiable at all $x \in \Zp$ 
with $g'(x) = 0$. Hence $g''(x)$ is identically $0$ (as are all higher derivatives). 
But $g \notin \sC^2(\zp,E)$. Indeed, $v_p(\varepsilon_{g,2}(0,p^h)) -2h = v_p(g(0+p^h) - g(0)) -2h= 2h-2h = 0$,
so the infimum of $\varepsilon_{g,2}(x,y) -2h$ over all $x \in \Zp$ and $y \in p^h\zp$ is at most $0$ and cannot
escape to $\infty$. (It turns out that $g \in \sC^1(\zp,\qp)$, however.)
\end{remark}

\subsubsection{Local properties of $\sC^r$-functions}

\begin{lemma}
 \label{sum estimate}
 Let $C(N) = \sum_{n=1}^N v_p(n!)$ and $a_j \in E$ for $0 \leq j \leq N$.
 Then for all $h \in {\mathbb Z}$
 $$\min_{0 \leq j \leq N} (v_p(a_j) + jh) \geq \left( \min_{z \in p^h \zp} v_p(\sum_{j=0}^N a_j \frac{z^j}{j!}) \right) - C(N).$$
 \end{lemma}
 
 \begin{proof}
   See \cite[Lemma I.5.3]{Col10a}.
 \end{proof}

\begin{Proposition}
If $r \geq 1$ and $g \in \sC^r(\zp,E)$, then $g$ is differentiable at each $x \in \zp$. Moreover
\begin{enumerate}
  \item [a)] $g' \in \sC^{r-1}(\zp,E)$ 
  \item [b)] $(g')^{(j)} = g^{(j+1)}$ if $j \leq r-1$.
\end{enumerate}

\begin{proof}
Say $g \in \sC^r(\zp,E)$ with $r \geq 1$. Let $x \in \Zp$ and $y \in p^h \zp \setminus p^{h+1}\zp$. Then 
$$\frac{g(x+y) - g(x)}{y} = \sum_{j=1}^{\lfloor r \rfloor} \frac{g^{(j)}(x) y^{j-1}}{j!} + \frac{\varepsilon_{g,r}(x,y)}{y}.$$
The last term on the right tends to $0$ as $h \rightarrow \infty$ since $v_p(\varepsilon_{g,r}(x,y)) - v_p(y) \geq  
v_p(\varepsilon_{g,r}(x,y)) - r h \geq C_{g,r}(h) \rightarrow \infty$. On the other hand as $y \rightarrow 0$, the LHS tends to $g'(x)$, and the first term on the right tends to $g^{(1)}(x)$. This proves that $g$ is differentiable at each $x \in \Zp$ 
and $g' = g^{(1)}$.

Also 
\begin{eqnarray}
  \label{vareps computation}
  \varepsilon_{g,r}(x,y+z) - \varepsilon_{g,r}(x+y,z)  
             & = &  g(x+y+z) - \sum_{j=0}^{\lfloor r \rfloor} \frac{g^{(j)}(x) (y+z)^{j}}{j!}
               - \left(g(x+y+z) - \sum_{j=0}^{\lfloor r \rfloor} \frac{g^{(j)}(x+y) z^{j}}{j!} \right) \nonumber \\            
             & = & \sum_{j=0}^{\lfloor r \rfloor} \frac{g^{(j)}(x+y) z^{j}}{j!} - 
                           \sum_{j=0}^{\lfloor r \rfloor} g^{(j)}(x) \sum_{k=0}^j \frac{y^{j-k} z^k}{(j-k)!k!} \nonumber \\
             & = & \sum_{j=0}^{\lfloor r \rfloor} \frac{z^{j}}{j!} \left(  g^{(j)}(x+y) - \sum_{k=0}^{\lfloor r \rfloor -j} g^{(j+k)}(x) \frac{y^{k}}{k!} \right)         
\end{eqnarray}
since the second sum in the second line is (substitute $u = j-k$ and then relabel):
$$ \sum_{k=0}^{\lfloor r \rfloor}  \sum_{j=k}^{\lfloor r \rfloor} g^{(j)}(x) \frac{y^{j-k} z^k}{(j-k)!k!} 
  = \sum_{k=0}^{\lfloor r \rfloor}  \sum_{u=0}^{\lfloor r \rfloor -k} g^{(u+k)}(x) \frac{y^{u} z^k}{u!k!}
  = \sum_{j=0}^{\lfloor r \rfloor}   \frac{z^j}{j!}  \sum_{k=0}^{\lfloor r \rfloor -j} g^{(j+k)}(x) \frac{y^{k}}{k!}. $$
If $v_p(y), v_p(z) \geq h$, then the LHS, hence the RHS,  of \eqref{vareps computation} is bounded below by $C_{g,r}(h) + rh$. Applying Lemma~\ref{sum estimate} with $N = \lfloor r \rfloor$ and $a_j$ the quantity in the large parentheses above we obtain that for each $0 \leq j \leq \lfloor r \rfloor$
\begin{eqnarray}
  \label{vareps estimate again}
  v_p \left( g^{(j)}(x+y) - \sum_{k=0}^{\lfloor r \rfloor -j} g^{(j+k)}(x) \frac{y^{k}}{k!} \right) - (r-j)h  \geq C_{g,r}(h) - C(\lfloor r \rfloor).
 \end{eqnarray}
This shows that for each $0 \leq j \leq \lfloor r \rfloor$, we have $g^{(j)} \in \sC^{r-j}(\zp,E)$. It also shows that 
  $$(g^{(j)})^{(k)} = g^{(j+k)}$$
  if $j + k \leq \lfloor r \rfloor$. This proves a) 
  and b).
\end{proof}
\end{Proposition}

\begin{remark}
  \label{derivative}
  It follows that if $g \in \sC^r(\zp,E)$, then $g$ is differentiable of order $\lfloor r \rfloor$ and 
  $g^{(j)}$ is the $j$-th derivative of $g$ for $j \leq \lfloor r \rfloor$. 
 \end{remark} 
   
 \begin{Proposition}
   \label{prop comp}
   If $g_1 : \zp \rightarrow \zp$ if of class $\sC^r$ and $g_2 \in \sC^r(\zp,E)$, then $g_2 \circ g_1 \in \sC^r(\zp,E)$.
 \end{Proposition}
 
 \begin{proof}
   See \cite[Proposition I.5.6]{Col10a}.
 \end{proof}

\subsubsection{Locally analytic functions and functions of class $\sC^r$}

\begin{Proposition}
  \label{LA vs Cr}
  If $h \in {\mathbb N}_{\geq 0}$ and $r \geq 0$, then $\mathrm{LA}_h(\zp,E) \subset \sC^r(\zp,E)$.
  Moreover, if $g \in \mathrm{LA}_h(\zp,E)$, then $v'_{\sC^r}(g) \geq v_{\mathrm{LA}_h}(g) - rh$.
\end{Proposition}

\begin{proof}
  Let $g \in \mathrm{LA}_h(\zp,E)$. On $a + p^h \zp$, we have $g(x) = \sum_{k =0}^{\infty} \tilde{a}_k(g,a) \left(\frac{x-a}{p^h}\right)^k$ so
  that $\frac{g^{(j)}(a)}{j!} = \frac{\tilde{a}_j(g,a)}{p^{hj}}$, hence  we have (see Remark~\ref{alt def val})
  \begin{eqnarray}
    \label{val der}
      v_p(\frac{g^{(j)}(x)}{j!}) \geq v_{\mathrm{LA}_h}(g) - jh
  \end{eqnarray}
  for all $j \geq 0$ and all $x \in \Zp$.
  On the other hand
  $$\varepsilon_{g,r}(x,y) = 
                          \begin{cases}
                             \sum_{j > r}  \frac{g^{(j)}(x)}{j!} y^j    & \text{if } v_p(y) \geq h \quad (\text{since $g$ is analytic on } x+p^h \zp) \\
                             g(x+y) - \sum_{j=0}^{\lfloor r \rfloor} \frac{g^{(j)}(x)}{j!} y^j & \text{if } v_p(y) < h \quad (\text{by definition}). 
                           \end{cases}
  $$
  So if $v_p(y) \geq h$, then by \eqref{val der}
  \begin{eqnarray}
    \label{vareps estimate}
     v_p(\varepsilon_{g,r}(x,y)) \geq  v_{\mathrm{LA}_h}(g) - jh + jv_p(y) 
                 \geq  v_{\mathrm{LA}_h}(g)  + (\lfloor r \rfloor + 1)(v_p(y) -h)
  \end{eqnarray}
  since $j \geq \lfloor r \rfloor + 1$, so that for all $h' \in {\mathbb N}_{\geq 0}$ and for $v_p(y) \geq h'$, we have
  $$v_p(\varepsilon_{g,r}(x,y)) - rh'  \geq v_p(\varepsilon_{g,r}(x,y)) - rv_p(y) 
                 \geq  v_{\mathrm{LA}_h}(g)  + (\lfloor r \rfloor + 1 - r)v_p(y) -h (\lfloor r \rfloor + 1).$$
  Letting $v_p(y) \rightarrow \infty$, the second term on the right tends to $\infty$ since $\lfloor r \rfloor + 1 - r > 0$
  (the last term is constant since $h$ is fixed). Thus the infimum over $y \in p^{h'} \zp$ of the term on the
  left is $\infty$. Letting $h' \rightarrow \infty$, we see $g \in \sC^r(\zp, E)$.
  
  We now show that $v'_{\sC^r}(g) \geq v_{\mathrm{LA}_h}(g) - hr$.
  The first term(s) in \eqref{val C^r prime}
  are bounded below by $v_{\mathrm{LA}_h}(g) - hr$ by \eqref{val der} since $-j \geq -r$. We estimate
  the second term in \eqref{val C^r prime}. If $v_p(y) \geq h$, then by \eqref{vareps estimate}, we have
  $$v_p(\varepsilon_{g,r}(x,y)) -r v_p(y) 
                 \geq  v_{\mathrm{LA}_h}(g)  + (\lfloor r \rfloor + 1)(v_p(y) -h) - rv_p(y) \geq v_{\mathrm{LA}_h}(g)  -hr$$
  since $\lfloor r \rfloor + 1 > r$. If $v_p(y) < h$, then adding and subtracting $rh$ and noting that 
  by the second formula for $\varepsilon_{g,r}(x,y)$ above that only $g(x+y)$ and $\frac{g^{(j)}(x)}{j!} y^j$ appear so
  that we may use \eqref{val der} above, we have
  $$v_p(\varepsilon_{g,r}(x,y)) -r v_p(y)  \geq v_{\mathrm{LA}_h}(g) - jh - r v_p(y) + jv_p(y) -rh+rh 
       =  v_{\mathrm{LA}_h}(g) + (h-v_p(y))(r-j) - rh$$
  which is again at least  $v_{\mathrm{LA}_h}(g) -rh$ since the second  term on the extreme right is non-negative.
  taking the infimum over all $x,y \in \Zp$, we see that the second term in \eqref{val C^r prime} is also bounded below by
  $v_{\mathrm{LA}_h}(g) -rh$, as desired.
\end{proof}

\subsubsection{Amplitude coefficients of locally polynomial functions}

Let $I \subset {\mathbb N}_{\geq 0}$ .
A function $g : \zp \rightarrow E$ is said to be {\it locally polynomial} (with degrees of the monomials in $I$) if there is 
an $h \in  {\mathbb N}_{\geq 0}$ such that $g |_{a+p^h\zp}(x) = \sum_{i\in I} a_i(g,a) x^i$ for all $a \in \Zp$ and some 
coefficients $a_i(g,a) \in E$.
Set 
\begin{eqnarray*}
  \mathrm{LP}^I_h(\zp,E) & = & \{ g : \zp \rightarrow E \mid g \text{ is locally polynomial (for $h$)} \} \\
  \mathrm{LP}^I(\zp,E) & = & \bigcup_{h \geq 0} \mathrm{LP}^I_h(\zp,E).
\end{eqnarray*}
We extend the definition to $I \subset {\mathbb R}$, by considering only indices $i \in I \cap {\mathbb N}_{\geq 0}$.

\begin{Proposition}
  Let $r \geq 0$. 
  \begin{enumerate}
  \item [i)] The ${\mathbbm 1}_{i + p^{l(i)} \zp}(x) \cdot  \left( \frac{x - i}{p^{l(i)}} \right)^k $
                for $0 \leq i \leq p^{h}-1$ and $0\leq k \leq \lfloor r \rfloor$ form a basis of  $\mathrm{LP}^{[0,r]}_h(\zp,E)$.
  \item [ii)] The ${\mathbbm 1}_{i + p^{l(i)} \zp}(x) \cdot  \left( \frac{x - i}{p^{l(i)}} \right)^k $
                for $i \in {\mathbb N}_{\geq 0}$ and $0\leq k \leq \lfloor r \rfloor$ form a basis of  $\mathrm{LP}^{[0,r]}(\zp,E)$.
  \end{enumerate}
\end{Proposition}

\begin{proof}
  The proof is similar to the proof of Proposition~\ref{LC basis}.
\end{proof}

We define for $i,k \in {\mathbb N}_{\geq 0}$ the following rescaled locally polynomial functions
$$ e_{i,k,r} := p^{\lfloor l(i) r \rfloor} \cdot {\mathbbm 1}_{i + p^{l(i)} \zp}(x) \cdot  \left( \frac{x - i}{p^{l(i)}} \right)^k.$$
The following lemma is immediate.

\begin{lemma}
  Let $r \geq 0$. 
  \begin{enumerate}
  \item [i)] The $e_{i,k,r}$
                for $0 \leq i \leq p^{h}-1$ and $0\leq k \leq \lfloor r \rfloor$ form a basis of  $\mathrm{LP}^{[0,r]}_h(\zp,E)$.
  \item [ii)] The $e_{i,k,r}$ for $i \in {\mathbb N}_{\geq 0}$ and $0\leq k \leq \lfloor r \rfloor$ form a basis of  $\mathrm{LP}^{[0,r]}(\zp,E)$.
  \end{enumerate}
\end{lemma}

We denote by $b_{i,k}(g)$, the {\it amplitude coefficients} of $g$, the coefficients
of $g \in \mathrm{LP}^{[0,r]}(\zp,E)$ in the basis $e_{i,k,r}$
for $i \in {\mathbb N}_{\geq 0}$ and $0 \leq k \leq \lfloor r \rfloor$.

\subsubsection{Decomposition of $\sC^r$ functions in wavelets}

The following result allows us to approximate functions of type $\sC^r$ by locally polynomial functions.
It is the starting point of many of the computations in \cite{CG23}.

\begin{Proposition}
  \label{prop loc poly}
  Let $r \geq 0$. If $g \in \sC^r(\zp,E)$, set for $h \in {\mathbb N}_{\geq 0}$
  $$\tilde{g}_h(x) := \sum_{i=0}^{p^h-1} {\mathbbm 1}_{i+p^h\zp}(x) \cdot \sum_{k=0}^{\lfloor r \rfloor} \frac{g^{(k)}(i)}{k!} (x-i)^k.$$
  Then
  \begin{enumerate}
    \item [i)]  $v_{\mathrm{LA}_{h+1}}(\tilde{g}_{h+1} - \tilde{g}_h) \geq rh + C_{g,r}(h) - C(\lfloor r \rfloor)$
    \item [ii)] $\tilde{g}_h \rightarrow g$ in $\sC^r(\zp,E)$  as $h \rightarrow \infty$.
  \end{enumerate}
\end{Proposition}

\begin{proof}
We prove i). Write
\begin{eqnarray}
  \label{diff of fns}
  \tilde{g}_{h+1} - \tilde{g}_h  = \sum_{i=0}^{p^h-1} \sum_{a = 0}^{p-1}  {\mathbbm 1}_{i+ap^h + p^{h+1}\zp}(x) \cdot 
  \left( \sum_{k=0}^{\lfloor r \rfloor}  \frac{g^{(k)}(i+ap^h)}{k!} (x-i-ap^h)^k 
     - \sum_{k=0}^{\lfloor r \rfloor}  \frac{g^{(k)}(i)}{k!} (x-i)^k \right).
\end{eqnarray}
Now subtracting and adding $ap^h$ we may write the last sum as
\begin{eqnarray*}
   \sum_{k=0}^{\lfloor r \rfloor} \sum_{j=0}^k  \frac{g^{(k)}(i)}{j!(k-j)!} (x-i-ap^h)^j (ap^h)^{k-j} & = & \sum_{j=0}^{\lfloor r \rfloor} \sum_{k=j}^{\lfloor r \rfloor}  \frac{g^{(k)}(i)}{j!(k-j)!} (x-i-ap^h)^j (ap^h)^{k-j} \\
  & = &  \sum_{j=0}^{\lfloor r \rfloor} \sum_{l=0}^{\lfloor r \rfloor -j}  \frac{g^{(l+j)}(i)}{j! l!} (x-i-ap^h)^j (ap^h)^{l}
\end{eqnarray*}
by substituting $l= k-j$. Thus, substituting for this sum and summing over $j$ instead of $k$ in the other inner sum in
\eqref{diff of fns}, 
we may rewrite the inner term on the RHS 
of \eqref{diff of fns} as
$$ \sum_{j=0}^{\lfloor r \rfloor} \frac{p^{j(h+1)}}{j!}  \left( g^{(j)}(i+ap^h) - \sum_{l=0}^{\lfloor r \rfloor -j}  \frac{g^{(l+j)}(i)}{l!} (ap^h)^{l} \right)  \left(\frac{x-i-ap^h}{p^{h+1}} \right)^j.$$
By \eqref{vareps estimate again}, the term in the large parentheses has $v_p$ at least $(r-j)h + C_{g,r}(h) 
           - C(\lfloor r \rfloor)$.
Thus by \eqref{diff of fns} and Remark~\ref{alt def val},  we have
\begin{eqnarray*}
  v_{\mathrm{LA}_{h+1}}(\tilde{g}_{h+1} - \tilde{g}_h) & \geq & \inf_{j \leq r} \left( v_p( \frac{p^{j(h+1)}}{j!}) +  (r-j)h + C_{g,r}(h) 
           - C(\lfloor r \rfloor)  \right) \\
          & = &  \inf_{j \leq r} (j - v_p(j!) + rh + C_{g,r}(h)    - C(\lfloor r \rfloor)) = rh + C_{g,r}(h)    - C(\lfloor r \rfloor)
\end{eqnarray*}
since $j-v_p(j!) \geq 0$, so the infimum is attained at the $j = 0$ term. This proves i).

We now turn to ii). By Proposition~\ref{LA vs Cr}, we have
$$v'_{\sC^r}(\tilde{g}_{h+1} - \tilde{g}_h) \geq v_{\mathrm{LA}_{h+1}}(\tilde{g}_{h+1} - \tilde{g}_h) -r(h+1) \geq C_{g,r}(h)    - C(\lfloor r \rfloor)) - r$$
by i). Since the RHS goes to $\infty$ as $h \rightarrow \infty$, we see that $\{\tilde{g}_h\}$ is a Cauchy sequence.
But $\sC^r(\zp,E)$ is complete (it's an $E$-Banach!) so there is a function $\tilde{g} \in \sC^r(\zp,E)$ such
that $\tilde{g}_h \rightarrow \tilde{g}$. But 
$$\tilde{g}_h(i) = g^{(0)}(i) = g(i) \quad \text{for all } 0\leq i \leq p^h-1.$$ 
So $\tilde{g}(i) = g(i)$ for all such $i$. Letting $h \rightarrow \infty$ gives
$\tilde{g} = g$ on ${\mathbb N}_{\geq 0}$. By continuity, we get $\tilde{g} = g$, proving ii).
\end{proof}

\begin{theorem}
  \label{theorem bb}
  The family $e_{i,k,r}$ for $i \in {\mathbb N}_{\geq 0}$ and $0 \leq k \leq \lfloor r \rfloor$ is a Banach basis of 
  $\sC^r(\zp,E)$
\end{theorem}

\begin{proof}
  This can be proved using the above proposition (and some auxiliary results). 
  For the details, see \cite[Theorem I.5.14]{Col10a}.
\end{proof}

The basis $e_{i,k,r}$ is called a {\it basis of wavelets} of $\sC^r(\zp,E)$. The coefficients $b_{i,k}(g)$ of
$g \in \sC^r(\zp,E)$ in this basis are called the {\it amplitude coefficients} of $g$.

\subsubsection{Mahler coefficients of functions of class $\sC^r$}

Finally, we state an important theorem which allows one to show that the topologies defined
by  two different valuations on $\sC^r(\Zp,E)$ are the same.
The proof is a bit long so we omit the proof.

\begin{theorem}
  \label{theorem comparing two val C^r}
  Let $r \geq 0$ and let $g \in \sC^r(\zp,E)$.
  If $g(x) = \sum_{n=0}^\infty a_n(g) {x \choose n}$ is the Mahler expansion of $g$, then 
  $v_p(a_n(g)) - r l(n) \rightarrow \infty$
  as $n \rightarrow \infty$.
  If $$v_{\sC^r}(g) = \inf_{n \geq 0} (v_p(a_n(g)) - r l(n))$$
  then $v_{\sC^r} \sim v'_{\sC^r}$ differ at most by a finite constant.
\end{theorem}

\begin{proof}
    See the proof of \cite[Theorem I.5.17]{Col10a}.
 \end{proof}

\begin{Corollary} 
   The $p^{ \lfloor r l(n) \rfloor} {x \choose n}$ 
   form a Banach basis of $\sC^r(\zp,E)$.
\end{Corollary}

\subsubsection{Poly$\cdot$log functions are $\sC^r$}

We now show that certain (finite sums of) polynomial times logarithmic  (poly$\cdot$log) functions are $\sC^r$. 
These functions are of key importance in the entire argument.

Let $\sL \in E$. We define the corresponding branch of the logarithmic function
\begin{eqnarray*}
  \logL : \Qp^* & \longrightarrow & E \\
               \zeta & \mapsto & 0 \quad \text{for } \zeta \in \mu_{p-1}\\
               1+px & \mapsto & \log(1+px) = px - \frac{(px)^2}{2!} + \frac{(px)^3}{3!} - \cdots \quad \text{for } x \in \zp \\
                 p     & \mapsto & \sL
\end{eqnarray*}
and such that $\logL(xy) = \logL(x) + \logL(y)$ for all $x,y \in \qp^*$.
Most number theorists leave this planet encountering only the Iwasawa $p$-adic logarithmic function 
which is the case $\sL = 0$. But the case of general $\sL \in E$ will be key to all that follows.

We make several elementary but enlightening remarks about the continuity and differentiability of this
function and it's extensions to $\qp$.
First note that $\logL$ is continuous and even infinitely differentiable on $\qp^*$. Indeed, on the open 
set $\{z \in \qp^*  \mid v_p(z) = i \}$ for $i \in {\mathbb Z}$, writing $z = p^i \cdot \frac{z}{p^i}$ we see that 
$\logL(z) = i\sL+ \log(\frac{z}{p^i})$ and $\log$ is locally analytic on $\Zp^*$.



    However,  $\logL$ does not extend to a continuous function on $\Qp$. Indeed, the limits of 
    two different sequences tending to $0$ can be different. For instance, $\logL(p^{1 + p^n}) = (1+p^n) \sL$ and $\logL(p^{p^n}) = p^n \sL$ tend to $\cL$ and $0$, respectively, which may be distinct.

    However, the function $z^n\logL(z)$ extends to a continuous function on $\Qp$ for any integer $n \geq 1$. Indeed, for any $z \in \Qp^*$, the valuation of $\logL(z)$ is bounded below by $\min\{v_p(\cL), 1\}$. Therefore the limit of $z^n\logL(z)$ as $z \to 0$ is equal to $0$. So $z^n\logL(z)$ can be extended to a continuous function on $\Qp$ if we set its value at $0$ to be $0$.

    We now discuss the differentiability of $z^n\logL(z)$ for $n \geq 1 $. This function is clearly differentiable on $\qp^*$.
    Indeed for $n = 1$, $z\logL(z)$ has derivative $\logL(z) + 1$ on $\qp^*$. Similarly the function $z^2\logL(z)$ is        
    differentiable on $\Qp^*$ with first derivative given by $2z\logL(z) + z$ and second derivative
    given by $2 \logL(z) + 3$. In fact in general  it is not hard to show that $g(z) = z^n\logL(z)$ 
    is differentiable on $\qp^*$ with $j$-th derivative for $0 \leq j \leq n$ given by
    \begin{eqnarray}
      \label{der formula}
      g^{(j)}(z) = \frac{n!}{(n-j)!} z^{n-j} \logL(z) + t_j z^{n-j}
    \end{eqnarray}
    for some $t_j \in \zp$. In fact, there is a more precise formula
    for the derivatives of the function $z^n\logL(z)$  in terms of certain partial harmonic sums
    $H_i = 1 + \frac{1}{2} + \cdots + \frac{1}{i}$. Indeed, for $z \neq 0$, we have
    \begin{eqnarray}\label{Small derivative formula for polynomial times logs}
        \frac{d^j}{dz^j}(z^n\logL(z)) = \frac{n!}{(n - j)!}\left[z^{n - j}\logL(z) + (H_n - H_{n - j})z^{n - j}\right], 
    \end{eqnarray}
    for $0 \leq j \leq n$. We remind the reader of the convention that $H_0 = 0$.
    
   We now check the differentiability of $z^n\logL(z)$ for $n \geq 1 $ at $z = 0$. By definition, the derivative at $z = 0$
   is given by the limit
    \[
        \lim_{z \to 0}\frac{z^n\logL(z) - 0}{z}.
    \]
    We have seen above that this limit does not exist for $n = 1$. Therefore $z\logL(z)$ is only differentiable for $z \neq 0$. For $n = 2$ however, one sees that the limit exists. Therefore, the function $z^2\logL(z)$ is differentiable everywhere on $\qp$ with derivative given by $2z\logL(z) + z$. Since the derivative involves $z\logL(z)$, we see that $z^2\logL(z)$ is differentiable everywhere only once. Extending
    this, we see that the function $z^n\logL(z)$ is differentiable everywhere only $n - 1$ times. Moreover, its $(n - 1)^{\mathrm{th}}$ derivative is continuous on $\Qp^*$.


As has been pointed out in Remark~\ref{diff but not C^r}, while $\sC^r$ functions are differentiable of 
order $\lfloor r \rfloor$, the converse is not necessarily true. However, we prove that the converse is indeed true for 
poly$\cdot$log functions. We have

\begin{Proposition}
  \label{prop poly log}
For $1 \leq n \leq p-1$, the function $z^{n}\logL(z)$ belongs to $\sC^s(\Zp, E)$
    for any $0 \leq s < n$. 
\end{Proposition}

\begin{proof}
    The argument is taken from \cite{CG23}. Let $g(x) = x^{n}\logL(x)$.
    We need to check that for
    \[
        \varepsilon_{g,s}(x,y) = g(x + y) - \sum_{j = 0}^{\lfloor s \rfloor}g^{(j)}(x)\frac{y^j}{j!},
    \]
    we have
    \[
            \inf_{x \in \Zp \atop y \in p^h\Zp} (v_p(\varepsilon_{g,s}(x,y)) - sh) \to \infty.
    \]
    as $h \to \infty$.  
    Fix $h$. Let $x \in \Zp$. There are two cases to consider:
    \begin{itemize}
    \item First assume that $h > v_p(x)$. 
    Therefore $x \neq 0$ and $h \geq 1$. For $y \in p^h\Zp$, we have a
    Taylor expansion for $\logL(1 + y/x)$ which is analytic. So for such $x$ and $y$
    \begin{eqnarray*}
        \varepsilon_{g,s}(x,y) & = & \sum_{j = \lfloor s \rfloor + 1}^{\infty} g^{(j)}(x)\frac{y^j}{j!}.
    \end{eqnarray*}
    Say $j \leq n$. Then by \eqref{der formula}, we have
    $g^{(j)}(x) = \frac{n!}{(n - j)!}x^{n - j}\logL(x) + t_jx^{n - j}$ for some $t_j \in \Zp$. We have seen  that the valuation of $\logL(x)$ is bounded below by $\min\{v_p(\cL), 1\} \geq \min\{v_p(\cL), 0\}$. Therefore the valuation of the $j^{\mathrm{th}}$ summand above is bounded below by $\min\{v_p(\cL), 0\} + jh - v_p(j!)$. Using the well-known formula
    \[
        v_p(j!) = \frac{j - \sigma_p(j)}{p - 1} \leq \frac{j}{p-1},
    \]
    where $\sigma_p(j)$ is the sum of the $p$-adic digits of $j$, we see that
    \[
        v_p\left(g^{(j)}(x)\frac{y^j}{j!}\right) 
        \geq \min\{v_p(\cL), 0\} + j\left(h - \frac{1}{p - 1}\right)   \geq \min\{v_p(\cL), 0\} + (\lfloor s \rfloor + 1) \left(h - \frac{1}{p - 1}\right).
    \]
    Next say $j > n$. Then since $x \neq 0$, we have
    \[
        g^{(j)}(x) = n!\frac{(-1)^{j - n - 1}(j - n - 1)!}{x^{j - n}}.
    \]
    So similarly
    \begin{eqnarray*}
        v_p\left(g^{(j)}(x)\frac{y^j}{j!}\right) & \geq & -(j - n)v_p(x) + jh - \frac{j}{p - 1}  \\
        & = & n\left(h - \frac{1}{p - 1}\right) + (j - n)\left(h - \frac{1}{p - 1} - v_p(x)\right). \\
        & \geq & n\left(h - \frac{1}{p - 1}\right) \> \geq  \> (\lfloor s \rfloor + 1)\left(h - \frac{1}{p - 1}\right)
    \end{eqnarray*}
    since $h - v_p(x) \geq 1$.
    Putting these cases together, we see that the valuation of the $j^{\mathrm{th}}$ term in
    $\varepsilon_{g,s}(x,y)$ is greater than or equal to
    \[
        \min\{v_p(\cL), 0\} + (\lfloor s \rfloor + 1)\left(h - \frac{1}{p - 1}\right).
    \]
    Hence 
     \[
        v_p(\varepsilon_{g,s}(x,y)) \geq \min\{v_p(\cL), 0\} + (\lfloor s \rfloor + 1)\left(h - \frac{1}{p - 1}\right).
    \]
    \item Now assume that $h \leq v_p(x)$. Say  $y \in p^h \zp$. Then $v_p(x+y) \geq h$.
    Now $g(x+y) = (x+y)^n \logL(x+y)$ implies $$v_p(g(x+y)) \geq nh + \min\{v_p(\cL), 1\} \geq nh + \min\{v_p(\cL), 0\}.$$
    Now using \eqref{der formula}
    we have for all $0 \leq j \leq \lfloor s \rfloor$ 
    $$v_p(\frac{g^{(j)}(x) y^j}{j!}) \geq (n-j)h + \min\{v_p(\cL), 0\} + jh -v_p(j!) 
              \geq nh + \min\{v_p(\cL), 0\} - v_p(\lfloor s \rfloor !).$$
    Since $n \geq \lfloor s \rfloor + 1$, it follows that
    \[
        v_p(\varepsilon_{g,s}(x,y)) \geq \min\{v_p(\cL), 0\} + (\lfloor s \rfloor + 1) h - v_p(\lfloor s \rfloor!).
    \]
\end{itemize}
    Both of these estimates then show that
    \[
        \inf_{x \in \Zp, y \in p^h\Zp} (v_p(\varepsilon_{g,s}(x,y)) - sh) \to \infty
    \]
    as $h \to \infty$, since $\lfloor s \rfloor + 1 - s > 0$.
    Therefore, the function $z^n\logL(z)$ belongs to $\sC^s(\Zp, E)$
    for $0 \leq s < n$.
\end{proof}

\begin{Corollary}
    \label{cor poly log}
    $E$-valued functions of $z \in \qp$
    of the form
     \[
       g(z) = 
                 (z - z_i)^{n_i}\logL(z - z_i),
     \]
    where $z_i \in \zp$ and $n_i \geq 1$ are $\sC^s$
    for $0 \leq s < n_i$. 
\end{Corollary}

\begin{proof}
  By Proposition~\ref{prop comp}, the composition of $\sC^s$ functions is $\sC^s$. Therefore 
  we are done by Proposition~\ref{prop poly log}. 
\end{proof}

\section{${\mathrm {GL}}_2(\qp)$-Banach spaces} 

\subsection{The Banach space $B_{k,\sL}$} 
\label{Definition of the Banach space}
In \cite[Section 4.2]{Bre04}, Breuil defined the Local Langlands correspondent $B(k, \cL)$ of the semi-stable representation $V_{k, \cL}$ for $k \geq 3$. Let $r = k-2$. In this section, we recall the alternative definition $\tB(k, \cL)$ of $B(k, \cL)$ 
given in \cite[Corollary 3.3.4]{Bre10} which uses the notion of functions of type $\sC^s$ for $s = \frac{r}{2}$.

By Proposition~\ref{theorem comparing two val C^r}, for a real number $s \geq 0$, a continuous function $g : \Zp \to E$ belongs to the space $\C{s}(\Zp, E)$ if in its Mahler's expansion $g(z) = \sum_{n = 0}^{\infty}a_n(g){z \choose n}$, the coefficients $a_n(g)$ satisfy $n^s\norm{a_n(g)} \to 0$ as $n \to \infty$. The space $\C{s}(\Zp, E)$ is a Banach space with the norm $\Bnorm{g}{s} = \sup_n (n + 1)^s\norm{a_n(g)}$.

    Let $D(k)$ be the $E$-vector space of functions $g : \Qp \to E$ such that 
    \begin{itemize}
    \item $g_1 : z \mapsto g(z)$ for $z \in \Zp$ belongs to $\C{\frac{r}{2}}(\Zp, E)$, and 
    \item $g_2 : z \mapsto z^{r}g(1/z)$ for $z \in {\Zp \setminus \{0\}}$ extends to a function on $\Zp$ belonging 
             to $\C{\frac{r}{2}}(\Zp, E)$.
    \end{itemize}
This space is a Banach space
under the norm to be $$\Bnorm{g}{} = \max(\Bnorm{g_1}{\frac{r}{2}}, \Bnorm{g_2}{\frac{r}{2}}).$$ We define an 
action of $G = {\mathrm {GL}}_2(\qp)$ on $D(k)$  by:
\begin{eqnarray}
    \label{G-action}
  \left(\begin{pmatrix}a & b \\ c & d  \end{pmatrix} \cdot g\right)(z) = \norm{ad - bc}^{\frac{r}{2}}(bz + d)^{r}g\left(\frac{az + c}{bz + d}\right).
\end{eqnarray}

Let $P(k)$ be the space of polynomials of degree less than or equal to $r = k - 2$. Note that $P(k) \subset D(k)$. Indeed, if $g$ is a polynomial of degree less than or equal to $r$, then both $g_1$ and $g_2$ are polynomials and therefore belong to $\C{\frac{r}{2}}(\Zp, E)$.
Moreover, the space of such polynomials is clearly stable under the $G$-action defined above.

    Set $\tB(k)$ to be the quotient of $D(k)$ by $P(k)$.

Now we define $\tB(k, \cL)$ as a quotient of $\tB(k)$.
Let $L(k, \cL)$ be the subspace of $D(k)$ generated by $P(k)$ and finite
sums of poly$\cdot$log functions of the form 
\begin{eqnarray}
    \label{polylog fn}
    g(z) = \sum_{i \in I} \lambda_i(z - z_i)^{n_i} \logL(z - z_i),
\end{eqnarray}
where 
\begin{itemize}
  \item $I$ is a finite (indexing) set 
  \item $\lambda_i \in E$
  \item $z_i \in \Qp$ 
  \item $n_i \in \{\lfloor\frac{r}{2}\rfloor + 1, \ldots, r\}$ and 
\end{itemize}  
\begin{eqnarray}
  \label{degree cond}
  \deg \left( \sum_{i \in I} \lambda_i (z - z_i)^{n_i} \right) < \frac{r}{2}.
\end{eqnarray}

The subspace $L(k,\sL)$ is $G$-stable (see \cite[Lemma 3.3.2]{Bre10}).
We make some comments on why $L(k,\sL) \subset D(k)$.
Assume that the $z_i \in \Zp$.
We already saw in Corollary~\ref{cor poly log} that the function $g(z)$ in \eqref{polylog fn} -  being a (a finite sum of scalar multiples of) poly$\cdot$log function(s) - is 
$\sC^{\frac{r}{2}}$ because each $n_i > r/2$. What is less clear is that $z^r g(1/z)$ is also $\sC^{\frac{r}{2}}$.
The somewhat mysterious degree condition \eqref{degree cond} is made precisely to ensure this. To see, let us assume that all the
$n_i$ are equal to some $n$ ($> \frac{r}{2}$) for simplicity. Then 
\begin{eqnarray*}
  z^r g(1/z) & = & \sum_{i \in I} \lambda_i z^{r-n} (1-zz_i)^n \logL(1-zz_i) - \sum_{i \in I} \lambda_i z^{r-n} (1-zz_i)^n \logL(z).
\end{eqnarray*}
Now the first term on the RHS is $\sC^{\frac{r}{2}}$ by Corollary~\ref{cor poly log} and Proposition~\ref{prop comp}
($z \mapsto 1-zz_i$ is $\sC^{\frac{r}{2}}$) and multiplication by the monomial $z^{r-n}$ preserves the property
of being $\sC^{\frac{r}{2}}$. Write the second term as
$$\sum_{k=0}^n {n \choose k} (-1)^k \left( \sum_{i \in I} \lambda_i z_i^k \right) z^{r-n+k} \logL(z).$$
Now an easy check using \eqref{degree cond} shows that $\sum_{i \in I} \lambda_i z_i^k = 0$ for $k \leq n - r/2$.
This forces the only powers $r-n+k$ of $z$ to survive are those which are strictly bigger than $r/2$. So again
Corollary~\ref{cor poly log} applies to show that the second term on the RHS is $\sC^{\frac{r}{2}}$!

  Define  $\tB(k, \cL)$ to be the quotient of $D(k)$ by the closure of $L(k,\cL)$ in $D(k)$.

It turns out that, $\tB(k, \cL)$ is isomorphic to $B(k, \cL)$, the Local Langlands correspondent of $V_{k, \cL}$
(cf. \cite[Corollary 3.3.4]{Bre10}). 

For convenience, we introduce a third notation and set 
$$B_{k,\sL} := \tB(k, \cL).$$
Thus $B_{k,\sL}$ is a unitary $G$-Banach space which corresponds to $V_{k,\sL}$ under the $p$-adic
Local Langlands correspondence. 

\subsection{Uniformizing $B_{k,\sL}$}

Recall that 
$$Z = \qp^* \subset B = \{ \left( \begin{smallmatrix} a & b \\ 0 & d \end{smallmatrix} \right) \} 
   \subset G = {\mathrm{GL}}_2(\qp) \supset K =  {\mathrm{GL}}_2(\Zp)  \supset I = 
\{ \left( \begin{smallmatrix} a & b \\ c & d \end{smallmatrix} \right) \mid c \equiv 0 \mod p \}$$ where 
$I$ is the Iwahori subgroup.

To compute the reduction of $B_{k,\sL}$ we need to define an integral structure on $B_{k,\sL}$. 
It is not clear how to proceed. However, if one is able to uniformize (a dense subspace of) $B_{k,\sL}$  by a 
compactly induced space of rational valued polynomial functions on $G$,
then one can define a lattice in $B_{k,\sL}$ by taking the (closure of the) image of the integral valued
polynomial functions on $G$ (much as the standard lattice is defined  in the corresponding
Banach space attached to 
a crystalline representation). 

Recall that $\SymE{k-2}$ is the $(k-2)$-th symmetric power representation of $KZ$ on $E^2$
(which is modelled on homogeneous polynomials of degree $k-2$ in two variables $X$ and $Y$ over $E$)
twisted by the character $|\det|^{\frac{k-2}{2}}$ so that $p \in Z$ acts trivially.

We prove (see \cite[Section 5]{CG23} for the details):

\begin{Proposition}
\label{uniformizing}
There are maps:
\begin{eqnarray*}
  \mathrm{ind}_{IZ}^{G} \SymE{k-2} \twoheadrightarrow \SymE{k-2} \otimes  \mathrm{ind}_{IZ}^G {\mathbbm 1}_E
        \rightarrow \SymE{k-2}  \otimes  (\Bind E)^{\sm}     
        \hookrightarrow D(k). 
\end{eqnarray*}
Under the composition of all these maps $\llbracket 1,X^iY^{r-i} \rrbracket \mapsto z^i {\mathbbm 1}_{p\zp}$.
\end{Proposition}

\begin{proof}
We recall the definitions of each of the maps in order:
\begin{enumerate}
  \item $\llbracket g,P \rrbracket \mapsto gP \otimes \llbracket g,1 \rrbracket$.
  \item $P \otimes \llbracket \id, 1 \rrbracket \mapsto P \otimes f_{\id}$ where
             $f_{\id} \in (\Bind E)^{\sm}$ is defined by
    \[
        f_{\id}(g) = \begin{cases}
            1 & \text{ if $g \in BIZ$} \\
            0 & \text{ otherwise.}
        \end{cases}
    \]
   One checks that $f_{\id} \left( \begin{smallmatrix} 1 & 0 \\ z & 1 \end{smallmatrix} \right) = {\mathbbm 1}_{p\zp}$.
  \item $P(X,Y) \otimes f \mapsto P(z,1) f \left( \begin{smallmatrix} 1 & 0 \\ z & 1 \end{smallmatrix} \right)$.
\end{enumerate}
An easy check shows that under these maps
\begin{eqnarray*}
  \llbracket \id, X^iY^{r-i} \rrbracket \mapsto X^iY^{r-i} \otimes \llbracket \id,1 \rrbracket 
          \mapsto  X^iY^{r-i} \otimes f_{\id} \mapsto z^i {\mathbbm 1}_{p\zp}
\end{eqnarray*}         
as claimed. 
\end{proof}

We make two remarks
\begin{itemize}
  \item The map 2. above induces an isomorphism
           $$\dfrac{\mathrm{ind}_{IZ}^G {\mathbbm 1}_E}{(T_{1,0} +1) + (T_{1,2} - 1)} \simeq  \dfrac{(\Bind E)^{\sm}}{{\mathbbm 1}_E} = \mathrm{St}.$$ 
  \item The map 3. factors as 
           $$\dfrac{ \SymE{k-2} \otimes (\Bind E)^{\sm}}{\SymE{k-2} \otimes {\mathbbm 1}_E \qquad} = \SymE{k-2} \otimes \mathrm{St} \hookrightarrow \dfrac{D(k)}{P(k)} = \tilde{B}(k) \twoheadrightarrow B_{k,\sL}.$$
           The image of this map consists of  the {\it locally algebraic vectors} in $B_{k,\sL}$ of which it is the completion. 
\end{itemize}
The second remark and Proposition~\ref{uniformizing} together say we have a uniformization (the image is dense)
\begin{eqnarray}
 \label{Uniformizing}
 \mathrm{ind}_{IZ}^G \SymE{k-2} \rightarrow  B_{k,\sL}.
\end{eqnarray}
           
\subsection{The lattice 
  $\latticeL{k}$}
    \label{Section defining the lattices}

We can now define the standard lattice $\latticeL{k}$
in the Banach space $\tB(k,\cL)$ introduced above.
We let
$$\Theta_{k,\sL} := \tilde{\Theta}(k,\sL) = \text{(closure of the) image} \left( \mathrm{ind}_{IZ}^{G} \SymO{k-2} \rightarrow B_{k,\sL} \right)$$ 
under the map in \eqref{Uniformizing}.
%
%
We now describe
some explicit elements of this lattice.
In this context, we are reminded of a warning by Prof. Murray Schacher who while teaching an algebraic number theory course at UCLA in 1993 mentioned that  graduate students experience `considerable giddiness'
when trying to write down integral elements in a number field. A similar warning applies
when trying to identify explicit elements in $\latticeL{k}$.
We now mention two important lemmas which help identify some (integral) elements in this lattice. 
        
    
    \begin{lemma}\label{Integers in the lattice}
        Let $r = k-2 \geq 1$, $0 \leq j \leq r$, $h \in \bZ$ and $z_0 \in \Qp$. Then the function
        \[
            p^{(h - 1)\left(r/2 - j \right)}(z - z_0)^j\mathbbm{1}_{z_0 + p^h\Zp}\in \latticeL{k}
        \]
    is integral.
    \end{lemma}
    
    \begin{proof}
        By the $G$-action defined in \eqref{G-action}
        \[
          \begin{pmatrix} 1 & 0 \\ -z_0 & p^{h - 1}\end{pmatrix} \cdot z^j\mathbbm{1}_{p\Zp} =
          p^{(h - 1)\left(r/2 - j \right)}(z - z_0)^j\mathbbm{1}_{z_0 + p^h\Zp}.
        \]
        The lemma follows since $z^j\mathbbm{1}_{p\Zp} \in \latticeL{k}$ and $\latticeL{k}$ is $G$-stable.
    \end{proof}
    

    The following stronger result applies if $j \geq r/2$. Note that $(h-1)(r/2-j) \geq h(r/2-j)$.
    
    \begin{lemma}\label{stronger bound for polynomials of large degree with varying radii}
        Let $r = k-2 \geq 1$, $r/2 \leq j \leq r$, $h \in \bZ$ and $z_0 \in \Qp$. Then the function
        \[
            p^{h\left(r/2 - j\right)}(z - z_0)^j\mathbbm{1}_{z_0 + p^h\Zp} \in \latticeL{k}
        \]
       is integral.
    \end{lemma}
    
    \begin{proof}
        Again, by \eqref{G-action},
        \[
          -\begin{pmatrix}0 & 1 \\ p^h & -z_0\end{pmatrix} \cdot z^{r - j}\mathbbm{1}_{p\Zp} =
          -p^{h\left(r/2 - j\right)}(z - z_0)^j\mathbbm{1}_{p\Zp}\left(\frac{p^h}{z - z_0}\right).
        \]
        Now
        \[
          \mathbbm{1}_{p\Zp}\left(\frac{p^h}{z - z_0}\right) = 1 \iff \frac{p^h}{z - z_0} \in p\Zp
          \iff v_p(z - z_0) < h \iff \mathbbm{1}_{\Qp \setminus \left(z_0 + p^h\Zp\right)}(z) = 1.
        \]
        Using the fact that polynomials of degree less than or equal to $r$ are equal to $0$ in $\tB(k,\sL)$, we see that
        \[
            -(z - z_0)^j\mathbbm{1}_{p\Zp}\left(\frac{p^h}{z - z_0}\right) = -(z - z_0)^j\left(1 - \mathbbm{1}_{z_0 + p^h\Zp}(z)\right) = (z - z_0)^j\mathbbm{1}_{z_0 + p^h\Zp}.
        \]
        The lemma again follows since $z^{r - j}\mathbbm{1}_{p\Zp} \in \latticeL{k}$ and $\latticeL{k}$ is $G$-stable.
      \end{proof}

\subsection{Uniformizing the reduced lattice and its subquotients}

 Let $\Fq = \co_E/\pi$ be the residue field of $E$. 
 Let $\SymF{k-2}$ be the symmetric power representation of $\mathrm{GL}_2(\Fq)$  modelled again on all homogeneous
 polynomials of degree $r = k-2$ over $\Fq$ in $X$ and $Y$ and
 thought of as a representation
 of $IZ$  by noting that $I \subset K \twoheadrightarrow 
 \mathrm{GL}_2(\Fq)$ and letting $p \in Z$ act trivially. 
 
 Recall that by definition there is a map
 $\IZind \SymO{k - 2} \to \latticeL{k}$ with dense image.
 Tensoring this map over $\co_E$ with $\Fq$ 
 gives us a surjection 
\begin{eqnarray}\label{Main surjection on mod p representations}
    \IZind \SymF{k - 2} \twoheadrightarrow \br{\latticeL{k}} = \br{\Theta_{k,\sL}}.
\end{eqnarray}

 Now $\SymF{k - 2}$ has an $IZ$-stable filtration 
 $$\langle X^r, X^{r-1}Y, \ldots, XY^{r-1}, Y^r \rangle \supset  \langle X^r, X^{r-1}Y, \ldots, XY^{r-1} \rangle \supset 
      \cdots \supset \langle X^r \rangle$$ 
      where the successive quotients are given by the characters  $d^r, ad^{r-1}, \ldots, a^r$ of $IZ$.
      Thus the module $\IZind \SymF{k - 2}$
      inherits a filtration where the successive subquotients are given by
      $$\IZind d^r, \IZind ad^{r-1}, \ldots, \IZind a^r.$$ 
      Pushing this filtration forward, we obtain a filtration on $\br{\latticeL{k}}$. Define
      $F_{2l,2l+1}$ to be the subquotient of  $\br{\latticeL{k}}$ corresponding to $ \IZind a^l d^{r-l}$ for $0 \leq l \leq r$. 
      Then, there is a surjection
      \[
        \IZind a^ld^{r-l} \twoheadrightarrow F_{2l,\> 2l + 1},
      \]
      for $0 \leq l \leq  r$.
   
     
As mentioned in the introduction, by the Iwahori mod $p$ LLC
      (Theorem~\ref{Iwahori mod p LLC}), to determine
      the reduction $\br{V}_{k,\cL}$, it suffices to determine the reduction
      $\br{\tilde{\Theta}(k, \cL)}$ of $\tilde{\Theta}(k, \cL)$. By what we have just said, it therefore suffices
      to determine the subquotients 
      $$F_{0, 1}, \> F_{2, 3}, \> \ldots, \> F_{2r,2r+1}.$$ 
     Theorem~\ref{Main theorem in the second part of my thesis}
    is proved by a detailed analysis of these subquotients!

\section{Reductions of semi-stable Galois representations}

\subsection{Some results concerning functions in $\sC^{r/2}(\Zp, E)$}
We now describe some further generalities on functions in the space $\sC^{s}(\Zp, E)$ that  are needed in our
computations of the reduction of the lattice $\br{\latticeL{k}}$. Much of the material in this section
is taken verbatim from \cite[Section 8]{CG23}.

Recall that Proposition~\ref{prop loc poly} says that if 
   $g \in \sC^{s}(\Zp, E)$  for some $s \geq 0$, then 
    \begin{eqnarray}
    \label{Convergence proposition}
        \tilde{g}_h(z) := \sum_{m = 0}^{p^h - 1} \left[\sum_{j = 0}^{\lfloor s \rfloor}\frac{g^{(j)}(m)}{j!}(z - m)^j\right]\mathbbm{1}_{m + p^h\Zp}  \to g \text{ as } h \to \infty. 
    \end{eqnarray}

We now modify $\tilde{g}_h$ slightly by fake-adding some terms which vanish mod $\pi \latticeL{k}$ to get a 
new function $g_h$ which is more amenable to further computation. More
precisely, we have the following general lemma.

    \begin{lemma}\label{Congruence lemma}
        Suppose $g \in \sC^{r/2}(\Zp, E)$ is a function that has continuous derivatives of order $\lfloor r/2 \rfloor + 1, \ldots, t$ for an integer $r/2 < t \leq r$. Define
        \[
            g_h(z) := \sum_{m = 0}^{p^h - 1}\left[\sum_{j = 0}^{t}\frac{g^{(j)}(m)}{j!}(z - m)^j\right]\mathbbm{1}_{m + p^h\Zp}.
        \]
        Then for large $h$, we have
        \[
            \tilde{g}_h(z) \equiv g_h(z) \mod \pi \latticeL{k}.
        \]
    \end{lemma}
    \begin{proof}
      The lemma follows by noting
      that for large $h$, $\lfloor r/2 \rfloor + 1 \leq j \leq t$ and $0 \leq m \leq p^h - 1$, we have
        \begin{eqnarray}\label{Excessive derivatives are 0}
            \frac{g^{(j)}(m)}{j!}(z - m)^j \mathbbm{1}_{m + p^h\Zp} \equiv 0 \mod \pi \latticeL{k}.
        \end{eqnarray}
        Indeed, by the continuity of the derivatives we see that the valuation of $\frac{g^{(j)}(m)}{j!}$ is bounded below by some finite rational number $M$ for all $m \in \Zp$ and all $\lfloor r/2 \rfloor + 1 \leq j \leq t$. We choose $h$ large enough so that $M > (h - 1)(r/2 - j)$, noting $r/2-j$ is negative.
        Then  \eqref{Excessive derivatives are 0} follows by Lemma \ref{Integers in the lattice}.
    \end{proof}

\begin{lemma}\label{Truncated Taylor expansion}
    Let $k/2 \leq n \leq r \leq p - 1$ and $z_0 \in \Zp$. Fix $x \in \bQ$ such that $x \geq - 1$ and $x + v_p(\cL) \geq r/2 - n$. 
    Let 
        $$g(z) = p^x(z - z_0)^n\logL(z - z_0)\mathbbm{1}_{\Zp}$$
    Then for $h \geq 3$, $0 \leq a \leq p^{h - 1} - 1$, $0 \leq \alpha \leq p - 1$ and $0 \leq j \leq n - 1$, we have
    \[
        g^{(j)}(a + \alpha p^{h - 1}) \equiv g^{(j)}(a) + \alpha p^{h - 1}g^{(j + 1)}(a) + \cdots + \frac{(\alpha p^{h - 1})^{n - 1 - j}}{(n - 1 - j)!}g^{(n - 1)}(a) \mod (p^{h - 1})^{r/2 - j}\pi.
    \]
    Moreover, we have
    \begin{eqnarray*}
        && g^{(j)}(a + \alpha p^{h - 1})(z - a - \alpha p^{h - 1})^j\mathbbm{1}_{a + \alpha p^{h - 1} + p^h\Zp} \\
        && \equiv \left[g^{(j)}(a) + \alpha p^{h - 1}g^{(j + 1)}(a) + \cdots + \frac{(\alpha p^{h - 1})^{n - 1 - j}}{(n - 1 - j)!}g^{(n - 1)}(a)\right](z - a - \alpha p^{h - 1})^{j}\mathbbm{1}_{a + \alpha p^{h - 1} + p^h \Zp}
    \end{eqnarray*}
    modulo $\pi \latticeL{k}$.
\end{lemma}
\begin{proof}
  We prove the first congruence in the conclusion of the lemma.
  The proof 
  is similar to the proof that $g(z) \in \sC^{n - 1}(\Zp, E)$. 
  The difference is that we need to keep track of the estimates. We consider two cases.
    \begin{enumerate}
        \item $v_p(a - z_0) < h - 1$:
Since $g$ is analytic in the neighborhood $a + p^{h - 1}\Zp$, using \eqref{Small derivative formula for polynomial times logs}, we write
\begin{eqnarray}\label{Far away from z0 in Truncated Taylor expansion}
    g^{(j)}(a + \alpha p^{h - 1}) & = & g^{(j)}(a) + \cdots + \frac{(\alpha p^{h - 1})^{n - 1 - j}}{(n - 1 - j)!}g^{(n - 1)}(a) \nonumber\\
    & & + \frac{(\alpha p^{h - 1})^{n - j}}{(n - j)!}n!p^{x}[\logL(a - z_0) + H_n] \nonumber \\
    & & + \sum_{l \geq 1}\frac{(\alpha p^{h - 1})^{n - j + l}}{(n - j + l)!}n! p^{x} \frac{(-1)^{l - 1}}{(a - z_0)^l}(l - 1)!.
\end{eqnarray}
Since $x + v_p(\cL) \geq r/2 - n$ and $x \geq -1 \geq r/2 - n$, we see that the valuation of the term in the second line above is at least $(n - j)(h - 1) + r/2 - n$. Now
\begin{equation}\label{h greater than 2}
    (n - j)(h - 1) + r/2 - n > (h - 1)(r/2 - j) \iff h > 2.
\end{equation}
Therefore for $h \geq 3$, the term in the second line in equation \eqref{Far away from z0 in Truncated Taylor expansion} is $0$ modulo $(p^{h - 1})^{r/2 - j}\pi$. Next, write the general term in the last sum in equation \eqref{Far away from z0 in Truncated Taylor expansion} as
\[
    p^{x}(\alpha p^{h - 1})^{n - j}\frac{n!}{(n - j)!}\left(\frac{\alpha p^{h - 1}}{a - z_0}\right)^{l}(-1)^{l - 1}\frac{1}{l{n - j + l \choose l}}.
\]
Using the estimate $v_p{n - j + l \choose l} \leq \lfloor \log_p(n - j + l)\rfloor - v_p(l)$ obtained using Kummer's theorem, the valuation of this term is greater than or equal to
\[
    -1 + (n - j)(h - 1) + l - \log_p(n - j + l).
\]
Moreover,
\begin{eqnarray*}
    &&-1 + (n - j)(h - 1) + l - \log_p(n - j + l) > (r/2 - j)(h - 1) \\
    && \qquad \stackrel{h \geq 3}{\iff} 2n - r - 1 + l > \log_p(n - j + l),
\end{eqnarray*}
which is true for all $l \geq 1$. Therefore we have proved the first congruence in the lemma when $v_p(a - z_0) < h - 1$.

\item $v_p(a - z_0) \geq h - 1$: Recall the derivative formula \eqref{der formula}:
\begin{eqnarray}\label{Derivative formula for g(z)}
    g^{(j)}(z) = \left[\frac{n!}{(n - j)!}p^x(z - z_0)^{n - j}\logL(z - z_0) + p^x t_j (z - z_0)^{n - j}\right]\mathbbm{1}_{\Zp},
\end{eqnarray}
where $t_j \in \Zp$. The valuations of the terms $g^{(j)}(a + \alpha p^{h - 1}), \> g^{(j)}(a), \> \alpha p^{h - 1}g^{(j+1)}(a), \> \ldots,$ $(\alpha p^{h - 1})^{n - 1 - j}g^{(n - 1)}(a)$ are greater than or equal to $(n - j)(h - 1) + r/2 - n$. So equation \eqref{h greater than 2} implies that these terms are congruent to $0$ modulo $(p^{h - 1})^{r/2 - j}\pi$. This proves the first congruence in the lemma when $v_p(a - z_0) \geq h - 1$.
\end{enumerate}

\noindent The second congruence in the conclusion of the lemma follows from the first
using Lemma~\ref{Integers in the lattice}.
\end{proof}

\begin{lemma}\label{telescoping lemma}
    Let $k/2 \leq n \leq r \leq p - 1$ and $z_0 \in \Zp$. Fix $x \in \bQ$ such that $x \geq - 1$ and $x + v_p(\cL) \geq r/2 - n$. 
    If
        $$g(z) = p^x(z - z_0)^n\logL(z - z_0)\mathbbm{1}_{\Zp},$$
    then we have $g_h(z) \equiv g_2(z) \mod \pi\latticeL{k}$ for $h \geq 3$.
\end{lemma}
\begin{proof}
    Making the substitution $m = a + \alpha p^{h - 1}$, we write
    \[
        g_h(z) = \sum_{a = 0}^{p^{h - 1} - 1}\sum_{\alpha = 0}^{p - 1}\left[ \sum_{j = 0}^{n - 1} \frac{g^{(j)}(a + \alpha p^{h - 1})}{j!}(z - a - \alpha p^{h - 1})^j\right]\mathbbm{1}_{a + \alpha p^{h - 1} + p^h \Zp}.
    \]
    %
By Lemma \ref{Truncated Taylor expansion}
    \begin{eqnarray*}
        g_h(z) \equiv \sum_{a = 0}^{p^{h - 1} - 1}\sum_{\alpha = 0}^{p - 1} && \!\!\!\! \!\!\!\! \left\{\left[\frac{1}{0! \> 0!}g(a) + \frac{\alpha p^{h - 1}}{0! \> 1!} g^{(1)}(a) + \cdots + \frac{(\alpha p^{h - 1})^{n - 2}}{0! \> (n - 2)!}g^{(n - 2)}(a) + \frac{(\alpha p^{h - 1})^{n - 1}}{0! \> (n - 1)!}g^{(n - 1)}(a)\right] \right.\\
        && \!\!\!\!\!\!\!\! + \left[\frac{1}{1! \> 0!}g^{(1)}(a) + \frac{\alpha p^{h - 1}}{1! \> 1!} g^{(2)}(a) + \cdots + \frac{(\alpha p^{h - 1})^{n - 2}}{1! \> (n - 2)!}g^{(n - 1)}(a)\right](z - a - \alpha p^{h - 1}) \\ 
        && \quad \quad \vdots \qquad \qquad \qquad \vdots \\
        && \!\!\!\!\!\!\!\! + \left.\left[\frac{1}{(n - 1)! \> 0!}g^{(n - 1)}(a)\right](z - a - \alpha p^{h - 1})^{n - 1}\right\} \mathbbm{1}_{a + \alpha p^{h - 1} + p^h\Zp} \mod \pi\latticeL{k}.
    \end{eqnarray*}
    Expanding $[(z - a) - \alpha p^{h - 1}]^j$ for $0 \leq j \leq n - 1$ and collecting like terms of $(z - a)$ diagonally, the terms involving $\alpha$ disappear
    \begin{eqnarray*}
        g_h(z) & \equiv & \sum_{a = 0}^{p^{h - 1} - 1}\sum_{\alpha = 0}^{p - 1}\left[\sum_{j = 0}^{n - 1}\frac{g^{(j)}(a)}{j!}(z - a)^{j}\right]\mathbbm{1}_{a + \alpha p^{h - 1} + p^h\Zp}  \\
        & \equiv & \sum_{a = 0}^{p^{h - 1} - 1}\left[\sum_{j = 0}^{n - 1} \frac{g^{(j)}(a)}{j!}(z - a)^j\right]\mathbbm{1}_{a + p^{h - 1}\Zp} \equiv g_{h - 1}(z) \mod \pi \latticeL{k}.
    \end{eqnarray*}
    Iterating this process, we get $g_h(z) \equiv g_2(z) \mod \pi \latticeL{k}$. 
    Therefore the claim follows.
\end{proof}

    \begin{lemma}\label{qp-zp part is 0}
      Let $k/2 \leq n \leq r \leq p - 1$. For $i$ in a finite indexing set $I$, let $\lambda_i \in E$ and $z_i \in \Zp$ be such that $\sum_{i \in I}\lambda_i z_i^j = 0$ for all $0 \leq j \leq n$.
      Fix $x \in \bQ$ such that $x \geq -1$. Then for
        \[
            f(z) = \sum_{i \in I}p^x\lambda_i z^{r - n}(1 - zz_i)^n \logL(1 - zz_i)\mathbbm{1}_{p\Zp},
        \]
        we have $f(z) \equiv 0 \mod \pi \latticeL{k}$.
    \end{lemma}

     We refer the reader to \cite[Lemma 8.5]{CG23} since giving the proof here will take us too far away
        from the goal of these notes.
    We do however give an example of a choice of $\lambda_i$ and $z_i$ which make the
    sums in the above lemma vanish.
    
        \begin{lemma}\label{Main coefficient identites}
            If
                $I = \{0, \> 1, \> \ldots, n, \> p\}$ for some $r/2 < n \leq r \leq p - 1$, and $z_i = i$ for $i \in I$,
            then there are $\lambda_i \in \zp$ not all zero such that $\sum\limits_{i \in I}\lambda_i z_i^j = 0$ for $0 \leq j \leq n$. Moreover, these $\lambda_i$ satisfy
                $\lambda_0 = 1 \!\! \mod p$, $\lambda_p = -1$ and $\lambda_i = 0 \!\! \mod p$ for $1 \leq i \leq n$.
        \end{lemma}
     
        \begin{proof}

  Taking $\lambda_p = -1$ and separating the $i = p$ summand from
    the equations $\sum\limits_{i \in I}\lambda_i i^j = 0$, we see that for $0 \leq j \leq n$, we have
                \[
                    \sum_{i = 0}^{n}\lambda_i i^j = p^j.
                \]
                Since $i \not \equiv i' \!\! \mod p$ for $0 \leq i, \> i' \leq n \leq p - 1$, there is a unique solution to the system of equations above with $\lambda_i \in \Zp$. Moreover, reducing the equations in the display above modulo $p$, we get
                \[
                    \sum_{i = 0}^{n}\br{\lambda_i}\> \br{i}^j = \begin{cases}
                    1 & \text{ if } j = 0 \\
                    0 & \text{ if } 1 \leq j \leq n.
                    \end{cases}
                \]
                Since $\br{\lambda_0} = 1$ and $\br{\lambda_i} = 0$ for $1 \leq j \leq n$ is the unique solution to the mod $p$ system of equations in the display above, we obtain the lemma.
               %
        \end{proof}

\subsubsection{Strategy} We now give an outline of the strategy we used to establish congruences using
the generalities above. These congruences allow us to determine the JH factors in the reduction 
 of $\latticeL{k}$.

We first choose in an intelligent
way a poly$\cdot$log function $$g : \qp \rightarrow E$$
as in equation \eqref{polylog fn}. 
The importance of choosing this function carefully cannot
be overemphasized. 
Note that $g(z) \in D(k)$. Write
$$g(z) = g(z) \mathbbm{1}_{\zp} + g(z) \mathbbm{1}_{\qp \setminus \zp}.$$
We claim that $g(z) \mathbbm{1}_{\zp} \in D(k)$. Indeed, the function
$$z^r g(1/z) \mathbbm{1}_{\zp}(1/z) = z^rg(1/z) \mathbbm{1}_{\zp \setminus p\zp}(z) = z^rg(1/z) - z^rg(1/z) \mathbbm{1}_{p\zp}(z)$$ on $\zp \setminus 0$ extends to a function in $\sC^{r/2}(\Zp, E)$, since
$z^rg(1/z)$ does by definition of $D(k)$, and an easy further check shows that multiplication by $\mathbbm{1}_{p\zp}$
preserves the space $\sC^{r/2}(\Zp, E)$. It follows that $g(z) \mathbbm{1}_{\qp \setminus \zp} \in D(k)$.
Now:

\begin{itemize}
\item For large $h$ we have
  $g(z) \mathbbm{1}_{\zp}  \equiv \tilde{g}_h(z) \mod \pi \latticeL{k}$. 
  
  Indeed  by \eqref{Convergence proposition}, we
  have that $\tilde{g}_h$ is close to $g(z) \mathbbm{1}_{\zp}$ in the $v'_{\sC^{r/2}}$-adic topology. Recall that
  by Theorem~\ref{theorem bb}, the locally polynomial functions $e_{i,j,{r/2}}$ for $i \geq 0$ and 
  $0\leq j \leq \lfloor r/2 \rfloor$ form a Banach basis
  of $\sC^{r/2}(\zp,E)$. So for large $h$ there are constants $a_{i,j}$ with $v_p(a_{i,j}) > 0$ such
  that $$g(z) \mathbbm{1}_{\zp}  - \tilde{g}_h(z) = \sum_{i,j} a_{i,j} e_{i,j,r/2}.$$ We can check that this is
  an equality of functions in $D(k)$. Since $$\lfloor l(i) (r/2) \rfloor - l(i) j \geq( l(i) - 1)(r/2 -j),$$ we deduce 
  by Lemma \ref{Integers in the lattice} that the $e_{i,j,{r/2}}$ are in the lattice $\latticeL{k}$.
  
\item Then using Lemma \ref{Congruence lemma}, we see that $g(z) \mathbbm{1}_{\Zp} \equiv g_h(z) \mod \pi\latticeL{k}$        
  for large $h$. 
\item Then using Lemma~\ref{Truncated Taylor expansion} and Lemma~\ref{telescoping lemma} we may assume
  that we can descend from large $h$ to $h = 2$
  (though as we see in \cite{CG23} for some exceptional $g$ we can only descend to $h = 3$.) 
\item We also prove that 
  $g(z)\mathbbm{1}_{\Qp \setminus \Zp} \equiv 0 \mod \pi \latticeL{k}$ using Lemma \ref{qp-zp part is 0} (though as we
  see in \cite{CG23} in an exceptional case this function also contributes to the argument).
\end{itemize}
Since $g(z)$ is equal to $0$ in $\tB(k, \cL)$, we get a congruence $$g_2(z) \equiv 0 \mod \pi\latticeL{k}$$ 
(or $g_3(z) \equiv 0 \mod \pi \latticeL{k}$ for some functions $g$).

These congruences
allow us to show that the subquotients $F_{2l,2l+1}$ for $0 \leq l \leq r$ of $\br{\latticeL{k}}$ are
quotients of cokernels of linear expressions involving Iwahori-Hecke operators acting on
compactly induced representations from
$IZ$ to $G$. Finally, we are able to deduce Theorem~\ref{Main theorem in the second part of my thesis}  by applying the Iwahori mod $p$ LLC stated in Theorem~\ref{Iwahori mod p LLC}.

\subsection{Analysis of $\br{\latticeL{k}}$ around all points but the last}\label{Common section}
    
Section 9 of \cite{CG23} gives a uniform treatment of all the subquotients around the marked points on the $\nu$ line appearing in
Theorem~\ref{Main theorem in the second part of my thesis} except for the last marked point.
The goal of this section is to describe how this is done in a simplified example.  

Note that the 
last marked point $\nu = \frac{1}{2}$, respectively $\nu = 0$, for $r$ odd, respectively $r$ even, is particularly tricky to deal
with and  requires much additional work (see \cite[Sections 10, 11]{CG23}). In fact, the former case requires establishing a formula for
the constant $\lambda_i$ for  $i = \frac{r+1}{2}$ with $r$ odd which requires a much more elaborate treatment. 
The latter case requires
working in a non-commutative Hecke algebra. We do not make any further comments about the
behaviour of the subquotients about these last marked points in these notes.

    In this section, we give a sketch of the proof of the fact that if $\nu = i - r/2$ for $i = 1, \> 2, \> \ldots, \> \lceil r/2 \rceil - 1$, then the map $\IZind a^i d^{r - i} \twoheadrightarrow F_{2i, \> 2i + 1}$ factors as
    \[
        \IZind a^i d^{r - i} \twoheadrightarrow \frac{\IZind a^i d^{r - i}}{\im(T_{1,  2} - \lambda_i)} \twoheadrightarrow F_{2i, \> 2i + 1}, 
    \]
    where $$\lambda_i = (-1)^i i {r - i + 1 \choose i}p^{r/2 - i}\cL.$$ Moreover, 
    we prove that 
    the second map in the display above induces a surjection 
    \[
        \pi(r - 2i, \lambda_i, \omega^i) \twoheadrightarrow F_{2i, \> 2i + 1}.
    \]
Further details may be found in \cite[Section 9.2]{CG23}.

Here is the key technical proposition. It starts with an intelligently chosen function $g : \qp \rightarrow E$ 
as in \eqref{polylog fn}.

    \begin{Proposition}{\cite[Proposition 9.6]{CG23}}\label{nu leq}
        Let $k/2 < n \leq k - 2 \leq p - 1$. For $\nu \leq 1 + r/2 - n$, set
        \[
            g(z) = p^x\left[\sum_{i \in I} \lambda_i (z - i)^n\logL(z - i)\right],
        \]
        where $x \in \bQ$ with $x + \nu = r/2 - n$, $I = \{0, 1, \ldots, n, p\}$ and the $\lambda_i \in \zp$
        are as in Lemma~\ref{Main coefficient identites}. 
        Then, we have
        \[
            g(z) \equiv p^{1 + x}\sum_{a = 1}^{p - 1}a^{-1}z^n\mathbbm{1}_{a + p\Zp} + \sum_{j = \lceil r/2 \rceil}^{n - 1}(-1)^{n - j + 1}{n \choose j}p^{x + n - j}\cL z^j \mathbbm{1}_{p\Zp} \mod \pi \latticeL{k}.
        \]
    \end{Proposition}
    \begin{proof}

      Write $g(z) = g(z)\mathbbm{1}_{\Zp} + g(z)\mathbbm{1}_{\Qp \setminus \Zp}$. By Lemma~\ref{Main coefficient identites}, 
      we have
      \begin{eqnarray}\label{Identities concerning the coefficients in leq}
        \sum_{i \in I}\lambda_i i^{j} = 0, \text{ for } 0 \leq j \leq n,
        \> \lambda_0 \equiv 1 \!\! \mod p, \> \lambda_i \equiv 0 \!\! \mod p \text{ for }
                1 \leq i \leq n  \text{ and } \lambda_p = -1. 
      \end{eqnarray}
      Since the summation identity in \eqref{Identities concerning the coefficients in leq} is also true
      for $j = n$, the $\logL(z)$ term dies and so
        \[
            w\cdot g(z)\mathbbm{1}_{\Qp \setminus \Zp} = \sum_{i \in I}p^x\lambda_iz^{r - n}(1 - zi)^n\logL(1 - zi)\mathbbm{1}_{p\Zp}.
        \]
        Lemma \ref{qp-zp part is 0} implies that $w\cdot g(z) \mathbbm{1}_{\Qp \setminus \Zp} \equiv 0 \mod \pi \latticeL{k}$.
    
        Next, using Lemmas~\ref{Congruence lemma}, \ref{Truncated Taylor expansion} and~\ref{telescoping lemma} 
         we see that $g(z)\mathbbm{1}_{\Zp} \equiv g_2(z) \mod \pi \latticeL{k}$. We next claim that $g_2(z) - g_1(z) \equiv 0 \mod \pi \latticeL{k}$. 
         We see that the coefficient of $(z - a - \alpha p)^j \mathbbm{1}_{a + \alpha p + p^2 \Zp}$ in $g_2(z) - g_1(z)$ is
        \begin{eqnarray}\label{jth summand in g2 - g1 in leq}
          && {n \choose j}
             \sum_{i \in I}p^x \lambda_i  \Big[(a + \alpha p - i)^{n - j}\logL(a + \alpha p - i) - (a - i)^{n - j}\logL(a - i) \\
            && - {n - j \choose 1}(\alpha p)(a - i)^{n - j - 1}\logL(a - i) - \cdots - {n - j \choose n - j - 1}(\alpha p)^{n - j - 1}(a - i)\logL(a -  i)\Big]. \nonumber
        \end{eqnarray}
        By Lemma~\ref{Integers in the lattice}, to prove the claim it suffices to show
        that this coefficient is congruent to $0$ modulo $p^{r/2 - j}\pi$. This can be done but we omit the details.

        This shows that $g(z) \equiv g_1(z) \!\! \mod \pi \latticeL{k}$. Next, we simplify $g_1(z)$. Recall that
        \begin{eqnarray}\label{g1 in leq}
            g_1(z) = \sum_{a = 0}^{p - 1}\left[\sum_{j = 0}^{n - 1}\frac{g^{(j)}(a)}{j!}(z - a)^j\right]\mathbbm{1}_{a + p\Zp},
        \end{eqnarray}
        where
        \begin{eqnarray}\label{jth summand in g1 in leq}
            \frac{g^{j}(a)}{j!} = \sum_{i \in I}{n \choose j}p^x\lambda_i(a - i)^{n - j}\logL(a - i).
        \end{eqnarray}
        First assume that $a \neq 0$. The $i \neq 0, a, p$ summands on the right side of equation \eqref{jth summand in g1 in leq} are congruent to $0$ modulo $\pi$ since $a \not \equiv i \!\! \mod p$ and $p \mid \lambda_i$ for $i \neq 0, p$. The $i = a$ term is equal to $0$. The sum of the $i = 0$ and $i = p$ terms
        in \eqref{jth summand in g1 in leq}
        is congruent to 
        \[
            {n \choose j}p^x\left[a^{n - j}\logL(a) - (a - p)^{n - j}\logL(a - p)\right] \mod \pi
          \]
        since, by \eqref{Identities concerning the coefficients in leq}, $\lambda_0 = 1 \!\! \mod p$
        and $\lambda_p = -1$.  
        Expanding $(a - p)^{n - j}$ and dropping the terms divisible by $p$ since $x \geq -1$ and $p \> \vert \> \logL(a - p)$, this equals
        \[
            {n \choose j}p^x(-a^{n - j})\logL(1 - a^{-1}p) \mod \pi.
        \]
        Expanding $\logL(1 - a^{-1}p)$ using the usual Taylor series and dropping the terms that are congruent to $0$ mod $\pi$, we see that the sum of the $i = 0$ and $p$ summands in \eqref{jth summand in g1 in leq} is
        \[
            {n \choose j}p^{1 + x}a^{n - j - 1} \mod \pi.
        \]
        Therefore for $a \neq 0$, we get
        \begin{eqnarray}\label{jth term in a neq 0 in g1 in leq}
            \frac{g^{(j)}(a)}{j!} \equiv {n \choose j}p^{1 + x}a^{n - j - 1} \mod \pi.
        \end{eqnarray}

        Next, assume that $a = 0$. Then again the $i \neq 0, \> p$ summands in
        \eqref{jth summand in g1 in leq} are congruent to $0$ modulo $\pi$ by the same reasoning as in the $a \neq 0$ case above. The $i = 0$ summand is $0$. The $i = p$ summand is
        \[
            -{n \choose j}p^x(-p)^{n - j}\cL.
        \]
        Therefore for $a = 0$, we have
        \begin{eqnarray}\label{jth term in a = 0 in g1 in leq}
            \frac{g^{(j)}(0)}{j!} \equiv (-1)^{n - j + 1}{n \choose j}p^{x + n - j}\cL \mod \pi.
        \end{eqnarray}
        Putting equations \eqref{jth term in a neq 0 in g1 in leq} and \eqref{jth term in a = 0 in g1 in leq} in equation \eqref{g1 in leq} and using Lemma \ref{Integers in the lattice}, we get
        \begin{eqnarray*}
            g_1(z) & \equiv & \sum_{a = 1}^{p - 1}\left[\sum_{j = 0}^{n - 1}{n \choose j}p^{1 + x}a^{n - j - 1}(z - a)^j\right]\mathbbm{1}_{a + p\Zp} + \left[\sum_{j = 0}^{n - 1}(-1)^{n - j + 1}{n \choose j}p^{x + n - j}\cL z^j\right]\mathbbm{1}_{p\Zp} \\
            & \equiv & \sum_{a = 1}^{p - 1}p^{1 + x}a^{-1}\left[z^n - (z - a)^n\right]\mathbbm{1}_{a + p\Zp} + \left[\sum_{j = 0}^{n - 1}(-1)^{n - j + 1}{n \choose j}p^{x + n - j}\cL z^j\right]\mathbbm{1}_{p\Zp} \\
            && \qquad \qquad \mod \pi \latticeL{k}.
        \end{eqnarray*}
        Using Lemma~\ref{stronger bound for polynomials of large degree with varying radii} for the first sum and Lemma \ref{Integers in the lattice} for the second sum above, we get
        \[
            g_1(z) \equiv \sum_{a = 1}^{p - 1}p^{1 + x}a^{-1}z^n\mathbbm{1}_{a + p\Zp} + \left[\sum_{j = \lceil r/2 \rceil}^{n - 1}(-1)^{n - j + 1}{n \choose j}p^{x + n - j}\cL z^j\right]\mathbbm{1}_{p\Zp} \mod \pi \latticeL{k}.
        \]
        This proves the proposition since $g(z) \equiv g_1(z) \!\! \mod \pi \latticeL{k}.$
    \end{proof}
    
    \begin{theorem}{\cite[Theorem 9.7]{CG23}} \label{Final theorem for leq}
        Let $i = 1, \> 2, \> \ldots, \> \lceil r/2 \rceil - 1$. If $\nu = i - r/2$, then the map $\IZind a^i d^{r - i} \twoheadrightarrow F_{2i, \> 2i + 1}$ factors as
        \[
            \IZind a^i d^{r - i} \twoheadrightarrow \frac{\IZind a^i d^{r - i}}{\im(T_{1, 2} - \lambda_i)} \twoheadrightarrow F_{2i, \> 2i + 1},
        \]
        where $$\lambda_i = (-1)^i i {r - i + 1 \choose i}p^{r/2 - i}\cL.$$ Moreover, 
        the second map in the display above induces a surjection 
    \[
        \pi(r - 2i, \lambda_i, \omega^i) \twoheadrightarrow F_{2i, \> 2i + 1}.
    \]
    \end{theorem}
    \begin{proof}
        All congruences in this proof are in the space $\br{\latticeL{k}}$ modulo the image of the subspace $\IZind \oplus_{j < r - i}\Fq X^{r - j}Y^j$ under $\IZind \SymF{k - 2} \twoheadrightarrow \br{\latticeL{k}}$.

        Fix an $i = 1, 2, \ldots, \lceil r/2 \rceil - 1$. Applying Proposition~\ref{nu leq} with $n = r - i + 1$ and the remark above, we get
        \begin{eqnarray}\label{Main equation in leq}
            0 \equiv p^{1 + x}\sum_{a = 1}^{p - 1}a^{-1}z^{r - i + 1}\mathbbm{1}_{a + p\Zp} + (r - i + 1)p^{x + 1}\cL z^{r - i}\mathbbm{1}_{p\Zp}.
        \end{eqnarray}

 \noindent Since $\nu = i - r/2$, we see that $x = -1$.
\noindent After much massaging of this equation using some inductive steps and some matrix computations (several
pages of work in \cite{CG23}), we get
        \[
            0 \equiv \sum_{a = 0}^{p - 1}p^{r/2 - i}(z - ap)^i\mathbbm{1}_{ap + p^2\Zp} - (-1)^{i}i {r - i + 1 \choose i}p^{r/2 - i}\cL z^i\mathbbm{1}_{p\Zp}.
        \]

        This equation is the image of $(T_{1, 2} - \lambda_i)\llbracket \id, X^iY^{r - i}\rrbracket$ under $\IZind a^i d^{r - i} \twoheadrightarrow F_{2i, \> 2i + 1}$. Indeed, 
        \[
            T_{1, 2}\llbracket \id, X^iY^{r - i}\rrbracket = \sum_{\lambda \in I_1}\left\llbracket \begin{pmatrix}1 & 0 \\ -p\lambda & p\end{pmatrix}, X^iY^{r - i}\right\rrbracket
            \mapsto \sum_{\lambda \in I_1}p^{r/2 - i}(z - \lambda p)^i\mathbbm{1}_{\lambda p + p^2 \Zp}.
        \]
        Therefore we have proved that the surjection $\IZind a^{i}d^{r - i} \twoheadrightarrow F_{2i, \> 2i + 1}$ factors as
        \[
            \IZind a^{i}d^{r - i} \twoheadrightarrow \frac{\IZind a^{i}d^{r - i}}{\im(T_{1, 2} - \lambda_i^{})} \twoheadrightarrow F_{2i, \> 2i + 1}.
        \]

        Now consider the following exact sequences
            \[
                \begin{tikzcd}
                    0 \arrow[r] & \im T_{1, 2} \arrow[r] \arrow[d, two heads] & \IZind a^{i}d^{r - i} \arrow[r]\arrow[d, two heads] & \dfrac{\IZind a^{i}d^{r - i}}{\im T_{1, 2}} \arrow[r]\arrow[d, two heads] & 0 \\
                    0 \arrow[r] & S \arrow[r] & F_{2i, \> 2i + 1} \arrow[r] & Q \arrow[r] & 0,
                \end{tikzcd}
            \]
            where $S$ is the image of $\im T_{1, 2}$ under the surjection $\IZind a^id^{r - i} \twoheadrightarrow F_{2i, \> 2i + 1}$ and $Q = F_{2i, \> 2i + 1}/S$. Since $\im (T_{1, 2} - \lambda_i^{})$ maps to $0$ under the middle vertical map in the diagram above, we see that the right vertical map is $0$. Therefore we have a surjection $\im T_{1, 2} \twoheadrightarrow F_{2i, \> 2i + 1}$.
            By the first isomorphism theorem and \cite[Proposition 3.1]{AB15}, we get a surjection
            \[
                \frac{\IZind a^{i}d^{r - i}}{\im T_{-1, 0}} \twoheadrightarrow F_{2i, \> 2i + 1},
            \]
            which factors through
            \[
              \frac{\IZind a^{i}d^{r - i}}{\im T_{-1, 0} + \im (T_{1, 2} - \lambda_i^{})}
              \twoheadrightarrow F_{2i, \> 2i + 1}.
            \]
            The space on the left is  
            \[
                \frac{\IZind a^{i}d^{r - i}}{\im T_{-1, 0} + \im (T_{1, 2} - \lambda_i)} = \pi(r - 2i, \lambda_i, \omega^{i}).
            \]
            This completes the sketch of the proof of the theorem.
    \end{proof}

        \subsection{Reduction mod $p$ of $\latticeL{k}$}
        \label{Section containing the proof of the main theorem}
    In this section, we summarize all the results proved in \cite{CG23} (such as Theorem~\ref{Final theorem for leq} above)
    and mention how they are used to prove Theorem \ref{Main theorem in the second part of my thesis}. 
    Recall that $$\nu = v_p(\cL - H_{-} - H_{+}).$$
\begin{theorem}{\cite[Theorem 12.1]{CG23}}\label{Main final theorem in the second part of my thesis}
    For $3 \leq k \leq p + 1$ and $p \geq 5$, the semi-simplification of the reduction $\br{V}_{k,\sL}$ of
    the semi-stable representation $V_{k,\sL}$ of $G_{\Qp}$ of Hodge-Tate weights $(0,k-1)$ and $\sL$-invariant $\sL$
    satisfies:
    \[
    \br{V}_{k, \cL} \sim
    \begin{cases}
        \ind (\omega_2^{r+1 + (i-1)(p-1)}), & \text{ if $(i-1) - r/2 < \nu < i -r/2$} \\
        \mu_{\lambda_i}\omega^{r+1-i} \oplus \mu_{\lambda_i^{-1}}\omega^{i}, & \text{ if $\nu = i -r/2$},
    \end{cases}
    \]
    where $1 \leq i \leq \dfrac{r+1}{2}$ if $r$ is odd and $1 \leq i \leq \dfrac{r+2}{2}$ if $r$ is even. The constants $\lambda_i$ are determined by
    \begin{eqnarray*}
        \lambda_i & = & \br{(-1)^i \> i {r+1-i \choose i}p^{r/2-i}(\cL - H_{-} - H_{+})}, \quad \text{ if } 1 \leq i < \dfrac{r + 1}{2} \\
        \lambda_{i} + \lambda_i^{-1} & = & \br{(-1)^i \> i{r + 1 - i \choose i}p^{r/2 - i}(\cL - H_{-} - H_{+})}, \quad \text{ if } i = \dfrac{r + 1}{2} \text{ and } r \text{ is odd}.
    \end{eqnarray*}
    We follow the conventions stated in the Introduction.
\end{theorem}
\begin{proof}
  We collect the necessary results proved in \cite{CG23} here.
  We first state the common results for odd and even weights (the fourth of which was
  sketched in 
  Theorem \ref{Final theorem for leq}):
    \begin{enumerate}
        \item For $i = 1, 2, \ldots, \lceil r/2 \rceil - 1$ if $\nu > i - r/2$, then $F_{2i - 2, \> 2i - 1} = 0$.
        \item For $i = 1, 2, \ldots, \lceil r/2 \rceil - 1$ if $\nu = i - r/2$, then $\pi([2i - 2 - r], \lambda_i^{-1}, \omega^{r - i + 1}) \twoheadrightarrow F_{2i - 2, \> 2i - 1}$.
        \item For $i = 1, 2, \ldots, \lceil r/2 \rceil - 1$ if $\nu < i - r/2$, then $F_{2i, \> 2i + 1} = 0$.
        \item For $i = 1, 2, \ldots, \lceil r/2 \rceil - 1$ if $\nu = i - r/2$, then $\pi(r - 2i, \lambda_i, \omega^{i}) \twoheadrightarrow F_{2i, \> 2i + 1}$.
    \end{enumerate}
    Next, we state the extra results for odd weights around the point $\nu = \frac{1}{2}$.
    \begin{enumerate}
        \item[5.] If $\nu \geq 0.5$, then $\pi(p - 2, \lambda_{\frac{r + 1}{2}}, \omega^{\frac{r + 1}{2}}) \oplus \pi(p - 2, \lambda_{\frac{r + 1}{2}}^{-1}, \omega^{\frac{r + 1}{2}}) \twoheadrightarrow F_{r - 1, \> r}$.
        \item[6.] If $-0.5 < \nu < 0.5$, then
          $\dfrac{\IZind a^{\frac{r - 1}{2}}d^{\frac{r + 1}{2}}}{\im T_{-1, 0}} \twoheadrightarrow F_{r-1, \> r}$.
          
    \end{enumerate}
    Finally, we state the extra results for even weights around the point $\nu = 0$.
    \begin{enumerate}
        \item[5.] If $\nu > 0$, then $F_{r - 2, \> r - 1} = 0$.
        \item[6.] If $\nu = 0$, then $\pi(p - 3, \lambda_{r/2}^{-1}, \omega^{\frac{r + 2}{2}}) \twoheadrightarrow F_{r - 2, \> r - 1}$.
        \item[7.] If $\nu < 0$, then $F_{r, \> r + 1} = 0$.
        \item[8.] If $\nu = 0$, then $\pi(p - 1, \lambda_{r/2}, \omega^{r/2}) \twoheadrightarrow F_{r}$ and $F_{r + 1}$ is not isomorphic to $\pi(0, 0, \omega^{r/2})$.
    \end{enumerate}


    
    The proof of the theorem is then a standard application of the
  compatibility with respect to mod $p$ reduction between the $p$-adic and the
  Iwahori mod $p$ LLC (stated in Theorem~\ref{Iwahori mod p LLC}).
\end{proof}

\vspace{0.3cm}

{\noindent \bf Acknowledgements}:   This article is based on notes of lectures in a mini-course I gave at IIT Jammu 
for a National Center for Mathematics workshop on the $p$-adic Local Langlands program in the semi-stable case 
in September 2024. I wish to thank A. Chitrao and A. Jana for several useful discussions during the preparation of 
these notes.  This article was 
solicited by IJPAM for a special volume of articles written by Fellows of the Indian National Science Academy to 
celebrate the 90th year of the academy. It was written up mostly during a visit to Japan in October 2025 and 
I thank S. Kobayashi, M. Kurihara and S. Yasuda for their hospitality.

\vspace{.2cm}



\begin{thebibliography}{1000000}
\bibitem[AB15]{AB15}
U. K. Anandavardhanan and Gautam Borisagar, \textit{Iwahori–Hecke model for supersingular representations of $\mathrm{GL}_2(\Qp)$.} J. Algebra \textbf{423} (2015), 1--27.

\bibitem[BL94]{BL94}
Laure Barthel and Ron Livné, \textit{Irreducible modular representations of $\mathrm{GL}_2$ of a local field.} Duke Math. J. \textbf{75} (1994), no. 2, 261--292.

\bibitem[BL95]{BL95}
Laure Barthel and Ron Livné, \textit{Modular representations of $\mathrm{GL}_2$ of a local field: the ordinary, unramified case.} J. Number Theory \textbf{55} (1995), no. 1, 1--27.

\bibitem[BL22]{BL22}
John Bergdall and Brandon Levin. \textit{Reductions of some two-dimensional crystalline representations 
via Kisin modules.} Int. Math. Res. Not., 2022(4), 3170--3197.

\bibitem[BLL23]{BLL22}
John Bergdall, Brandon Levin and Tong Liu, \textit{Reductions of 2-dimensional semi-stable representations with large $\sL$-invariant.} J. Inst. Math. Jussieu \textbf{22} (2023), no. 6, 2619--2644.


\bibitem[BLZ04]{BLZ04}
Laurent Berger, Hanfeng Li and Hui June Zhu, \textit{Construction of some families of 2-dimensional crystalline representations.} Math. Ann. \textbf{329} (2004), 365--377.

\bibitem[BG15]{BG15}
Shalini Bhattacharya and Eknath Ghate, \textit{Reductions of Galois representations of slope in $(1,2)$.} Doc. Math.  \textbf{20} (2015), 943--987.

\bibitem[BGR18]{BGR18}
Shalini Bhattacharya, Eknath Ghate and Sandra Rozensztajn, \textit{Reductions of Galois representations of slope 1.} J. Algebra \textbf{508} (2018), 98--156.

\bibitem[BGV25]{BGR25}
Shalini Bhattacharya, Eknath Ghate and V. Ravitheja, \textit{Reductions of Galois representations of slope $< p$.},
90 pp., in preparation.

\bibitem[Bre03a]{Bre03}
Christophe Breuil, \textit{Sur quelques repr\'esentations modulaires et $p$-adiques de $\mathrm{GL}_2(\qp)$ I}. Compositio Math. \textbf{138} (2003), 165--188.


\bibitem[Bre03b]{Bre03b}
Christophe Breuil, \textit{Sur quelques représentations modulaires et $p$-adiques de $\mathrm{GL}_2(\qp)$ II,} J. Inst. Math. Jussieu \textbf{2} (2003), 1--36.

\bibitem[Bre04]{Bre04}
Christophe Breuil, \textit{Invariant $\sL$ et s\'erie sp\'eciale $p$-adique.} Ann. Sci. \'Ecole Norm. Sup. (4) 37.4 (2004), 559--610.

\bibitem[Bre10a]{Breuil ICM}
Christophe Breuil, \textit{The emerging p-adic Langlands programme.} Proceedings of the I.C.M. (2010) Vol. II, 203--230.

\bibitem[Bre10b]{Bre10}
Christophe Breuil, \textit{S\'erie sp\'eciale $p$-adique et cohomologie \'etale compl\'et\'ee.} Ast\'erisque
\textbf{331} (2010), 65--115.

\bibitem[BM02]{BM}
  Christophe Breuil and Ariane M\'ezard, \textit{Multiplicit\'es modulaires et repr\'esentations de $\mathrm{GL}_2(\Zp)$ et de $\mathrm{Gal}(\brqp/\Qp)$ en
    $l = p$. With an appendix by Guy Henniart.} Duke Math. J. \textbf{115} (2002), no. 2, 205--310.

\bibitem[BM10]{BM10}
  Christophe Breuil and Ariane M\'ezard, \textit{Repr\'esentations semi-stables de
    ${\mathrm{GL}}_2({\mathbb Q}_p)$, demi-plan $p$-adique et r\'eduction modulo $p$.}
  Ast\'erisque \textbf{331}  (2010), 117--178.
    
\bibitem[BG09]{BG09}
Kevin Buzzard and Toby Gee, \textit{Explicit reduction modulo p of certain 2-dimensional crystalline representations.} Int. Math. Res. Not. IMRN (2009), no.12, 2303–-2317.

\bibitem[BG13]{BG13}
Kevin Buzzard and Toby Gee, \textit{Explicit reduction modulo p of certain 2-dimensional crystalline representations, II.} Bull. Lond. Math. Soc. \textbf{45} (2013), no. 4, 779--788.

\bibitem[Chi23]{Chi23}
Anand Chitrao, \textit{An Iwahori theoretic mod $p$ Local Langlands Correspondence.} Can. Math. Bull. \textbf{68} (2025), 805-817.

\bibitem[CGY25]{CGY21}
  Anand Chitrao, Eknath Ghate and Seidai Yasuda, \textit{Semi-stable representations as limits of crystalline representations.} Alg. Num. Theory \textbf{19} (2025), no. 6, 1049--1097. 
  
\bibitem[CG24]{CG23}
  Anand Chitrao and Eknath Ghate, \textit{Reductions of semi-stable representations using the Iwahori mod $p$ Local Langlands Correspondence} (2024), Submitted, \url{https://arxiv.org/abs/2311.03740v2}.
  


\bibitem[Col10a]{Col10b}
Pierre Colmez, \textit{Repr\'esentations de $\mathrm{GL}_2(\qp)$ et $(\varphi,\Gamma)$-modules.} Ast\'erisque \textbf{330} (2010), 281--509.

\bibitem[Col10b]{Col10a}
Pierre Colmez, \textit{Fonctions d’une variable p-adique.} Ast\'erisque \textbf{330} (2010), 13--59.

\bibitem[Edi92]{Edi92} 
Bas Edixhoven, \textit{The weight in Serre’s conjectures on modular forms.} Invent. Math. \textbf{109} (1992), no. 3, 563--594.

\bibitem[GG15]{GG15}
Abhik Ganguli and Eknath Ghate, \textit{Reductions of Galois representations via the mod $p$ Local Langlands correspondence.} J. Number Theory \textbf{147} (2015), 250--286. 


\bibitem[Gao17]{Gao17}
Hui Gao, \textit{Galois lattices and strongly divisible lattices in the unipotent case.} J. Reine Angew. Math. \textbf{728} (2017), 263--299.

\bibitem[Gha21]{Gha21}
Eknath Ghate, \textit{A zig-zag conjecture and local constancy for Galois representations.} RIMS K\^oky\^uroku Bessatsu \textbf{B86} (2021), 249--268.

\bibitem[Gha23]{Gha22}
Eknath Ghate, \textit{Zig-zag for Galois representations.} Preprint, \url{https://arxiv.org/abs/2211.12114v2}.


\bibitem[GP12]{GP12}
Eknath Ghate and Pierre Parent, \textit{On uniform large Galois images for modular abelian varieties.} 
Bull. LMS \textbf{44} (2012), no. 6, 1169-1181. 

 



\bibitem[GR25]{GR19}
Eknath Ghate and Vivek Rai, \textit{Reductions of Galois representations of $\frac{3}{2}$.} Kyoto J. Math. \textbf{65} (2025), no. 3, 595--636.

\bibitem[GV22]{GV22}
Eknath Ghate and Ravitheja Vangala, \textit{The monomial lattice in modular symmetric power representations.}
Algebr. Represent. Theory \textbf{25} (2022), 121--185.

\bibitem[GP19]{GP}
Lucio Guerberoff and Chol Park, \textit{Semistable deformation rings in even Hodge-Tate weights.} Pacific J. Math. \textbf{298} (2019), no. 2, 299--374.

\bibitem[Hid93]{Hid93}
Haruzo Hida, \textit{Elementary Theory of $L$-functions and Eisenstein Series, London Mathematical Society Student Texts} \textbf{26}, Cambridge University Press, 1993.

\bibitem[LP22]{LP22}
Wan Lee and Chol Park, \textit{Semi-stable deformation rings in even Hodge-Tate weights: the residually reducible case.} Int. J. Number Theory \textbf{18} (2022), no. 10, 2171–2209.

\bibitem[Ser62]{Ser62}
Jean-Pierre Serre, \textit{Endomorphismes compl\`etement continus des espaces de Banach $p$-adiques.} 
Publications Math\'ematiques de l'IH\'ES \textbf{12} (1962), 69--85. 


\end{thebibliography}
\end{document}